\RequirePackage{rotating}
\documentclass[12pt]{amsart}
\usepackage[graphicx]{realboxes}
\usepackage{float}
\restylefloat{table}
\usepackage{placeins}

\usepackage{amsmath}
\usepackage{amssymb}
\usepackage{epsfig}
\usepackage{wasysym}
\usepackage{graphicx}
\usepackage{tikz}
\usepackage{tikz-cd}
\usepackage{bm}
\usepackage{xcolor}
\usepackage{listings}
\numberwithin{equation}{section}
\usepackage{extpfeil}
\usepackage{array}
\usepackage{adjustbox}
\usepackage{longtable}
\usepackage{makecell}
\usepackage[all]{xy}
\usepackage{longtable}
\usepackage{rotating}
\input xy
\xyoption{all}
\usepackage{multirow}

\usetikzlibrary{positioning}
\usepackage[colorlinks=false,urlbordercolor=white]{hyperref}
  
\tikzset{sgplattice/.style={inner sep=1pt,norm/.style={red!50!blue},char/.style={blue!50!black},
  lin/.style={black!50}},cnj/.style={black!50,yshift=-2.5pt,left=-1pt of #1,scale=0.5,fill=white}}

\DeclareFontFamily{U}{mathb}{\hyphenchar\font45}
\DeclareFontShape{U}{mathb}{m}{n}{
      <5> <6> <7> <8> <9> <10> gen * mathb
      <10.95> mathb10 <12> <14.4> <17.28> <20.74> <24.88> mathb12
      }{}
\DeclareSymbolFont{mathb}{U}{mathb}{m}{n}
\DeclareMathSymbol{\righttoleftarrow}{3}{mathb}{"FD}

\calclayout
\allowdisplaybreaks[3]

\theoremstyle{plain}
\newtheorem{prop}{Proposition}[section]

\newtheorem{theo}[prop]{Theorem}
\newtheorem{coro}[prop]{Corollary}

\newtheorem{lemm}[prop]{Lemma}

\theoremstyle{definition}

\newtheorem{rema}[prop]{Remark}

\newtheorem{exam}[prop]{Example}

\newtheorem{probn}{Problem}

\newcommand{\actsfromleft}{\mathrel{\reflectbox{$\righttoleftarrow$}}}
\newcommand{\actsfromright}{\righttoleftarrow}

\def\cB{{\mathcal B}}

\def\cN{{\mathcal N}}
\def\cO{{\mathcal O}}

\def\cT{{\mathcal T}}

\def\cX{{\mathcal X}}

\def\fA{{\mathfrak A}}

\def\fD{{\mathfrak D}}

\def\fS{{\mathfrak S}}

\def\fS{{\mathfrak S}}

\def\bA{{\mathbb A}}

\def\bP{{\mathbb P}}

\def\bZ{{\mathbb Z}}

\def\bC{{\mathbb C}}

\def\rH{{\mathrm H}}
\def\rI{{\mathrm I}}

\def\bF{{\mathbb F}}

\def\Pic{\mathrm{Pic}}

\def\Aut{\mathrm{Aut}}

\def\SL{\mathsf{SL}}
\def\GL{\mathsf{GL}}

\def\PGL{\mathsf{PGL}}

\def\Burn{\mathrm{Burn}}

\def\lim{\mathrm{lim}}

\def\IJ{\mathrm{IJ}}
\def\JJ{\mathrm{J}}

\def\Cr{\mathrm{Cr}}

\makeatother
\makeatletter

\begin{document}

\title[Equivariant geometry of cubic threefolds]{Equivariant geometry of singular cubic threefolds}

\author{Ivan Cheltsov}
\address{Department of Mathematics, University of Edinburgh, UK}

\email{I.Cheltsov@ed.ac.uk}

\author{Yuri Tschinkel}
\address{
  Courant Institute,
  251 Mercer Street,
  New York, NY 10012, USA
}

\email{tschinkel@cims.nyu.edu}

\address{Simons Foundation\\
160 Fifth Avenue\\
New York, NY 10010\\
USA}

\author{Zhijia Zhang}

\address{
Courant Institute,
  251 Mercer Street,
  New York, NY 10012, USA
}

\email{zhijia.zhang@cims.nyu.edu}

\date{\today}

\begin{abstract}
We study linearizability of actions of finite groups on singular cubic threefolds, using cohomological tools, intermediate Jacobians, Burnside invariants, and the equivariant Minimal Model Program.  
\end{abstract}

\maketitle

\section{Introduction}

In this paper, we continue our investigations of actions of finite groups on rational threefolds over an algebraically closed field $k$ of characteristic zero, up to equivariant birationality. The main problem is to decide {\em linearizability}, i.e., birationality of the given action to a {\em linear} action on projective space, see, e.g., \cite{CS} for background and references. The linearizability problem is essentially settled in dimension 2 \cite{DI,sari}, but remains largely open in dimension 3. 
Here, we focus on: 

\begin{probn}
Let $X\subset \bP^4$ be a singular rational cubic threefold and let $G$ be a finite 
subgroup of its automorphisms. When is the $G$-action on $X$ linearizable?
\end{probn}

Note that linearizability of a $G$-action for a cubic threefold $X$ is equivalent to projective linearizability, since the action lifts 
 to $\GL_5$ (see Section~\ref{sect:intjac} for a proof,  and \cite{HT-intersect} for a general discussion of these notions). 

Smooth cubic threefolds are not rational, and their 
automorphisms have been classified in \cite[Theorem 1.1]{weiyu}: there are 6 maximal groups
$$
C_3^4\rtimes \fS_5, ((C_3^2\rtimes C_3)\rtimes C_4) \times \fS_3, C_{24}, C_{16}, \mathsf{PSL}_2(\bF_{11}), C_3\times \fS_5.
$$
On the other hand, all singular ones, except cones over smooth cubic curves, are rational. Cubic threefolds with isolated singularities have been classified in \cite{viktorova}; but it is not immediately clear how to identify possible symmetries from that analysis. 

Here, we restrict our attention to nodal cubics, i.e., those with ordinary double points, as this is the most interesting and difficult class of singular cubics, from the perspective of equivariant geometry. In all these cases the automorphism group $\Aut(X)$ is finite, by, e.g., \cite[Theorem 1.1]{All}.   

Note that the existence of a $G$-fixed node yields a straightforward linearization construction: projection from this node gives an equivariant birational map to $\bP^3$, with linear action. Thus, we will be primarily interested in actions not fixing a singular point of $X$. Another such  construction comes from a $G$-stable plane and a disjoint $G$-stable line: 

\begin{lemm}
\label{lemm:line}
Let $X$ be a nodal cubic threefold. Let $G\subseteq \Aut(X)$ be such that it preserves a plane $\Pi\subset X$ and a line $l\subset X$, disjoint from $\Pi$. Then the $G$-action on $X$ is linearizable.
\end{lemm}

\begin{proof}
Let $\phi: X\dashrightarrow X_{2,2}$ be the unprojection from $\Pi$; $X_{2,2}\subset \bP^5$ is a (nodal) complete intersection of two quadrics. Then $\phi$ is $G$-equivariant, and $\phi(l)$ is a $G$-invariant line. Taking a projection from this line, we obtain a $G$-equivariant birational map $X\dashrightarrow \bP^3$; see \cite[Proposition 2.2]{CT-quad} for an application of this construction over nonclosed fields. 
\end{proof}

Let $s=s(X)$ be the cardinality of the set $\mathrm{Sing}(X)$ of nodes of $X$.  It is well known that $s\le 10$. 
Moreover, there is a unique cubic threefold $X$ with $s=10$, the Segre cubic, 
treated in \cite{Avilov,CTZ}. In \cite{Avilov}, it is shown that the subgroup $\fA_5\subset \fS_6=\Aut(X)$ that 
leaves invariant a hyperplane section is not linearizable. 
In \cite{CTZ}, we have completed this analysis by proving
that the action of $G\subseteq\fS_6$ is linearizable if and only if:
\begin{itemize}
\item $G$ fixes a singular point of $X$, or
\item $G$ is contained in the subgroup $\fS_5\subset \fS_6$ that does not leave invariant a hyperplane section of $X$,  or
\item $G\simeq C_2^2$, $X$ contains three $G$-invariant planes, and $\mathrm{Sing}(X)$ splits as a union of five $G$-orbits of length $2$.
\end{itemize}
Using this description one can list all subgroups of $\fS_6$ giving rise to linearizable actions --- there are 37 such subgroups up to conjugation (among 56 conjugacy classes of subgroups of $\fS_6$).

In this paper,  we study the cases where 
$$
2\le s(X)\le 9.
$$
To address the linearizability problem for these outstanding cases, we use explicit geometric constructions, as well as the following techniques: 
\begin{itemize}
\item 
cohomological tools \cite{BogPro,KT-Cremona},
\item 
intermediate Jacobians, in the equivariant context \cite{BW}, 
\item 
Burnside invariants and their specialization \cite{BnG},
\item $G$-birational rigidity and $G$-solidity, see, e.g., \cite{CSar}. 
\end{itemize}

To describe our results, we distinguish cases based on linear position properties of nodes, following \cite{Finkelnbergcubic}.  
According to \cite{Finkelnbergcubic}, 
there are 15 configurations, labeled (J1), ..., (J15),
with (J1), ..., (J5) corresponding to 1, ..., 5 nodes in general linear position, 
and (J15) corresponding to the Segre cubic threefold. 
The relevant invariants are:
\begin{itemize}
\item $s$, the number of nodes of $X$,
\item $d=\mathrm{rk}\, \mathrm{Cl}(X)-1$, the defect of $X$, which equals the number of dependent linear conditions imposed on $\rH^0(X,\cO_X(1))$ by the nodes,  
and
\item $p$ - the number of planes $\Pi\subset X$. 
\end{itemize}
We list all possibilities for the triples $(s,d,p)$, and describe our results in each of the cases:

\begin{itemize}
    \item $s=2, d=0, p=0$: We prove that the $G$-action is linearizable if and only if $G$ fixes each of the two nodes, and classify actions of cyclic groups not fixing any node,
    see Section~\ref{sect:two}. 
    \item $s=3, d=0, p=0$: We conjecture that the $G$-action is linearizable if and only if it fixes a node, and classify all automorphism groups not fixing any node. We provide examples of nonlinearizable actions of $G=C_3^2$, see Section~\ref{three}. 
    \item $s=4$:
    \begin{itemize}
        \item $d=0, p=0$: There is an equivariant birational map to a smooth divisor of degree $(1,1,1,1)$ in $(\bP^1)^4$. Following considerations over nonclosed fields in \cite[Conjecture 1.3]{KuznetsovProkhorov2022}, we conjecture that the $G$-action is not linearizable if it is transitive on the nodes. We provide examples of nonlinearizable actions of $G=C_2^2$ confirming this conjecture. We classify all automorphism groups not fixing any node.
        \item $d=1, p=1$: An action is nonlinearizable if and only if it does not fix a node and $X$ does not  contain $G$-stable lines disjoint from the unique plane in $X$, see Section~\ref{sect:four}. 
    \end{itemize}
    \item $s=5$:
    \begin{itemize}
        \item $d=0,p=0$: All actions are linearizable, except for actions of $\fA_5$ and $\fS_5$, which are birational to standard actions on a smooth quadric threefold. The $\fS_5$-action is not linearizable \cite{VAZ}; we conjecture that the $\fA_5$-action is also not linearizable, see Section~\ref{sect:5}.  
        \item $d=1,p=1$: All actions are linearizable, as there is
        a unique node outside the plane and fixed by the action  \cite{Finkelnbergcubic}.
    \end{itemize}
    \item $s=6$:
    \begin{itemize}
        \item $d=1,p=0$: We classify all automorphism groups not fixing any node, establish a sufficient condition for nonlinearizability, and apply it to provide examples of nonlinearizable actions of $G=C_2$.
        \item $d=1,p=1$: All actions are linearizable by Lemma~\ref{lemm:line}. 
        \item $d=2,p=3$: We classify all actions and solve the linearizability problem for most of them, see Section~\ref{sect:six}.
    \end{itemize}
    \item $s=7$: 
    \begin{itemize}
        \item $d=2, p=2$: All actions are linearizable, since 
        each of the two planes contains 4 nodes, and exactly one of the nodes is on both planes,
        thus preserved by the action \cite{Finkelnbergcubic}.   
        \item $d=2, p=3$: All actions are linearizable, since there is a unique node not contained in any plane in $X$, thus fixed by the action \cite{Finkelnbergcubic}. 
    \end{itemize}
    \item $s=8, d=3, p=5$: We classify automorphism groups not fixing any node, and solve the linearizability problem in Section~\ref{sect:8}. 
    \item $s=9, d=4, p=9$: Linearizability problem is solved, except for specific actions of $\fS_3$ and $\fD_6$, which are birational to actions on a smooth quadric, see
    Section~\ref{sect:nine}.
\end{itemize}

We conclude the introduction by summarizing the cases for which the linearizability problem remains open. 
\begin{itemize}
    \item $s=3, d=0, p=0$: Actions in Proposition~\ref{prop:3nodclassification} not containing the $C_3^2$ in Example~\ref{exam:3nodabelian} and not fixing any node.
    \item $s=4, d=0, p=0$: Actions  in Theorem~\ref{theo:Aut-4-nodes} not containing the $C_2^2$ in Example~\ref{exam:4nodBurn} and not fixing any node.
    \item $s=5, d=0, p=0$: A unique $\fA_5$-action described in Section~\ref{sect:5}, which is equivariantly birational to the $\fA_5$-action on a smooth quadric~\eqref{eq:Quadric5}.
    \item $s=6, d=1, p=0$: Actions  in Proposition~\ref{prop:6nodegeneralprop} not fixing any node and not containing an involution not fixing any nodes.
    \item$s=6, d=2, p=3$: Actions in Proposition~\ref{prop:Aut6nod3pl} not containing the $C_2^2$ in Lemma~\ref{lemm:63Burn}, not containing the $C_2^2$ or $\fS_3$ in Remark~\ref{rema:63Burn}, not contained in the $C_2^2$ in Lemma~\ref{lemm:63linear}, and not fixing any node.
    \item $s=9, d=4, p=9$: The actions of $\fD_6$, $\fS_3$ and $\fS_3'$ specified in \eqref{eq:9nodalremain}; these are also equivariantly birational to actions on a smooth quadric \eqref{eqn:9nodquadric}.
\end{itemize}
In many of these cases, equivariant  specialization of \cite{BnG}, applied here in the geometric context for the first time, shows nonlinearizability of the actions for a very general member of the family, see Propositions~\ref{prop:spec3-9},~\ref{prop:special4to8},~\ref{prop:special4to10}, Lemmas~\ref{lemm:spec63pl-9nod}, ~\ref{lemm:spec63pl-10nod} and Remark~\ref{rema:C2C4spec6nod}.

\

\noindent
{\bf Acknowledgments:} 
The first author was partially supported by the Leverhulme Trust grant RPG-2021-229.
The second author was partially supported by NSF grant 2301983. We are grateful to Anton Mellit for his help with computations in Proposition~\ref{prop:2-nodes} and to Olivier Wittenberg for his comments.

\section{Obstructions to linearizability}
\label{sect:intjac}
Among available obstruction theories to  linearizability are:
\begin{itemize}
\item Existence of fixed points 
upon restriction to abelian subgroups,
    \item Group cohomology,  
    \item Intermediate Jacobians, and their equivariant versions,
    \item Burnside invariants,
    \item Specialization of birational types, 
    \item Birational rigidity. 
\end{itemize}
We briefly review relevant results and constructions.   

\subsection*{Fixed points by abelian subgroups}
Recall that existence of fixed points for actions of abelian groups is an equivariant birational invariant, see \cite{reichsteinyoussinessential}. Precisely, let $G$ be a finite abelian group acting generically freely on a smooth projective variety $V$. Assume that  there exists a $G$-equivariant birational map $W\dashrightarrow V$ from a smooth $G$-variety $W$. Then 
$$
W^G\ne \emptyset \quad\iff\quad V^G\ne\emptyset.
$$
Linear actions of abelian groups on projective spaces always have fixed points, and thus:

\begin{lemm}
\label{lemm:fixedptabelian}
Let $V$ be a smooth projective variety with generically free and linearizable action by a finite group $G$. Then
$V^H\ne\emptyset$ for all abelian subgroups $H\subseteq G$.
\end{lemm}

In particular, let $X$ be a nodal cubic threefold and $G\subseteq\Aut(X)$. The $G$-action on $X$ is linearizable if and only if it is projectively linearizable. 
To see this, one can apply the argument in \cite{lifting}. Alternatively, we provide a direct proof below. 

First, we show that the $G$-action is induced from an action of the ambient $\bP^4$, i.e., $G\subset\PGL_5.$ Indeed, by Lefschetz hyperplane theorem, the Picard group $\Pic(X)=\bZ$ is generated by a general hyperplane section of $X$. The induced $G$-action on $\Pic(X)$ is trivial, sending a hyperplane to another hyperplane in $\bP^4$. This implies that the $G$-action on $X$ lifts to $\bP^4$. 

   It follows that the $G$-action naturally lifts to $\cO_{X}(-5)$, the restriction of the canonical bundle of $\bP^4$ to $X.$ Similarly, the $G$-action lifts to $\cO_X(-2)$, the canonical bundle of $X$. Since $\cO_X(-5)$ and $\cO_X(-2)$ generate $\cO_X(1)$, we know that the $G$-action lifts to $\cO_X(1)$, and thus to $\rH^0(X,\cO_{X}(1))$ and $\rH^0(\bP^4,\cO_{\bP^4}(1))$. Therefore the $G$-action lifts to $\GL_5.$

   This also shows that the Amitsur group $\mathrm{Am}(X,G)$ (see \cite[Section 6]{blancfinite}) is trivial. If the $G$-action is projectively linearizable, i.e., equivariantly birational to a $G$-action on $\bP^3$, then $\mathrm{Am}(\bP^3,G)=0$ since the Amitsur group is an equivariant birational invariant. This implies that the $G$-action on $\bP^3$ is linear, namely, it lifts to $\GL_4$. So the notions of linearizable and projectively linearizable actions on $X$ are equivalent.

 Thus, if an abelian subgroup $H\subseteq G$ does not fix a point in the standard desingularization $\widetilde{X}$ of $X$ (the blowup of the nodes), then there exists no $G$-equivariant birational map $X\dasharrow\mathbb{P}^3$. We found two applications of this obstruction, see Example~\ref{exam:3nodabelian} and Section~\ref{sect:nine}.

\subsection*{Cohomology}
Let $X$ be a nodal cubic threefold, 
$\tilde{X}\to X$ its standard desingularization, 
and $G\subseteq \Aut(X)$. 
Here we consider the induced $G$-actions on the Picard group
$\Pic(\tilde{X})$ and the class group $\mathrm{Cl}(X)$; we often identify divisors and their classes. 

A well-studied obstruction to (stable) linearizability is the failure of $\Pic(\tilde{X})$ to be a {\em stably permutation module}, we call this the {\bf (SP)}-obstruction. 
In turn, this is implied by the 
nonvanishing of 
$$
\rH^1(G',\Pic(\tilde{X})), \quad 
\text{ or } \quad \rH^1(G',\Pic(\tilde{X})^\vee),  \quad 
\text{for some $G'\subseteq G$}.
$$
We call this the {\bf (H1)}-obstruction, see \cite[Section 2]{CTZ}.

\begin{prop}
\label{prop:coho}
When $s(X)\le 7$ and $s(X)\ne 6$, or when $s=6$ and the nodes are not in general linear position,
$\Pic(\tilde{X})$ is a permutation module. 
\end{prop}

\begin{proof}
We use labels for configurations of nodes and planes from \cite{Finkelnbergcubic}. 

\noindent

\begin{itemize} 
\item $s=1, \ldots, 5$, $p=0$; (J1-J5):

\noindent 
 $\mathrm{Cl}(X)=\bZ$ is freely generated by the hyperplane section, with trivial $G$-action, and $\Pic(\tilde X)$ is freely generated by the exceptional divisors of the blowup $\tilde X\to X$ and the pullback to $\tilde X$ of the basis of $\mathrm{Cl}(X)$. The $G$-action permutes the exceptional divisors. So $\Pic(\tilde X)$ is a permutation module.

\item $s=4,5,6$, $p=1$; (J6-J8):

\noindent  
$\mathrm{Cl}(X)=\bZ^2$ is freely generated by the hyperplane section and the unique, necessarily $G$-stable, plane in $X$, with trivial $G$-action on their classes. 
Thus $\Pic(\tilde X)$ is a permutation module.

\item $s=7$, $p=2$; (J10):

\noindent
$\mathrm{Cl}(X)=\bZ^3$ is freely generated by the hyperplane section and the classes of the
two planes in $X$, with $G$ possibly permuting these two classes. Thus $\Pic(\tilde X)$ is a permutation module.

\item $s=6,7$, $p=3$; (J11-J12):

\noindent
$\mathrm{Cl}(X)=\bZ^3$ is freely generated by the classes of the three planes in $X$ which form a tetrahedron (with one face missing), with $G$ possibly permuting these classes. So $\Pic(\tilde X)$ is a permutation module. 
\end{itemize}
\end{proof}

The remaining cases are more involved; we handle these in subsequent sections.

\subsection*{Intermediate Jacobians}

Applications of intermediate Jacobians to rationality problems (over $k=\bC$) go back to the seminal work of Clemens-Griffiths \cite{CG}: if a smooth threefold $X$ is rational then its intermediate Jacobian $\IJ(X)$ is a product of Jacobians of curves. Refinements of this, taking into account group actions, have appeared in, e.g.,  \cite{Beauville2012}; an arithmetic analog of these arguments has been developed in \cite{BW}. In particular, intermediate Jacobians exist over arbitrary fields, see \cite{ACV,BW-tor}. 

The key point is that, geometrically, $\IJ(X)$ could be a product of Jacobians of curves, but this does not necessarily hold equivariantly, respectively, over a nonclosed base field. 
This idea is implemented in, e.g., \cite[Theorem 1.1]{BW}. Pursuing the analogy, we have:

\begin{theo}
\label{theo:IJ}
Let $X$ be a smooth projective rationally connected threefold over an algebraically closed field such that its intermediate Jacobian $\IJ(X)$ is isomorphic to the Jacobian of a smooth nonhyperelliptic curve $C$ of genus $g\ge 3$ as a principally polarized abelian variety.
Suppose that $\mathrm{Aut}(X)$ contains an involution $\tau$ acting 
on $\IJ(X)$ by multiplication by $(-1)$.
Then $X$ is not $\langle \tau\rangle$-equivariantly birational to any smooth projective variety with trivial intermediate Jacobian.
\end{theo}

\begin{proof}
Suppose first that there exists a $\langle \tau\rangle$-equivariant blowup 
$\pi\colon X\to Y$ of a nonhyperelliptic curve $C\subset Y$, 
where $Y$ is a smooth threefold with trivial intermediate Jacobian. 
Since $C$ is not hyperelliptic, it follows from Theorem 3 in \cite[Appendix]{LauterSerre} that
\begin{equation}
    \label{eqn:second}
\mathrm{Aut}\big(\IJ(X)\big)\simeq\mathrm{Aut}(C)\times C_2,
\end{equation}
where the second factor corresponds to the action of multiplication by $(-1)$. 
If $C$ is pointwise fixed by $\tau$, then $\tau$ acts trivially on $\IJ(X)$. If $\tau$ acts faithfully on $C$, then its action on $\IJ(X)=\JJ(C)$ is induced by the action on $C$; thus,
$\tau$ cannot project nontrivially to the second factor in \eqref{eqn:second}. 

The general case is treated similarly, using equivariant weak factorization. 
\end{proof}

\begin{exam}
Consider the conic bundle
$$
x_1x_2 = f(y_1,y_2,y_3)\subset\mathbb{A}^2\times\bP^2,  
$$
where $f$ is a form of degree $\geqslant 4$
defining a smooth curve in~$\bP^2$, and $C_2$-action via permutation on $x_1,x_2$. Then this action is not linearizable by Theorem~\ref{theo:IJ}, see the proof of Theorem~\ref{theo:2-nodal}. 
\end{exam}

\begin{rema}
\label{rema:IJ}
In the assumptions and notation of Theorem~\ref{theo:IJ}, 
suppose that there exists a $G$-equivariant birational map $X\dasharrow\mathbb{P}^3$,
for some subgroup $G\subseteq\mathrm{Aut}(X)$.
From the isomorphism \eqref{eqn:second}, we deduce that 
the $G$-action on $\IJ(X)$ gives rise to a homomorphism 
$$
\nu\colon G\to\Aut(\IJ(X))=\mathrm{Aut}(C)\times C_2, 
$$
The projection of $\nu(G)$ to the $C_2$-factor must be trivial, cf. the proof of \cite[Proposition~3.2]{BW}.
\end{rema}

\subsection*{Burnside obstructions}

It is well-known that the classification of involutions $\tau\in \Cr_2$, the plane Cremona group, is based on the geometry of $\tau$-fixed loci $F(\tau)$, see, e.g.,  \cite{baylebeauville}. The 
different cases are characterized by 
geometric properties of a (necessarily unique) curve $C$ of genus $\ge 1$ in $F(\tau)$, primarily by whether or not this curve is hyperelliptic. A more refined birational invariant of actions of general cyclic groups on rational surfaces, the {\em normalized fixed curve with action}, 
appeared in \cite{dF}
and \cite{Blanc}; the invariant takes into account the stabilizer of the fixed curve, as well as the residual action on it. 

These invariants are special cases of the Burnside formalism of \cite{BnG}, which applies to actions of arbitrary finite groups and 
takes into account {\em all} strata with nontrivial generic stabilizers. 
We will use a simplified version, explained in \cite[Section 4]{CTZ}. 
It is based on the notion of {\em incompressible divisorial symbols}, which should be viewed as analogs of higher-genus curves in the classification of involutions on rational surfaces.
A sample result in our context is the following:

\begin{prop}
\label{prop:burnform}
Let $X$ be a nodal cubic threefold, with a regular action of $G$, and  assume that there is an element  $\tau \in G$ such that 
\begin{itemize}
    \item[(1)] the $\tau$-fixed locus contains a cubic surface $S\subset X$,
    \item[(2)] the subgroup $Y\subseteq G$ preserving $S$ acts generically nontrivially on $S$ and contains an element fixing a curve of genus $\ge 1$. 
\end{itemize}
Then the $G$-action on $X$ is not linearizable. 
\end{prop}

\begin{proof}
Let $H=\langle \tau\rangle$; 
the action produces the symbol
$$
(H, Y/H\actsfromleft k(S), (b)),
$$
By \cite{BogPro}, $\rH^1(Y/H, \Pic(S))\neq 0$, which implies that the symbol is {\em incompressible}, see \cite[Section 4]{CTZ}. Such symbols cannot appear for linear actions, see \cite[Corollary 6.1]{TYZ-3}.
\end{proof}

\begin{exam}
    Let $X\subset \bP^4$ be a 2-nodal cubic given by
    $$
    x_1x_2x_3+x_1(x_4^2+x_5^2)+x_2(x_4^2-x_5^2)+x_3^3=0,
    $$
    with $G\simeq C_4$-action generated by
   $$
\iota : (x_1,x_2,x_3,x_4,x_5)\mapsto(x_2,x_1,x_3,x_4,\zeta_4x_5).
   $$
    The model satisfies the conditions in Proposition~\ref{prop:burnform}. In particular, the subgroup $\langle \iota^2\rangle$ fixes the cubic surface $S=X\cap\{x_5=0\}$. The residual $C_2$-action fixes a genus $1$ curve $S\cap\{x_1=x_2\}$. By Proposition~\ref{prop:burnform}, the $G$-action on $X$ is not linearizable. 
\end{exam}

\subsection*{Specialization of birational types}

We will use the 
specialization homomorphism for Burnside groups
$$
\rho_{\pi}^G: \Burn_{n,K}(G)\to \Burn_{n,k}(G)
$$ 
from \cite[Definition 6.4]{BnG}, and in particular, 
\cite[Corollary 6.8]{BnG}. Here $K$ is the fraction field of a DVR and $k$ its residue field. 
In applications, one considers the local geometry of fibrations, seeking to specialize the birational type of the generic fiber $X$ to a special fiber $X_0$. In practice,  
the special fiber $X_0$ is an irreducible variety, with {\em mild} singularities; the relevant notion of $BG$-rational singularities on the special fiber $X_0$ is in \cite[Definition 6.9]{BnG}.  

\begin{exam}
\label{exam:BG}
Let $\cX\to\cB$ be a $G$-equivariant flat and projective morphism onto a smooth curve $\cB$, with smooth generic fiber $X$ and a special fiber $X_0$ with ordinary double points. Then the singularities of $X_0$ are $BG$-rational singularities in the following situations:
\begin{itemize}
\item $G$-orbits of isolated ordinary double points, with trivial stabilizers \cite[Example 6.10]{BnG};
\item $G=C_2$, fixing a singular point; one verifies directly that the required condition for $BG$-rational singularities holds, namely,
$$
\rho_\pi^G([X\actsfromright G]) = [X_0\actsfromright G].
$$
\end{itemize}
\end{exam}


A Hilbert scheme argument, used in \cite{voisin},  \cite{CTP}, and \cite[Theorem 9]{HKT-conic} in the context of specialization of rationality properties, implies:

\begin{prop} 
\label{prop:flat}
Let $k$ be an uncountable algebraically closed field of characteristic zero and $G$ a finite group. 
Let 
$$
\pi :\cX\to B
$$
be a $G$-equivariant flat and projective morphism onto a smooth curve over $k$ with smooth generic fiber, such that 
\begin{itemize}
\item $G$ acts trivially on $B$ and generically freely on the fibers of $\pi$,
\item 
for some $b_0\in B$,
the special fiber $\cX_{b_0}$ is irreducible, has $BG$-rational singularities, and 
the $G$-action on $\cX_{b_0}$ is not linearizable.  
\end{itemize}
Then, for very general $b\in B$, the $G$-action on the special fiber $\cX_b$ is not 
linearizable. 
\end{prop}

Specialization allows to exhibit nonlinearizable actions which are ``invisible'' to classical obstructions, i.e., cannot be distinguished from linearizable actions with other available tools. 
On the other hand, the {\em very general} condition makes it difficult to determine
linearizability for any specific variety in the family. A central problem is to give {\em criteria} for linearization.

In our applications, we work with models with nodes in the generic fiber. We reduce to the situation of Proposition~\ref{prop:flat} by equivariantly resolving the nodes in the generic fiber.

\begin{exam} 
    \label{exam:2-spec}
    Let $\cX\to\bA_k^1$ be a family of cubic threefolds whose fibers $X_a:=\cX_a$ over $a\in k$ are given in $\bP^4$ by
\begin{multline*}
    a(x_1x_2^2-4x_3^2x_4+x_3x_4^2-3x_3^2x_5-x_4^2x_5)+\\
    +(a+1)(x_1x_2x_3+x_1x_2x_4+x_1x_2x_5-x_1x_4x_5-x_2x_4x_5)+\\
    +x_1x_3x_4-3x_2x_3x_4-3x_1x_3x_5+x_2x_3x_5+(5a+3)x_3x_4x_5=0.
\end{multline*}
One can check that the family carries a $G=\langle\iota\rangle\simeq C_2$-action, with $\iota$ acting on $\bP^4$ via
$$
(x_1,\ldots,x_5)\mapsto (-x_3+x_5, -x_3+x_4,-x_3,-x_3+x_2,-x_3+x_1).
$$
For a very general $a\in k$, $X_a$ is a 2-nodal cubic threefold with nodes at
$$
p_1=[1:0:0:0:0]\quad\text{and }\quad p_2=[0:0:0:0:1].
$$
But the special fiber over $a=0$ is 6-nodal; the nodes are in general linear position and $\iota$ does not fix any of the nodes. By Proposition~\ref{prop:H16nodescroll}, the $G$-action on the special fiber $X_0$ is not stably linearizable. The 4 additional nodes form two $G$-orbits with trivial stabilizer, so they are $BG$-rational singularities, by Example~\ref{exam:BG}. Blowing up the singularities in the generic fiber, we are in the situation of Proposition~\ref{prop:flat}. This allows us to conclude that the $G$-action on a very general member in the family $\cX$ is not stably linearizable.
\end{exam}



\subsection*{Birational rigidity}

Let $X\subset \mathbb{P}^4$ be a nodal cubic threefold and let $G\subseteq \Aut(X)$.
If $\mathrm{rk}\,\mathrm{Cl}^G(X)=1$, then $X$ is a $G$-Mori fiber space \cite[Definition~1.1.5]{CheltsovShramov},
and every $G$-birational map from $X$ to another $G$-Mori fiber space can be decomposed into a sequence of elementary links, known as $G$-Sarkisov links  \cite{Corti1995,HaconMcKernan2013}. If there are no $G$-Sarkisov links that start at $X$, we say that $X$ is $G$-birationally super-rigid.
Similarly, if every $G$-Sarkisov link that starts at $X$ also ends at $X$,  we say that $X$ is $G$-birationally rigid. 
Finally, if $X$ is not $G$-birational to any $G$-Mori fiber space with a positive dimensional base (a conic bundle or a Del Pezzo fibration), we say that $X$ is $G$-solid.
We have the following implications:
\begin{center}
$G$-birationally super-rigid $\Rightarrow$ $G$-birationally rigid $\Rightarrow$ $G$-solid.
\end{center}
Note that all of these conditions assume that $\mathrm{rk}\,\mathrm{Cl}^G(X)=1$.

Recall that the $G$-action on $X$ lifts to $\mathbb{P}^4$.
Using the $G$-action on  $\mathbb{P}^4$, we can state an obstruction for a cubic threefold $X$ to be $G$-solid: 

\begin{lemm}[{\cite[Lemma 2.6]{Avilov}}]
\label{lem:G-solid}
If $G$ leaves invariant a line or a plane in $\mathbb{P}^4$, then $X$ is not $G$-solid.
\end{lemm}

\begin{proof}
Note that $G$ leaves invariant a plane in $\mathbb{P}^4$ if and only if it leaves invariant a line.
Thus, we may assume that there exists a $G$-stable plane in $\mathbb{P}^4$.
Linear projection $\mathbb{P}^4\dasharrow\mathbb{P}^1$ from this plane
induces a rational dominant map $X\dasharrow\mathbb{P}^1$ whose general fiber is a (possibly singular) cubic surface.
Taking a $G$-equivariant resolution of indeterminacies of this map, a $G$-invariant resolution of singularities (if necessarily), and applying the $G$-equivariant Minimal Model program over $
\bP^1$, we obtain a $G$-birational map from $X$ to 
a $G$-Mori fiber space with a positive-dimensional base.
\end{proof}

This yields the following result:

\begin{theo}[Avilov]
\label{theo:Avilov}
Let $X\subset \bP^4$ be a nodal cubic threefold and 
$G\subseteq \mathrm{Aut}(X)$ a finite subgroup such that $\mathrm{rk}\,\mathrm{Cl}^G(X)=1$. If $X$ is $G$-solid, then one of the following holds:
\begin{enumerate}
\item[(1)] $|\mathrm{Sing}(X)|=10$, $X$ is the Segre cubic, $\mathrm{Aut}(X)\simeq\mathfrak{S}_6$, and $G$ contains a subgroup isomorphic to $\fA_5$ that leaves invariant a hyperplane section of $X$,
\item[(2)] $|\mathrm{Sing}(X)|=9$, $X$ is given in $\mathbb{P}^5$ by
$$
x_1x_2x_3-x_4x_5x_6=\sum_{i=1}^6x_i=0,
$$
$\mathrm{Aut}(X)\simeq\mathfrak{S}_3^2\rtimes C_2$, $G$ acts transitively on $\mathrm{Sing}(X)$,
and is isomorphic to  $\mathfrak{S}_3^2\rtimes C_2$, $\mathfrak{S}_3^2$ or $C_3^2\rtimes C_4$,
\item[(3)] $|\mathrm{Sing}(X)|=5$, $X\subset \bP^4$ is given by 
\begin{multline*}
x_1x_2x_3+x_1x_2x_4+x_1x_2x_5+x_1x_3x_4+x_1x_3x_5+\\
+x_1x_4x_5+x_2x_3x_4+x_2x_3x_5+x_2x_4x_5+x_3x_4x_5=0,
\end{multline*}
$\mathrm{Aut}(X)\simeq\mathfrak{S}_5$, and either $G\simeq\mathfrak{S}_5$ or $G\simeq\mathfrak{A}_5$.
\end{enumerate}
\end{theo}

\begin{proof}
Suppose that $X$ is $G$-solid. 
If there exists a $G$-equivariant birational map $X\dasharrow\mathbb{P}^3$,
then $\mathbb{P}^3$ is $G$-solid, which contradicts \cite{CSar}. 
Thus, the $G$-action on $X$ is not linearizable.
It follows from \cite{avilov-forms,Avilov-note,Avilov} and the proofs of the main results in these papers
that either $X$ and $G$ are as in (1), (2), (3), or $X$ is the cubic threefold in (3) and $G\simeq C_4\rtimes C_5$.
Let us show that $X$ is not $G$-solid in the latter case, contradicting the assumption.

Namely, suppose $X$ is the  threefold from (3), and $G\simeq C_4\rtimes C_5$.
By  \cite{Avilov,VAZ}, there exists the~following $\mathfrak{S}_5$-Sarkisov link:
$$
\xymatrix{
&\widetilde{X}\ar@{-->}[rr]^{\beta}\ar@{->}[ld]_{\alpha}&&\widehat{X}\ar@{->}[rd]^{\gamma}&\\
X\ar@{-->}[rrrr]^{\chi} &&  && Q}
$$
where $Q$ is the smooth quadric threefold 
$$
\{x_1x_2+x_2x_3 +\cdots +x_4x_5= 0\}\subset\mathbb{P}^4,
$$
$\mathfrak{S}_5$ acts on $Q$ by permuting the coordinates,
$\chi$ is the birational map induced by the standard Cremona involution of $\mathbb{P}^4$,
$\alpha$ is the standard resolution of singularities,
$\beta$ is a~composition of $10$ Atiyah flops, 
$\gamma$ is a~blowup of an $\mathfrak{S}_5$-orbit of length $5$.
Let $\eta\colon Q\to\mathbb{P}^3$ be the double cover induced by the projection from the $\mathfrak{S}_5$-fixed point in $\mathbb{P}^4$. Then $\eta$ is $\mathfrak{S}_5$-equivariant, 
and $\mathbb{P}^3$ contains two skew lines $L_1$ and $L_2$ such that the curve $L_1+L_2$ is $G$-invariant. 
Let $C_1$ and $C_2$ be the preimages of these lines on $Q$. Then $C_1$ and $C_2$ are disjoint conics,
and the curve $C_1+C_2$ is $G$-invariant.
Blowing up these two conics, we obtain a $G$-equivariant birational map from $Q$ to a conic bundle over $\mathbb{P}^1\times\mathbb{P}^1$, in particular, $X$ is not $G$-solid, which contradicts our assumption. In fact, the $G$-action on $Q$ is linearizable, see Section~\ref{sect:5}.
\end{proof}

Moreover, in Case (1) in Theorem~\ref{theo:Avilov}, $X$ is $G$-birationally super-rigid \cite{avilov-forms}.
Similarly, if follows from \cite{VAZ} that $X$ is $G$-solid in Case (3) when $G\simeq\mathfrak{S}_5$.
In Section~\ref{sect:nine}, we show that $X$ is $G$-birationally super-rigid in Case (2) when $G\simeq\mathfrak{S}_3^2\rtimes C_2$. We believe that $X$ is $G$-solid for the remaining groups $G$ in Cases (2) and (3). 


\section{Two nodes}
\label{sect:two}

\subsection*{Standard form}
We may assume that 
$\mathrm{Sing}(X)$ consists of the points $$
[1:0:0:0:0], \quad  [0:1:0:0:0],
$$
and that $G=\mathrm{Aut}(X)$ swaps these points. 
Then $X$ can be given by: 
\begin{equation}
    \label{eqn:cube-2}
x_1x_2x_3+x_1q_1+x_2q_2+f_3=0,
\end{equation}
for forms $q_1, q_2 \in k[x_4,x_5]$, and $f_3\in k[x_3, x_4, x_5]$, of degree $2$, $2$, $3$, respectively.

\subsection*{Conic bundle}
Introducing new coordinates $y_1=x_1x_3$ and $y_2=x_2x_3$
(of weight two), and 
multiplying \eqref{eqn:cube-2} by $x_3$, we rewrite \eqref{eqn:cube-2} as
$$
y_1y_2+y_1q_1+y_2q_2+x_3f_3=0,
$$
which defines a quartic hypersurface $V_4\subset\mathbb{P}(1,1,1,2,2)$;
the coordinate change defines a $G$-equivariant birational map 
$$
\chi\colon X\dasharrow V_4.
$$
We can $G$-equivariantly simplify the equation of $V_4$ further as 
$$
z_1z_2=q_1q_2-x_3f_3,
$$
where $z_1=y_1+q_2$ and $z_2=y_2+q_1$.
Observe that $V_4$ has $2$ singular points of type $\frac{1}{2}(1,1,1)$ --- these 
are the points 
$$
[0:0:0:1:0], \quad \text{ and } \quad [0:0:0:0:1],
$$
in coordinates $(x_3,x_4,x_5,z_1,z_2)$.
This yields the following $\mathrm{Aut}(X)$-equivariant commutative diagram:
$$
\xymatrix{
&\widehat{X}\ar@{->}[dl]_{\alpha}\ar@{->}[dr]^{\beta}&\\%
X\ar@{-->}[rr]_{\chi} &&V_4}
$$
where $\alpha$ is an extremal divisorial contraction of a surface to the line $\{x_3=x_4=x_5=0\}\subset X$,
and $\beta$ is an extremal divisorial contraction of the strict transform of the non-normal cubic surface $\{x_3=0\}\cap X$. 
The description of the morphism $\alpha$ can be found in the proof of \cite[Proposition~6.1]{CP2010}, see also \cite{Tziolas}.
Note that $\widehat{X}$ has $2$ singular points of type $\frac{1}{2}(1,1,1)$,
which are mapped to the nodes of $X$.

Let $D$ be the quartic curve $\{q_1q_2-x_3f_3=0\}\subset\mathbb{P}^2_{x_3,x_4,x_5}$.
Then $D$ is smooth,
and we have the following $G$-equivariant commutative diagram
\begin{align}
    \label{diagram:conicbundle}
\xymatrix{
&Y\ar@{->}[dl]_{\gamma}\ar@{->}[dr]^{\pi}&\\%
V_4\ar@{-->}[rr]&&\mathbb{P}^2_{x_3,x_4,x_5}}
\end{align}
where $\gamma$ is the blow up of the singular points of $V_4$,
$\pi$ is a conic bundle with discriminant curve $D$, and the dashed arrow is the projection 
$$
(x_3,x_4,x_5,z_1,z_2)\mapsto(x_3,x_4,x_5).
$$
This gives a natural homomorphism
$$
\gamma:\Aut(X)\to \Aut(D).
$$

\subsection*{Automorphisms}
The full classification of automorphisms of 2-nodal cubics can be 
addressed via the conic bundle presentation, combined with the 
(classically known) classification of automorphisms of smooth quartic plane curves, 
see, e.g., \cite[Lemma 6.16 and Table 6]{DI}, \cite{Kur}; and using Torelli for nodal cubics, as in \cite[Section 7]{CGH}.
Starting with equation \eqref{eqn:cube-2} and passing to the conic bundle, we see that  
the $G$-action on $X$ gives rise to 
\begin{itemize}
    \item a {\em linear} representation on $\bP^2$, preserving a line, corresponding to $x_3=0$, and thus a fixed point in $\bP^2$,
    \item an automorphism of the discriminant curve $D$. 
\end{itemize}
Combining these two conditions with the list of automorphisms of plane quartic curves, we find that the possible images of the $G$-actions on the base $\bP^2_{x_3,x_4,x_5}$ of the conic bundle are
\begin{align*}
    &C_2,  \,\,C_3,\,\, C_4,\,\,C_2^2,\,\,\fS_3,\,\,C_6, \,\,C_7, \,\,C_2\times C_4,\,\, C_8,\,\, Q_8, \,\,C_9,\,\, 
    C_4^2,\\
    & C_{12}, \,\,D_4\rtimes C_2,\,\, \SL_2(\bF_3), \,\,  OD16,\,\,  \fD_4, \,\,
C_4wrC_2,\,\, \SL_2(\bF_3)\rtimes C_2.
\end{align*}
Here are examples with interesting groups $\Aut(X)$:

\begin{exam}
    \label{exam:group}
    We keep the notation of \eqref{eqn:cube-2}, with $X\subset \bP^4$ and the discriminant curve $D\subset \bP^2_{x_3,x_4,x_5}$, with $G=\Aut(X)$ and $G'=\Aut(D)$.   

\begin{itemize}
    \item[(1)] Let $D=x_3^4+x_4^4+x_5^4+(4\rho+2)x_3^2x_4^2$, and $X$ be given by 
    $$
     f_3=x_3^3,\quad q_1=x_4^2+(2\rho+1+2i)x_5^2, \quad 
   q_2=x_4^2+(2\rho+1-2i)x_5^2.
    $$
    Then $G=\fD_4\rtimes C_2$ and $G'={\SL_2(\bF_3)\rtimes C_2}$.
  \item[(2)] Let $D=x_3^4+x_4^4+x_5^4$, and $X$ be given by 
    $$
    f_3=-x_3^3,\quad q_1=x_4^2+ix_5^2,\quad  q_2=x_4^2-ix_5^2.
    $$
    Then $G=C_4wr C_2$ and $G'=C_4^2\rtimes \fS_3$. 
\end{itemize}
\end{exam}


\begin{prop}
\label{prop:2-nodes}
Let $X\subset \bP^4$ be a 2-nodal cubic threefold with an action of 
a cyclic group $G=\langle \iota\rangle \subseteq\Aut(X)$ not fixing any node. Then, up to a change of coordinates, $X$ is given by 
$$
x_1x_2x_3+x_1q_1+x_2q_2+f_3=0,
$$
for $q_1,q_2\in k[x_4,x_5]$ and $f_3\in k[x_3,x_4,x_5]$ that can be described together with $\iota$ as follows.
\begin{enumerate}
\item[($C_2$)] $\iota(x_1,x_2,x_3,x_4,x_5)=(x_2,x_1,x_3,x_4,x_5)$, 
\begin{align*} 
\hskip -3.5cm q_1&=q_2=x_4x_5, \\ 
\hskip -3.5cm f_3&\in k[x_3,x_4,x_5];
\end{align*}
\item[($C_2^\prime$)] $\iota(x_1,x_2,x_3,x_4,x_5)=(-x_2,-x_1,x_3,x_4,-x_5)$, 
\begin{align*} 
q_1&=a_1x_4^2+x_4x_5+a_3x_5^2,\\
q_2&=-a_1x_4^2+x_4x_5-a_3x_5^2,\\
f_3&=c_1x_3^3+d_1x_3^2x_4+x_3(e_1x_4^2+e_3x_5^2)+r_1x_4^3+r_3x_4x_5^2,
\end{align*}
for some $a_1,a_3,c_1,d_1,e_1,e_3,r_1,r_3\in k$;

\item[($C_2^{\prime\prime}$)] $\iota(x_1,x_2,x_3,x_4,x_5)=(x_2,x_1,x_3,x_4,-x_5)$, 
\begin{align*} 
q_1&=q_2=x_4^2+x_5^2\\
f_3&=c_1x_3^3+d_1x_3^2x_4+x_3(e_1x_4^2+e_3x_5^2)+r_1x_4^3+r_3x_4x_5^2, 
\end{align*}
for some $c_1,d_1,e_1,e_3,r_1,r_3\in k$;

\item[($C_4$)] $\iota(x_1,x_2,x_3,x_4,x_5)=(x_2,x_1,x_3,\zeta_4x_4,-\zeta_4x_5)$, $\zeta_4=e^{\frac{2\pi i}4},$ 
\begin{align*} 
\hskip -2.5cm
q_1&=a_1x_4^2+x_4x_5+a_3x_5^2,\\
\hskip -2.5cm q_2&=-a_1x_4^2+x_4x_5-a_3x_5^2,\\
\hskip -2.5cm f_3&=x_3^3+e_2x_3x_4x_5, 
\end{align*}
for some $a_1,a_2,e_2\in k$;

\item[($C_4^\prime$)] $\iota(x_1,x_2,x_3,x_4,x_5)=(x_2,x_1,x_3,x_4,\zeta_4x_5)$, 
\begin{align*}
\hskip -3.6cm q_1&=x_4^2-x_5^2, \\
\hskip -3.6cm q_2&=x_4^2+x_5^2, \\
\hskip -3.6cm f_3&\in k[x_3,x_4];
\end{align*}
%

\item[($C_4^{\prime\prime}$)] $\iota(x_1,x_2,x_3,x_4,x_5)=(x_2,x_1,x_3,-x_4,\zeta_4x_5)$,
\begin{align*} 
\hskip -2.5cm q_1&=x_4^2-x_5^2, \\
\hskip -2.5cm q_2&=x_4^2+x_5^2,\\
\hskip -2.5cm f_3&=x_3^3+e_1x_3x_4^2+r_3x_4x_5^2,
\end{align*}
for some $e_1,r_3\in k$;

\item[($C_6$)] $\iota(x_1,x_2,x_3,x_4,x_5)=(x_2,x_1,x_3,\zeta_3x_4,\zeta_3^2x_5)$ for $\zeta_3=e^{\frac{2\pi i}{3}}$, 
\begin{align*}\hskip 1cm
 \hskip -2.9cm   q_1&=q_2=x_4x_5,\\
  \hskip -2.9cm  f_3&=x_3^3+e_2x_3x_4x_5+r_1(x_4^3+x_5^3),
\end{align*}
for some $e_2,r_1\in k$;

\item[($C_6^\prime$)] $\iota(x_1,x_2,x_3,x_4,x_5)=(\zeta_6x_2,\zeta_6x_1,x_3,\zeta_6^5x_4,\zeta_6^2x_5)$ for $\zeta_6=e^{\frac{2\pi i}{6}}$, 
\begin{align*}
  \hskip -2.7cm    q_1&= a_1x_4^2+ x_4x_5+ a_3x_5^2,   \\
   \hskip -2.7cm   q_2&= -a_1x_4^2+ x_4x_5-a_3x_5^2, \\ 
    \hskip -2.7cm  f_3&= x_3^2x_5,
\end{align*}
for some $a_1,a_3\in k$;
\item[($C_{12}$)] $\iota(x_1,x_2,x_3,x_4,x_5)=(\zeta_{12}^8x_2,\zeta_{12}^8x_1,x_3,\zeta_{12}^4x_4,\zeta_{12}x_5)$ for $\zeta_{12}=e^{\frac{2\pi i}{12}}$, 
\begin{align*}
  \hskip -3.8cm     q_1&=(x_4^2-x_5^2), \\
 \hskip -3.8cm     q_2&=(x_4^2+x_5^2), \\
   \hskip -3.8cm   f_3&=x_3^2x_4.
\end{align*}
\end{enumerate}  
\end{prop}

\begin{proof}
We can choose the coordinates so that
$$
\iota: (x_1,x_2,x_3,x_4,x_5)\mapsto(sx_2,tx_1,x_3,ux_4,vx_5),
$$
for some $s,t,u,v\in k^\times$, and $X$ is given by 
$$
x_1x_2x_3+x_1q_1+x_2q_2+f_3=0, 
$$
for
\begin{align*}
q_1&=a_1x_4^2+a_2x_4x_5+a_3x_5^2 + a_4x_3^2+a_5x_3x_4+a_6x_3x_5,
\\
q_2&=b_1x_4^2+b_2x_4x_5+b_3x_5^2 + b_4x_3^2+b_5x_3x_4+b_6x_3x_5,
\\
f_3&=c_1x_3^3+x_3^2d(x_4,x_5)+x_3e(x_4,x_5) + r(x_4,x_5),
\end{align*}
where 
\begin{align*} 
&\hskip -3cm d=d_1x_4+d_2x_5, \\
& \hskip -3cm e=e_1x_4^2+e_2x_4x_5+e_3x_5^2,
\\ 
&\hskip -3cm r=r_1x_4^3+r_2x_4^2x_5+r_3x_4x_5^2+r_4x_5^3.
\end{align*}
Since $X$ is $\langle\iota\rangle$-invariant, one has $\iota^*(f)=stf$, and thus
the zero loci of $q_1q_2$ and $f_3$ are preserved; and these polynomials cannot identically vanish, under our assumptions on singularities of $X$. Concretely, 
$$
\iota^* (f)  = stx_1x_2x_3+sx_2\iota^*(q_2)+tx_1\iota^*(q_2)+\iota^*(f_3)=stf
$$
which implies that 
\begin{equation} 
\label{eqn:q1}
\iota^*(q_2)  = tq_1, \quad 
\iota^* (q_1) = sq_2, 
\end{equation}
and 
\begin{equation}
    \label{eqn:f3}
\iota^*(f_3)=stf_3.
\end{equation}
Expanding
and substituting into \eqref{eqn:q1} 
we obtain 12 equations: 
\begin{align*}
& \hskip -2.8cm u^2a_1-tb_1=uva_2-tb_2=v^2a_3-tb_3=0, 
\\
& \hskip -2.8cm sa_1-u^2b_1=sa_2-uvb_2=sa_3-v^2b_3=0, 
\\
& \hskip -2.8cm a_4-tb_4=ua_5-tb_5=va_6-tb_6=0,  
\\
& \hskip -2.8cm sa_4-b_4=sa_5-ub_5=sa_6-vb_6=0, 
\end{align*}
and, writing down the \eqref{eqn:f3} constraints on $f_3$, additional equations
\begin{align*}
& c_1(1-st)=0,  \\
& d_1(u-st)=d_2(v-st)=0,\\
& e_1(u^2-st)=e_2(uv-st)=e_3(v^2-st)=0, \\
& r_1(u^3-st)=r_2(uv-st)=r_3(uv^2-st)=r_4(v^3-st)=0, 
\end{align*}
in the variables 
$$
a_1,\ldots,a_6,b_1,\ldots,b_6,c_1,d_1,d_2,e_1,e_2,r_1,\ldots,r_4\in k.
$$

Since $X$ has nodes at $[1:0:0:0:0]$ and $[0:1:0:0:0]$, we have 
$$
a_2^2-4a_1a_3\ne 0, \quad b_2^2-4b_1b_3\ne 0.
$$
Thus, up to scaling $x_4,x_5$ and swapping them, we may further assume that one of the following holds: 
\begin{itemize} 
\item $a_2=b_2=1$, or
\item $a_2=b_1=b_3=1$, $b_2=0$, or 
\item $a_1=b_1=b_3=1$, $a_2=b_2=0$.
\end{itemize}
The second option is impossible, since $b_2=0$ forces $a_2=0$. 
Solving the system of equations for the remaining two options using {\tt Magma}, we obtain a complete set of solutions:

\


\ 

\begin{center}

{\small  

\begin{tabular}{|r|l|l|l|l|c|}
\hline
      & $(a_1,a_2,a_3,a_4,a_5,a_6)$ &  $(d_1,d_2)$  & $c_1$   &$s$ & $u$    \\
      & $(b_1,b_2,b_3,b_4,b_5,b_6)$ & $(e_1,e_2,e_3)$ &$(r_1,r_2,r_3,r_4)$        &$t$ & $v$   \\
\hline 
\hline
(1)   & $(a_1,1,a_3,a_4,a_5,a_6)$ &  & &    1  &   1 \\
      & $(a_1,1,a_3,a_4,a_5,a_6)$ &  & &    1  &   1\\
\hline 
\hline
(2)   & $(a_1,1,a_3,a_4,a_5,a_6)$   &   $(d_1,0)$  &  & -1      &  1   \\
      & $(-a_1,1,-a_3,-a_4,-a_5,a_6)$ & $(e_1,0,e_3)$ & $(r_1,0,r_3,0)$ & -1     &  -1\\
\hline
\hline
(3)   & $(0,1,0,a_4,0,0)$   &   $(0,0)$  &                    & 1      &  $\zeta_3$   \\
      & $(0,1,0,a_4,0,0)$ & $(0,e_2,0)$ & $(r_1,0,0,r_4)$ & 1     &  $\zeta_3^2$\\
\hline 
\hline
(4)   & $(a_1,1,a_3,a_4,0,0)$   &   $(0,0)$  &                  & 1   &  $\zeta_4$   \\
      & $(-a_1,1,-a_3,a_4,0,0)$ & $(0,e_2,0)$ & $(0,0,0,0)$ & 1   &  $-\zeta_4$\\
\hline
\hline
(5)   & $(a_1,1,a_3,0,0,0)$   &   $(0,d_2)$  &    0                & $\zeta_6$      &  $\zeta_6^5$   \\
      & $(-a_1,1,-a_3,0,0,0)$ & $(0,0,0)$ & $(0,0,0,0)$ & $\zeta_6$     &  $\zeta_6^2$\\
\hline 
\hline
(6)   & $(1,0,1,a_4,a_5,a_6)$ &  & &    1  &   1 \\
      & $(1,0,1,a_4,a_5,a_6)$ &  & &    1  &   1\\
\hline 
\hline 
(7)   & $(1,0,1,a_4,a_5,a_6)$ & $(d_1,0)$     & &    1  &   1 \\
      & $(1,0,1,a_4,a_5,-a_6)$ & $(e_1,0,e_3)$ & $(r_1,0,r_3,0)$ &    1  & -1  \\
\hline 
\hline
(8)   & $(1,0,-1,a_4,a_5,0)$ & $(d_1,0)$     & &    1  &   1 \\
      & $(1,0,1,a_4,a_5,0)$ & $(e_1,0,0)$ & $(r_1,0,0,0)$ &    1  & $\zeta_4$  \\
\hline 
\hline
(9)   & $(1,0,-1,a_4,a_5,0)$ & $(0,0)$     & &    1  &   -1 \\
      & $(1,0,1,a_4,-a_5,0)$ & $(e_1,0,0)$ & $(0,0,r_3,0)$ &    1  & $\zeta_4$  \\
\hline 
\hline
(10)   & $(1,0,-1,0,0,0)$ & $(d_1,0)$     & 0 &    $\zeta_{12}^8$  &   $\zeta_{12}^4$ \\
      & $(1,0,1,0,0,0)$ & $(0,0,0)$ & $(0,0,0,0)$ &    $
      \zeta_{12}^8$ & $\zeta_{12}$  \\
\hline 
\end{tabular} 

}
\end{center}

\

Here, we omitted solutions obtained by swapping coordinates $x_4$ and $x_5$ and scaling coordinates $x_1$ and $x_2$. After an additional equivariant change of coordinates, we  obtain the required assertion.
\end{proof}

\subsection*{Intermediate Jacobian}
Using arguments as in Section~\ref{sect:intjac}, we settle the linearizability problem for 2-nodal cubic threefolds.

\begin{theo}
\label{theo:2-nodal}
Let $X\subset \bP^4$ be a 2-nodal cubic, and $G\subseteq \Aut(X)$ a subgroup not fixing any node of $X$.
Then the $G$-action on $X$ is not linearizable.
\end{theo}

\begin{proof}
By the assumptions, $G$ contains an element $\iota$ switching the nodes of $X$.
It suffices to prove the required assertion for $G=\langle\iota\rangle$.
With the notation as above, 
we may assume that $X$ is given by \eqref{eqn:cube-2}, i.e., 
$$
f=x_1x_2x_3 + x_1q_2 + x_2q_2 + f_3=0,
$$
and $\iota$ acts on the coordinates via
$$
\iota: (x_1,x_2,x_3,x_4,x_5)\mapsto(sx_2,tx_1,x_3,ux_4,vx_5),
$$
for roots of unity $s,t,u,v$. 
Introducing new coordinates
$$
w_1=\sqrt{t}x_1+\sqrt{s}x_2,\quad 
w_2=\sqrt{t}x_1-\sqrt{s}x_2,
$$
we diagonalize $\iota$, so that it acts via 
\begin{align}\label{eq:iota2nod}
    \iota: (w_1,w_2,x_3,x_4,x_5)\mapsto(\lambda w_1,-\lambda w_2,x_3,ux_4,vx_5),
\end{align}
where $\lambda=\sqrt{st}$. The equation of $X$ in the new coordinates becomes 
$$
f'=(w_1^2-w_2^2)x_3+(w_1+w_2)2\sqrt s q_1+(w_1-w_2)2\sqrt t q_2+4\sqrt{st}f_3=0
$$
Note that $\iota^*(f')=\lambda^2f'$. 

Recall that $X$ is $G$-birational to the conic bundle \eqref{diagram:conicbundle}. The conic bundle is not standard. In particular, the intermediate Jacobian of $\widetilde{X}$ and the Jacobian of the curve $D$ are isomorphic, as principally polarized abelian varieties:
$$
\IJ(\widetilde{X})\simeq \IJ(Y)\simeq \JJ(D),
$$
 where $\widetilde{X}$ is the standard desingularization of $X$, $Y$ is the conic bundle in \eqref{diagram:conicbundle} and $D$ is its discriminant curve in $\bP^2_{x_3,x_4,x_5}$ given by the quartic form $h=q_1q_2-x_3f_3$. Note that $\iota^*(h)=\lambda^2h$.

Arguing as in the proof of \cite[Lemma~1]{Beauville2013},
we see that $\iota$ acts faithfully on
$\IJ(\widetilde{X})$ and preserves the principal polarization. On the other hand, the $\iota$-action on coordinates $x_3,x_4,x_5$ induces an action on $D$ and its Jacobian $\JJ(D)$. We claim that the $\iota$-actions on $\IJ(\widetilde X)$ and $\JJ(D)$ differ by multiplication by $-1$. Since $D$ is not hyperelliptic, this would imply that the $G$-action on $X$ is not linearizable, by Remark~\ref{rema:IJ}.

To compute the action of $\iota$ on $\IJ(\widetilde{X})$, recall that
its tangent space at zero
$\cT_0 \,\IJ(\widetilde{X})$ is isomorphic to $\rH^2(\widetilde{X},\Omega^1_{\widetilde{X}})$.
We show that $\rH^2(\widetilde{X},\Omega^1_{\widetilde{X}})^\vee$ 
is canonically isomorphic to the linear subspace in 
$$
\rH^0\big(\mathbb{P}^4,\Omega^4_{\mathbb{P}^4}\otimes\mathcal{O}_{\mathbb{P}^4}(2X)\big)\cong 
\rH^0\big(\mathbb{P}^4,\mathcal{O}_{\mathbb{P}^4}(K_{\mathbb{P}^4}+2X)\big)\cong \rH^0\big(\mathbb{P}^4,\mathcal{O}_{\mathbb{P}^4}(1)\big)
$$
consisting of all sections that vanish at the nodes of $X$. 
The proof is essentially contained in \cite{Cynk}. We follow the proof of \cite[Lemma 1]{Beauville2013}. Let $\pi: \widetilde\bP^4\to\bP^4$ be the blow up of $\bP^4$ centered at two nodes $p_1$ and $p_2$ of $X$, and identify $\widetilde X$ with the strict transform of $X$ in $\widetilde \bP^4.$ The exact sequence 
$$
0\to (\cN_{\widetilde{X}/\widetilde{\bP^4}})^\vee\to\Omega_{\widetilde{\bP^4}\vert\widetilde{X}}^1\to\Omega^1_{\widetilde{X}}\to 0
$$
gives rise to a $\langle\iota\rangle$-equivariant exact sequence 
$$
0\to \rH^2(\widetilde{X},\Omega_{\widetilde{X}}^1)\to\rH^3(\widetilde{X},(\cN_{\widetilde{X}/\widetilde{\bP^4}})^\vee)\to \rH^3(\widetilde X,\Omega^1_{\widetilde{\bP^4}\vert\widetilde{X}})\to 0.
$$
By \cite{Cynk}, the dimension of $\rH^3(\widetilde X,\Omega^1_{\widetilde{\bP^4}\vert\widetilde{X}})$ equals the defect of $X$, which is $0$ in our case. It follows that 
$$
\rH^2(\widetilde{X},\Omega_{\widetilde{X}}^1)\cong\rH^3(\widetilde{X},(\cN_{\widetilde{X}/\widetilde{\bP^4}})^\vee).
$$
Similarly, the $\langle\iota\rangle$-equivariant exact sequence 
$$
0\to\cO_{\widetilde{\bP^4}}(-2\widetilde X)\to\cO_{\widetilde{\bP^4}}(-\widetilde{X})\to (\cN_{\widetilde{X}/\widetilde{\bP^4}})^\vee\to0
$$ and the vanishing of $\rH^i(\widetilde{\bP^4},\cO_{\widetilde{\bP^4}}(-\widetilde X))$ (\cite[Corollary 2]{Cynk}) provides an $\langle\iota\rangle$-isomorphism
$$
\rH^3(\widetilde{X},(\cN_{\widetilde{X}/\widetilde{\bP^4}})^\vee)\cong\rH^4(\widetilde\bP^4,\cO_{\widetilde{\bP^4}}(-2\widetilde{X})).
$$
By Serre duality, we have canonical isomorphisms between
$$
\rH^4(\widetilde\bP^4,\cO_{\widetilde{\bP^4}}(-2\widetilde{X}))^\vee\cong\rH^0(\widetilde{\bP^4},K_{\widetilde{\bP^4}}\otimes\cO_{\widetilde{\bP^4}}(2\widetilde{X}))\cong\rH^0(\widetilde{\mathbb{P}^4},\Omega^4_{\widetilde{\mathbb{P}^4}}\otimes\mathcal{O}_{\widetilde{\mathbb{P}^4}}(2\widetilde X)).
$$
So we have a $\langle\iota\rangle$-equivariant isomorphism 
$$
\rH^2(\widetilde{X},\Omega^1_{\widetilde{X}})^\vee\cong\rH^0(\widetilde{\bP^4},K_{\widetilde{\bP^4}}\otimes\cO_{\widetilde{\bP^4}}(2\widetilde{X}))
$$
Let $E_1$ and $E_2$ be the exceptional divisors of $\pi$ over $p_1$ and $p_2$ respectively. By adjunction,
$$
K_{\widetilde{\bP^4}}=\pi^*(\cO_{\bP^4}(-5))\otimes\cO_{\widetilde{\bP^4}}(3E_1+3E_2),
$$
and 
$$
\cO_{\widetilde{\bP^4}}(2\widetilde{X})=\pi^*(\cO_{\bP^4}(6))\otimes\cO_{\widetilde{\bP^4}}(-4E_1-4E_2).
$$
Then we know
$$
K_{\widetilde{\bP^4}}\otimes\cO_{\widetilde{\bP^4}}(2\widetilde{X})=\pi^*(\cO_{\bP^4}(1))\otimes\cO_{\widetilde{\bP^4}}(-E_1-E_2).
$$
It follows that we can canonically identify $\rH^2(\widetilde{X},\Omega^1_{\widetilde{X}})^\vee$ with linear subspace in $\rH^0(\mathbb{P}^4,\Omega^4_{\mathbb{P}^4}\otimes\mathcal{O}_{\mathbb{P}^4}(2X))$, (or equivalently, in $\rH^0(\bP^4,\cO_{\bP^4}(1))$), which consists of all sections that vanish at $p_1$ and $p_2.$
Now we can compute the induced $G$-action on $\cT_0\IJ(\widetilde X)^\vee$ explicitly. Set
$$
z_2=\frac{w_2}{w_1}, \, z_3=\frac{x_3}{w_1}, \, z_4=\frac{x_4}{w_1}, \, z_5=\frac{x_5}{w_1},
$$
and consider the rational $4$-forms 
$$
z_3\omega, \, z_4\omega,\, z_5\omega, \qquad \text{where}\quad\omega=\frac{dz_2\wedge dz_3\wedge dz_4\wedge dz_5}{\left(f'(1,z_2,z_3,z_4,z_5)\right)^2}.
$$
These give sections of 
$\rH^0(\mathbb{P}^4,\Omega^4_{\mathbb{P}^4}\otimes\mathcal{O}_{\mathbb{P}^4}(2X))$,
forming a basis of the subspace consisting of sections that vanish at the nodes of $X$. One computes 
$$
\iota^*(z_2)=-z_2,\quad\iota^*(z_3)=\frac1\lambda z_3,\quad\iota^*(z_4)=\frac{u}{\lambda}z_4,\quad\iota^*(z_5)=\frac{v}{\lambda}z_5
$$
and 
$$
\iota^*(f'(1,z_2,\ldots,z_5)^2)=\iota^*\left(\frac{f'(w_1,w_2,x_3,x_4,x_5)^2}{w_1^6}\right)=\frac{f'(1,z_2,\ldots,z_5)^2}{\lambda^2}.
$$
Using these, we see that $\iota$ acts on $\cT_0 \,\IJ(\widetilde{X})^\vee$ 
with eigenvalues 
$$
-\frac{uv}{\lambda^2},\quad -\frac{u^2v}{\lambda^2},\quad-\frac{uv^2}{\lambda^2}.
$$
Similarly, to compute the action of $\iota$ on $\JJ(D)$, we note that $\cT_0 \,\JJ(D)^\vee$ is canonically isomorphic to   
$$
\rH^0\big(\mathbb{P}^2,\Omega^2_{\mathbb{P}^2}\otimes\mathcal{O}_{\mathbb{P}^2}(D)\big). 
$$
 Set $y_4=\frac{x_4}{x_3}$ and $y_5=\frac{x_5}{x_3}$.
The rational $2$-forms 
$$
\frac{dy_4\wedge dy_5}{h(1,y_4,y_5)},\quad
y_4\frac{dy_4\wedge dy_5}{h(1,y_4,y_5)},\quad 
y_5\frac{dy_4\wedge dy_5}{h(1,y_4,y_5)}
$$
define sections of 
$\rH^0(\mathbb{P}^2,\Omega^2_{\mathbb{P}^2}\otimes\mathcal{O}_{\mathbb{P}^2}(D))$,
forming its basis. One has 
$$
\iota^*(y_4)=uy_4,\quad \iota^*(y_5)=vy_5,\quad \iota^*(h(1,y_4,y_5))=\lambda^2h(1,y_4,y_5)
$$
and $\iota$ acts on $\cT_0 \,\JJ(D)^\vee$ 
with eigenvalues 
$$
\frac{uv}{\lambda^2},\quad \frac{u^2v}{\lambda^2},\quad \frac{uv^2}{\lambda^2}.
$$
This shows that the $\iota$-action on $\IJ(\widetilde{X})\simeq\JJ(D)$
differs from the action on $\JJ(D)$ induced by the action on $D$ by 
multiplication by $-1$ as claimed.    
Therefore, the $G$-action on $X$ is not linearizable.
\end{proof}





\section{Three nodes}
\label{three}

\subsection*{Standard form}
The three nodes are necessarily in general linear position; 
they span a distinguished $G$-stable plane, which is {\em not} contained in $X$. 
This case is labelled  (J3), in \cite{Finkelnbergcubic}. Assume the nodes are 
$$
p_1=[1:0:0:0:0],\quad p_2=[0:1:0:0:0],\quad p_3=[0:0:1:0:0].
$$
The standard form is given by
\begin{equation}
\label{eqn:form}
x_1x_2x_3+x_1q_1 + x_2q_2+x_3q_3+ f_3=0,
\end{equation}
with quadratic $q_1,q_2,q_3\in k[x_4,x_5]$, and cubic $f_3\in k[x_4,x_5]$.
Note that $q_1,q_2,q_3$ must have rank $2$, and
$q_1,q_2,q_3,f_3$ do not share common factors.

\subsection*{Automorphisms} 
We proceed to classify automorphism groups of 3-nodal cubics acting transitively on nodes.

\begin{prop}\label{prop:3nodclassification}
Let $X\subset \bP^4$ be a 3-nodal cubic threefold. 
Assume that $\Aut(X)$ contains an element acting transitively on the nodes. 
Then, up to a change of coordinates, $X$ is given by 
\begin{equation*}
x_1x_2x_3+x_1q_1 + x_2q_2+x_3q_3+ f_3=0,
\end{equation*}
for $q_1,q_2,q_3,f_3\in k[x_4,x_5]$ that can be described together with $\Aut(X)$ as follows.  \begin{enumerate}
    \item  $\Aut(X)=C_3$, generated by 
    $$
     \sigma_1:(x_1,x_2,x_3,x_4,x_5)\mapsto (x_2,x_3,x_1, x_4, \zeta^2_6 x_5), \quad \zeta_6=e^{\frac{2\pi i}{6}},
    $$

        \item[$\bullet$]  with $f_3=ax_4^3+bx_5^3,
            $ for $b\ne 0$, $(a,b)\ne(0,1)$ and
        \begin{align*}
  \hskip -3cm          q_1&=x_4(x_4+x_5),\\ 
   \hskip -3cm         q_2&=x_4(x_4+\zeta^2_6 x_5),\\
   \hskip -3cm         q_3&=x_4(x_4+\zeta^4_6x_5), \quad\text{or}
         \end{align*}
            
        \item[$\bullet$]  with  $f_3=dx_4^3+ex_5^3$,
        \begin{align*}
\hskip -0.6cm q_1&=x_4^2+bx_4x_5+x_5^2,
\\
\hskip -0.6cm  q_2&=x_4^2+\zeta^2_6 bx_4x_5 +\zeta^4_6 x_5^2,\\ \hskip -0.6cm q_3&=x_4^2+\zeta^4_6bx_4x_5+\zeta^2_6x_5^2, \quad  d\ne \pm e, \,\,\text{and } \,\,b\ne\pm2.
\end{align*}

    \item $\Aut(X)=C_6$ generated by 
    $$
     \sigma_2:(x_1,x_2,x_3,x_4,x_5)\mapsto (x_2,x_3,x_1, x_4, \zeta_6 x_5),
    $$
    $f_3=dx_4^3$ for some $d\ne 0$, 
    and \begin{align*}
 \hskip -3.7cm   q_1&=x_4^2+x_5^2,\\
   \hskip -3.7cm   q_2&=x_4^2+\zeta^2_6 x_5^2,\\
   \hskip -3.7cm   q_3&=x_4^2+\zeta^4_6x_5^2.\end{align*}

    \item $\Aut(X)\simeq\fS_3$ generated by  $\sigma_1$ and 
    $$
     \sigma_3: (x_1,x_2,x_3,x_4,x_5)\mapsto (\zeta^2_6x_2,\zeta^4_6x_1,x_3, \zeta^4_6x_5, \zeta^2_6 x_4),
    $$
    \begin{align*} 
\hskip -1.4cm     f_3&=d(x_4^3+x_5^3),\\
\hskip -1.4cm   q_1&=x_4^2+bx_4x_5+x_5^2,\\ 
\hskip -1.4cm   q_2&=x_4^2+\zeta^2_6 bx_4x_5 +\zeta^4_6 x_5^2, \\
\hskip -1.4cm   q_3&=x_4^2+\zeta^4_6bx_4x_5+\zeta^2_6x_5^2,\quad d\ne 0, b\ne\pm2.\end{align*}

    \item $\Aut(X)\simeq C_2\times\fS_3$ generated by  $\sigma_1, \sigma_3$ and 
    $$
     \iota: (x_1,x_2,x_3,x_4,x_5)\mapsto (x_1,x_2,x_3,-x_4,-x_5),
    $$
    with $f_3=0$ for $b\ne\pm2$ and $b^2\ne -2$, and 
    \begin{align*} 
\hskip -2.8cm q_1&=x_4^2+bx_4x_5+x_5^2,\\
\hskip -2.8cm q_2&=x_4^2+\zeta^2_6 bx_4x_5 +\zeta^4_6 x_5^2, \\
\hskip -2.8cm q_3&=x_4^2+\zeta^4_6bx_4x_5+\zeta^2_6x_5^2.\end{align*}
   
    \item $\Aut(X)\simeq C_2\times\fS_3$ generated by 
    $$
     \sigma_4: (x_1,x_2,x_3,x_4,x_5)\mapsto (x_2,x_3,x_1,x_4,x_5),
    $$
    $$
     \sigma_5: (x_1,x_2,x_3,x_4,x_5)\mapsto (x_2,x_1,x_3,x_4,x_5),
    $$
    $$
     \sigma_6: (x_1,x_2,x_3,x_4,x_5)\mapsto (x_1,x_2,x_3,x_5,x_4),
    $$
    with 
    $$
    f_3=\begin{cases} 
    d(x_4+x_5)^3 & d\ne 0, \text{or} \\
    d(x_4+x_5)(x_4-x_5)^2& d\ne 0, \text{or} \\
    (x_4+x_5)(ax_4+bx_5)(bx_4+ax_5) & \text{for } a,b\ne 0, \end{cases}
    $$
    $$
    q_1=q_2=q_3=x_4x_5.
    $$
    
    \item $\Aut(X)\simeq C_3\rtimes\fD_4$ generated by $\sigma_2$ and 
    $$
     \sigma_7: (x_1,x_2,x_3,x_4,x_5)\mapsto (\zeta^4_6x_2,\zeta^2_6x_1,x_3,\zeta^2_6 x_5,\zeta_6x_4),
    $$
    with $f_3=0$ and 
    \begin{align*} 
\hskip -3.7cm    q_1&=x_4^2+x_5^2,\\
 \hskip -3.7cm    q_2&=x_4^2+\zeta_6^2 x_5^2,\\
 \hskip -3.7cm    q_3&=x_4^2+\zeta^4_6x_5^2.\end{align*}
 
    \item $\Aut(X)\simeq \fS_3\times\fS_3$ generated by $\sigma_4,\sigma_5,\sigma_6$ and 
    $$
     \sigma_8: (x_1,x_2,x_3,x_4,x_5)\mapsto (x_1,x_2,x_3,\zeta^2_6 x_4,\zeta^4_6x_5),
    $$
    with $f_3=d(x_4^3+x_5^3)$ for some $d\ne 0$ and 
    $$q_1=q_2=q_3=x_4x_5.$$ 

    \item $\Aut(X)\simeq \GL_2(\bF_3)$ generated by  $\iota, \sigma_1, \sigma_3$ and 
    $$
     \sigma_9: (x_1,x_2,x_3,x_4,x_5)\mapsto(x_2,\zeta^5_6x_1,\zeta_6 x_3, \frac{\zeta_6 bx_4+x_5}{1-\zeta^2_6},\frac{\zeta_6 x_4+bx_5}{1-\zeta^2_6}),
    $$
    with $f_3=0$, $b^2= -2$, and 
\begin{align*} 
 \hskip -2.8cm  q_1&=x_4^2+bx_4x_5+x_5^2,\\ 
  \hskip -2.8cm  q_2&=x_4^2+\zeta^2_6 bx_4x_5 +\zeta^4_6 x_5^2,\\  \hskip -2.8cm  q_3&=x_4^2+\zeta^4_6bx_4x_5+\zeta^2_6x_5^2.\end{align*}

\item  $\Aut(X)=\fS_3=\langle\sigma_4,\sigma_5\rangle$, $q_1=q_2=q_3=x_4x_5$ and $f_3$ such that $X$ is not isomorphic to any cubic in cases $\mathrm{(5)}$ and $\mathrm{(7)}$.
 \end{enumerate}
\end{prop}

\begin{proof}
 
Let $X\subset\bP^4$ be a 3-nodal cubic threefold given by~\eqref{eqn:form}, with 
$\Aut(X)$ not fixing any node.  
There exists an exact sequence
\begin{align}\label{eq:3nodaexactseq}
0\to H\to\Aut(X)\stackrel{\rho}{\to}\fS_3\to 0
\end{align}
and a $\sigma_{123}\in \Aut(X)$  acting transitively on the nodes, so that $\rho(\sigma_{123})=(1,2,3)$.
The zeroes of $q_1,q_2,q_3$ define at most 6 points on  $\bP^1_{x_4,x_5}$, 
thus 
\begin{equation}
    \label{eqn:diag-form}
\sigma_{123}=\begin{pmatrix}
0&0&s_3&0&0\\
s_1&0&0&0&0\\
0&s_2&0&0&0\\
0&0&0&1&0\\
0&0&0&0&\zeta_6^r
\end{pmatrix}
\end{equation}
for some $s_1,s_2,s_3\in k^\times$, where $\zeta_6=e^{\frac{2\pi i}{6}}$.
We have  the following cases:
\begin{enumerate}    
    \item[(a)] $\mathrm{gcd}(q_1,q_2,q_3)=1$. We may assume that     \end{enumerate} 
$$
q_1=x_4^2+bx_4x_5+x_5^2, \qquad 
b\in k,\quad b\ne\pm 2.
$$ 
The cyclic action on $x_1,x_2$, and $x_3$ implies that $q_2$ and $q_3$ are multiples of $\sigma_{123}^*(q_1)$ and $\sigma_{123}^*(q_2)$, respectively, and $\sigma_{123}^*(f_3)=s_1s_2s_3f_3.$ The torus action on the coordinates $x_1,x_2$ and $x_3$
allows us to assume that  
\begin{align*}
q_2&=\sigma_{123}^*(q_1),\\
q_3&=\sigma_{123}^*(q_2).
\end{align*}
Since $q_1$, $q_2$, $q_3$ are coprime,
we have $r\ne 0, 3$. 
Thus, $r=1$ or $r=2$. 
    
\begin{itemize}
\item If $r=1$, then $b=0$, the entries in \eqref{eqn:diag-form} are $s_1=s_2=s_3=\pm 1$, and
\begin{align*}
    q_1&=x_4^2+x_5^2, \\ q_2&=x_4^2+\zeta^2_6x_5^2, \\ q_3&=x_4^2+\zeta^4_6x_5^2,
\end{align*}
There are subcases: 
\begin{itemize}
    \item $f_3(x_4,x_5)\not\equiv 0.$ Then $\sigma_{123}$ fixes the points defined by $f_3$ in $\bP^1$. And up to isomorphism, $f_3=dx_4^3$ or $dx_4x_5^2$, for some $d\ne 0$. Since $\sigma_{123}^*(f_3)=s_1s_2s_3f_3$, the latter is impossible. So $f_3=dx_4^3$ and $s_1=1$. This gives $\rho(\Aut(X))=C_3$. On the other hand, any $\gamma\in H$ takes the form 
$$
(x_1,x_2,x_3,x_4,x_5)\mapsto(t_1x_1,t_2x_2,t_3x_3,t_4x_4+t_5x_5,t_6x_4+t_7x_5),
$$
for some $t_j\in k^\times$.
Since $\gamma$ leaves \eqref{eqn:form} invariant, one finds 
$H=C_2=\langle\sigma_{2}^3\rangle$ and $\Aut(X)\simeq C_6=\langle\sigma_{2}\rangle$,
where
    $$
   \sigma_2\colon (x_1,x_2,x_3,x_4,x_5)\mapsto (x_2,x_3,x_1,x_4,\zeta_6 x_5).
    $$
\item $f_3(x_4,x_5)\equiv 0$. Then 
    $$
    H\simeq C_2^2=\langle\iota,\sigma_{2}^3\rangle,
    \quad \mathrm{Aut}(X)\simeq C_3\rtimes\mathfrak{D}_4 = \langle \sigma_{2}, \sigma_{7}\rangle, 
    $$
    where
\begin{align} 
\label{eqn:inv-iota}
\iota\colon (x_1,x_2,x_3,x_4,x_5) & \mapsto (x_1,x_2,x_3,-x_4,-x_5), \\ 
\sigma_{7} \colon (x_1,x_2,x_3,x_4,x_5)&\mapsto 
(\zeta^4_6x_2,\zeta^2_6x_1,x_3,\zeta^2_6x_5,\zeta_6 x_4).
\end{align}
\end{itemize}

\item If $r=2$, then 
\begin{align*} 
q_1&=x_4^2+bx_4x_5+x_5^2, \\ q_2&=x_4^2+\zeta^2_6 bx_4x_5 +\zeta^4_6 x_5^2, \\ q_3&=x_4^2+\zeta^4_6bx_4x_5+\zeta^2_6x_5^2,
\end{align*}
$b\ne 0$ and $s_1=s_2=s_3=\pm 1$. When $b\ne 1$, $q_1,q_2,q_3$ define $6$ points in $\bP^1$, but when $b=1$, they define 3 points. There are subcases:
\begin{itemize}
    \item $f_3(x_4,x_5)=el_1l_2l_3$, $e\in k^\times$.  Then $\sigma_{123}$ permutes the points defined by $f_3$ in $\bP^1$,
    i.e., 
    $$l_1=x_4+dx_5, \quad l_2=x_4+\zeta^2_6dx_5, \quad  l_3=x_4+\zeta^4_6 dx_5, \quad d\in k^\times. 
    $$ 
    In this case, $\sigma_{123}$ takes the form
    $$
\sigma_{1}\colon (x_1,x_2,x_3,x_4,x_5)\mapsto (x_2,x_3,x_1,x_4,\zeta^2_6 x_5).
$$
One finds that $H=0$ and 
$$
\mathrm{Aut}(X)\simeq \begin{cases} 
C_3=\langle\sigma_{1}\rangle & d^6\ne 1,\\
\fS_3=\langle\sigma_{1},\sigma_{3}\rangle, & d^3=1,\\
\fS_3=\langle\sigma_{1},\sigma_{3}'\rangle, & d^3=-1.
\end{cases}
$$
$$
\sigma_{3}\colon (x_1,x_2,x_3,x_4,x_5)\mapsto (\zeta^2_6x_2,\zeta^4_6x_1,x_3,\zeta^4_6x_5,\zeta^2_6 x_4),
$$ 
$$
\sigma_{3}'\colon (x_1,x_2,x_3,x_4,x_5)\mapsto (\zeta^2_6x_2,\zeta^4_6x_1,x_3,\zeta_6x_5,\zeta^5_6 x_4).
$$
    \item $f_3=l_1l_2^2$. Then $\sigma_{123}$ fixes two points defined by $f_3$ in $\bP^1$, and $f_3=dx_4^2x_5$ or $dx_4x_5^2$, for some $d\in k^\times$. But then ~\eqref{eqn:form} cannot be $\sigma_{123}$-invariant. So this case does not exist.
    \item $f_3=l^3$. Then $f_3=dx_4^3$ or $dx_5^3$, for some $d\ne 0$. One finds 
    $$H\simeq 0, \quad \Aut(X)\simeq C_3=\langle\sigma_{1}\rangle.
    $$
    \item $f_3\equiv 0$. Then $\Aut(X)$ contains the involution $\iota$ from \eqref{eqn:inv-iota}. Up to a twist by $\iota$, we may assume $\sigma_{123}=\sigma_1$.
Note that $\mathrm{Aut}(X)$ also contains $\sigma_3.$
Using the same argument to find $H$ as above, one gets that when $b^2\ne -2$, 
$$
H\simeq C_2=\langle\iota\rangle, \quad \mathrm{Aut}(X)
\simeq C_2\times\mathfrak{S}_3
=\langle\iota,\sigma_{1},\sigma_{3}\rangle;
$$
when $b^2=-2$, 
$$
H\simeq Q_8, \quad 
\mathrm{Aut}(X)\cong\mathrm{GL}_2(\mathbb{F}_3)=\langle 
\iota,\sigma_{1},\sigma_{3}, \sigma_{9}\rangle, 
$$ 
$$
\sigma_{9}: (x_1,x_2,x_3,x_4,x_5)\mapsto(x_2,\zeta^5_6x_1,\zeta_6 x_3, \frac{\zeta_6 bx_4+x_5}{1-\zeta^2_6},\frac{\zeta_6 x_4+bx_5}{1-\zeta^2_6}).
$$
\end{itemize}
\end{itemize}
\begin{enumerate}
    \item[(b)]
When $q_1=l_1l, q_2=l_2l$ and $q_3=l_3l$ and $l_1, l_2, l_3$ are coprime. Then $\Aut(X)$ fixes the point in $\bP^1_{x_4,x_5}$ defined by $l$, 
and acts as $C_3$ on the three points defined by $l_1, l_2$ and $l_3$. 
This implies that $r=2$ in \eqref{eqn:diag-form},
and that 
$$
l=x_4, \quad l_1=x_4+x_5,\quad  l_2=x_4+\zeta_3 x_5, \quad  l_3=x_4+\zeta^{2}_3x_5.
$$ 
Then either $f_3=l'^3$ defines one point and $\sigma_{123}$ fixes the point, or $f_3=l_1'l_2'l_3'$, defining three distinct points, with $\sigma_{123}$ permuting them, i.e., 
$f_3=ax_4^3+bx_5^3$ for some $a,b\in k$. Since $X$ is 3-nodal, one has $b\ne 0$ and $(a,b)\ne(0,1)$. From the form of $f_3$, one sees $\sigma_{123}=\sigma_1$. And $H=0$ since any element in $H$ fixes 4 points on $\bP^1$, defined by $l,l_1,l_2,l_3$, and acts trivially on $x_4, x_5$. Moreover, one can show that no action on $\bP^1$ fixes two points defined by $l$ and $l_1$ and swaps those defined by $l_2$ and $l_3$ at the same time. Therefore, $\rho(\Aut(X))=C_3$ and 
$$
H\simeq 0, \quad \Aut(X)\simeq C_3=\langle\sigma_{1}\rangle.
$$
\end{enumerate}
\begin{enumerate}
    \item[(c)] $q_1=q_2=q_3=q$:
We may assume that $q=x_4x_5$. In this case, the exact sequence~\eqref{eq:3nodaexactseq} splits and
$$
\quad \Aut(X)\simeq H\times\fS_3,
$$
with the factor $\fS_3$ acting via permutations of $x_1, x_2$ and $x_3$ and trivially on $x_4, x_5$. 
Moreover, it is easy to see that $H$ must act faithfully on $\bP^1_{x_4,x_5}$. 
Since $H$ preserves the pair of points defined by $q$ in $\bP^1_{x_4,x_5}$, 
it is either cyclic or dihedral. 
Assume that $H\ne 1$. Then the structure of $\Aut(X)$ depends on $f_3$ as follows:
\end{enumerate} 
\begin{itemize}
    \item $f_3=dl^3$, for some $d\ne0$ and linear form $l$ in $x_4$ and $x_5$. Then $H$ has a fixed point, i.e., $H$ is a cyclic group. Moreover, $H$ swaps two points and thus $H\simeq C_2$ with $l=x_4+x_5$ and $H$ is generated by swapping coordinates $x_4$ and $x_5.$
    \item $f_3=dl_1^2l_2$, for some $d\ne 0$ and linear forms $l_1$ and $l_2$ defining two distinct points in $\bP^1$. Then $H$ fixes two points defined by $l_1$ and $l_2$, and swaps two points defined by $q$. Similarly, we have $H\simeq C_2$ with 
    $$l_1=x_4-x_5,\quad l_2=x_4+x_5,
    $$
    where $H$ is generated by swapping $x_4$ and $x_5.$
    \item $f_3=dl_1l_2l_3$, defining three distinct points. There are subcases:
    \begin{itemize}
        \item $H$ permutes three points defined by $f_3$ and swaps two points defined by $q$. Then $H\simeq \fS_3$, 
        generated by 
        \begin{align*} 
        \sigma_8:(x_1,x_2,x_3,x_4,x_5)&\mapsto (x_1,x_2,x_3,\zeta_3 x_4, \zeta^2_3 x_5)\\
        \sigma_6:(x_1,x_2,x_3,x_4,x_5)&\mapsto (x_1,x_2,x_3,x_5, x_4),       
        \end{align*}
        and 
        \begin{equation} \label{eqn:prev}
        l_1=x_4+x_5, \quad l_2=\zeta_3 x_4+\zeta_3^2x_5, \quad l_3=\zeta^2_3x_4+\zeta_3 x_5.
        \end{equation}
         \item $H$ permutes three points defined by $f_3$ and fixes two points defined by $q$, thus $H\simeq C_3$, and  $l_1,l_2,l_3$ are
     as in \eqref{eqn:prev}. However, 
         we know that $X$ admits an additional symmetry swapping two points in $q$ as in the case above.
        \item $H$ fixes the point defined by $l_1$. Then $H$ swaps two points defined by $q$ and two points defined by $l_2$ and $l_3$ because otherwise $H$ is trivial. In this case $H\simeq C_2$, with 
        $$l_1=x_4+x_5, \quad l_2=ax_4+bx_5, \quad l_3=bx_5+ax_4,
        $$ for some $a,b\ne 0$, and $\left(\frac ab\right)^3\ne1$. Similarly, $H$ is generated by the involution swapping $x_4$ and $x_5$.
    \end{itemize}
\end{itemize}
   
\end{proof}

\subsection*{Del Pezzo fibration}
We have an $\mathrm{Aut}(X)$-equivariant commutative diagram:
$$
\xymatrix{
\widetilde{X}\ar@{->}[d]_{\pi}\ar@{-->}[rr]^{\varrho}&&Y\ar@{->}[d]^{\phi}\\%
X\ar@{-->}[rr]_{\rho} &&\mathbb{P}^1}
$$
where $\pi$ is a blow up of the nodes of $X$, 
$\varrho$ is a composition of flops in the strict transforms of the lines 
$$
\{x_1=x_4=x_5=0\}, \quad \{x_2=x_4=x_5=0\},
\quad \{x_3=x_4=x_5=0\},
$$
$\phi$ is a fibration into Del Pezzo surfaces of degree~$6$,
and $\rho$ is the projection given by $$
(x_1,x_2,x_3,x_4,x_5)\mapsto (x_4,x_5).
$$
The anticanonical model of $\widetilde{X}$ is a singular Fano threefold of degree $18$
that has $3$ nodes, which can be smoothed to a smooth Fano threefold of the same degree with Picard rank $1$.


\subsection*{Fixed point obstruction} Among actions in Proposition~\ref{prop:3nodclassification}, we find one example where the linearizability is obstructed by the absence of fixed points upon restriction to abelian subgroups.
\begin{exam}\label{exam:3nodabelian}
    Consider the 3-nodal $X$ in Case (7), Proposition~\ref{prop:3nodclassification}, and the $G=C_3^2=\langle\sigma_4,\sigma_8\rangle$ action on it. The $G$-action does not have a fixed point on $X$ and $\widetilde X^G=\emptyset$. By Lemma~\ref{lemm:fixedptabelian}, the $G$-action on $X$ is not linearizable.  
\end{exam}
\subsection*{Specialization}
Here we exhibit specialization to the 9-nodal cubic with $C_3$-action  giving an {\bf (H1)}-obstruction to stable linearizability. 

\begin{prop}\label{prop:spec3-9}
Let $\cX\to\bA^1_k$ be a family of cubic threefolds $X_b:=\cX_b$  given by 
$$
f_b:=x_1x_2x_3+(x_1+x_2+x_3)x_4x_5+(x_4+x_5)(x_4+bx_5)(bx_4+x_5)=0
$$
for $b\in k$. Consider the $G=C_3$ action on $X_b$ generated by 
$$
(x_1,x_2,x_3,x_4,x_5)\mapsto(x_2,x_3,x_1,x_4,x_5).
$$
Then a very general element in $\cX$ is not $G$-stably linearizable.
\end{prop}

\begin{proof}
Arguing as in Example~\ref{exam:2-spec}, let $\cX\to \bA^1_k$ be the family given by $f_b$. For a very general $b$, $\cX_b:=X_b$ is a 3-nodal cubic described as Case (5) in Proposition~\ref{prop:3nodclassification}.
The special fiber $X_0:=\cX_0$ is a 9-nodal cubic, and the $G$-action fixes a smooth genus 1 curve on $X_0$.  From computations in Section~\ref{sect:nine}, there exists an {\bf (H1)}-obstruction to stable linearizability of the $G$-action on $X_0$.
The six additional nodes form two $G$-orbits with trivial stabilizer. By Proposition~\ref{prop:flat}, a very general element in the family $\cX$ is not $G$-stably linearizable.
\end{proof}

\section{Four nodes}
\label{sect:four}

\subsection*{Factorial cubics}
We first consider the case when the four nodes are in general linear position, forming a ``tetrahedron''. 
This is case (J4) in \cite{Finkelnbergcubic}.
We may assume that the nodes of $X$ are contained in the hyperplane $x_5=0$, and are the points 
$$
[1:0:0:0:0], \,\,[0:1:0:0:0],\,\, [0:0:1:0:0],\,\, [0:0:0:1:0].
$$
The intersection $X\cap\{x_5=0\}$ is the unique cubic surface with $4$ nodes, the Cayley cubic surface. 
Using this, we see that $X$ can be given by 
\begin{multline}
\label{eqn:form4-simplified} 
x_1x_2x_3+x_1x_2x_4+x_1x_3x_4+x_2x_3x_4+\\
+ax_5^3+x_5^2\big(b_1x_1+b_2x_2+b_3x_3+b_4x_4\big)+\\
+x_5\big(a_1(x_1x_2+x_3x_4)+a_2(x_1x_3+x_2x_4)+a_3(x_1x_4+x_2x_3)\big)=0
\end{multline} 
for some $a,b_1,b_2,b_3,b_4,a_1,a_2,a_3\in k$. 

\begin{theo}
\label{theo:Aut-4-nodes}
Suppose that $X\subset \bP^4$ is a 4-nodal cubic threefold and $\mathrm{Aut}(X)$ does not fix any node of $X$.
Then, up to a change of coordinates, one of the following holds:
\begin{enumerate}
\item[($C_2$)] $b_1=b_2$ and $b_3=b_4$ in \eqref{eqn:form4-simplified},
and $\mathrm{Aut}(X)\simeq C_2$, generated by
$$
\sigma_{1}\colon(x_1,x_2,x_3,x_4,x_5)\mapsto(x_2,x_1,x_4,x_3,x_5).
$$

\item[($C_4$)] $a=0$, $a_1=a_2=a_3=0$, $b_1=-b_2$, $b_3=-b_4$  in \eqref{eqn:form4-simplified}, 
and $\mathrm{Aut}(X)\simeq C_4$, generated by
$$
\sigma_{1}^\prime\colon(x_1,x_2,x_3,x_4,x_5)\mapsto(x_2,x_1,x_4,x_3,ix_5),\quad i=e^{\frac{2\pi i}{4}}.
$$

\item[($C_2^2$)] $b_1=b_2=b_3=b_4$ in \eqref{eqn:form4-simplified}, and $\mathrm{Aut}(X)\simeq C_2^2$, generated by $\sigma_1$ and
$$
\sigma_{2}\colon(x_1,x_2,x_3,x_4,x_5)\mapsto(x_3,x_4,x_1,x_2,x_5).
$$

\item[($C_8$)]  $a=0$, $a_1=a_2=a_3=0$, $b_1=1$, $b_2=-\zeta_8^2$, $b_3=-1$, $b_4=\zeta_8^2$ in \eqref{eqn:form4-simplified}, 
and $\mathrm{Aut}(X)\simeq C_8$, generated by 
$$
\sigma_3^\prime\colon(x_1,x_2,x_3,x_4,x_5)\mapsto(x_4,x_1,x_2,x_3,\zeta_8x_5),
\quad \zeta_8=e^{\frac{2\pi i }{8}}.
$$

\item[($\mathfrak{D}_4$)]  $a_1=a_3=0$, $a_2=1$, $b_1=b_2=b_3=b_4$ in \eqref{eqn:form4-simplified}, and $\mathrm{Aut}(X)\simeq \mathfrak{D}_4$, generated by $\sigma_1$, $\sigma_2$ and
$$
\sigma_{3}\colon(x_1,x_2,x_3,x_4,x_5)\mapsto(x_4,x_1,x_2,x_3,x_5).
$$


\item[($\mathfrak{S}_4$)] $a\ne 0$, $a_1=a_2=a_3=0$, $b_1=b_2=b_3=b_4=1$ in \eqref{eqn:form4-simplified}, and $\mathrm{Aut}(X)\simeq \mathfrak{S}_4$, generated by $\sigma_1$, $\sigma_2$, $\sigma_3$ and
$$
\sigma_{4}\colon(x_1,x_2,x_3,x_4,x_5)\mapsto(x_2,x_3,x_1,x_4,x_5).
$$
\end{enumerate}
\end{theo}

\begin{proof}
Let $\phi\colon\mathrm{Aut}(X)\to\mathfrak{S}_4$ be the homomorphism given by the action on the nodes of $X$. Since $\mathrm{Aut}(X)$ does not fix nodes, we may assume that there is a $\iota\in\mathrm{Aut}(X)$ such that  $\phi(\iota)=(12)(34)$ or $\phi(\iota)=(1234)$. 

Suppose that  $\phi(\iota)=(12)(34)$. Then $\iota$ is given by 
$$
(x_1,x_2,x_3,x_4,x_5)\mapsto(x_2+sx_5,x_1+sx_5,x_3+sx_5,x_4+sx_5,tx_5)
$$
for some $s,t\in k$ such that $t\ne 0$. 
Considering how $\iota$ acts on \eqref{eqn:form4-simplified}, we see that $s=0$ or $a_1=a_2=a_3$. In the former case, we have $b_1=b_2$ and $b_3=b_4$, which implies $t=1$, 
because otherwise $t=-1$ and $a=a_1=a_2=a_3=0$, which implies that $X$ is not 4-nodal. 
Thus, if $(s,t)=(0,1)$ and $\mathrm{im}(\phi)\simeq C_2$, then we are in the case ($C_2$).

If $a_1=a_2=a_3$, then, after a coordinate change, we may assume that $a_1=a_2=a_3=0$.
In this case, we get 
$$
0=s=a(1-t^3)=b_3-b_4t^2=b_4-b_3t^2=b_2-b_1t^2=b_1-b_2t^2.
$$
Since $X$ is 4-nodal, this gives $a=0$, $b_1=-b_2$, $b_3=-b_4$ and $t=\pm i$.
Hence, if $\mathrm{im}(\phi)\simeq C_2$, then  we are in the case ($C_4$).

Now, we suppose that  $\phi(\iota)=(1234)$. Then $\iota$ is given by 
$$
(x_1,x_2,x_3,x_4,x_5)\mapsto(x_4+sx_5,x_1+sx_5,x_2+sx_5,x_3+sx_5,tx_5)
$$
for some $s,t\in k$ such that $t\ne 0$. 
Then 
\begin{itemize}
\item $a=2a_2b_4-a_2^3$, $a_1=2a_2-a_3$, $b_1=b_2=b_3=b_4$, or
\item $a_1=a_3$, $b_1=b_2=b_3=b_4$, or 
\item $a_1=a_2=a_3$.
\end{itemize}
In the former case, $X$ is not 4-nodal.
If $a_1=a_2=a_3$, then after a coordinate change, we may assume that $a_1=a_2=a_3=0$, which gives 
$$
0=s=a(1-t^3)=b_2-b_3t^2=b_3-b_4t^2=b_4-b_1t^2=b_1-b_2t^2,
$$
so, after an appropriate scaling of $x_5$, we see that
\begin{itemize}
\item $a\ne 1$, $b_1=b_2=b_3=b_4=1$, or
\item $a=0$, $b_1=1$, $b_2=-1$, $b_3=1$, $b_4 =-1$, $t=i$, or
\item $a=0$, $b_1=1$, $b_2=-\zeta_8^2$, $b_3=-1$, $b_4=\zeta_8^2$, $t=\zeta_8$,
\end{itemize}
which implies that we are in cases ($\mathfrak{S}_4$), ($C_4$), ($C_8$), respectively. 

If $a_1=a_3$ and $b_1=b_2=b_3=b_4$, then, after a coordinate change,
we may assume that $a_1=a_3=0$. If $a_2=0$, then we are in the case ($\mathfrak{S}_4$).
Finally, if $a_2\ne 0$, then, scaling $x_5$, we may further assume that $a_1=1$, so $X$ is given by 
\begin{multline*} 
x_1x_2x_3+x_1x_2x_4+x_1x_3x_4+x_2x_3x_4+\\
+ax_5^3+b_1x_5^2\big(x_1+x_2+x_3+x_4\big)+x_5\big(x_1x_3+x_2x_4\big)=0,
\end{multline*}
which gives $\mathrm{Aut}(X)=\langle\sigma_1,\sigma_2,\sigma_3\rangle\cong\mathfrak{D}_4$,
so we are in the case ($\mathfrak{D}_4$).

To proceed, we may assume that $\mathrm{im}(\phi)\not\simeq C_2$ and $\mathrm{im}(\phi)\not\simeq C_4$. 
Then, up to a coordinate change, one of the following four cases holds:
\begin{itemize}
\item $\mathrm{im}(\phi)=\langle(12)(34),(14)(23)\rangle\simeq C_2^2$,
\item $\mathrm{im}(\phi)=\langle(12)(34),(1234)\rangle\simeq \mathfrak{D}_4$,
\item $\mathrm{im}(\phi)=\langle(12)(34),(14)(23),(123)\rangle\simeq\mathfrak{A}_4$,
\item $\mathrm{im}(\phi)=\langle(12)(34),(14)(23),(1234),(123)\rangle\simeq \mathfrak{S}_4$.
\end{itemize}
Since $\mathrm{im}(\phi)$ contains $(12)(34)$ or $(1234)$,
the cubic $X$ must be given by one of the equations explicitly described above.
Using additional symmetries of $X$, we conclude that we are in one of the cases ($C_2^2$), ($\mathfrak{D}_4$), ($\mathfrak{S}_4$), or the cubic $X$ is given by \eqref{eqn:form4-simplified} with 
$$
a\ne 0,\quad a_1=a_2=a_3=0,\quad b_1=b_3=1,\quad b_2=b_4=-1,
$$
or
$$
a_1=1, \quad a_2=\zeta_3,\quad a_3=\zeta_3^2,\quad b_1=b_2=b_3=b_4=0.
$$
In the first of the latter two cases, $X$ has $8$ nodes, and in the last case, the singularities of $X$ are not nodes. This completes the proof of the theorem.  
\end{proof}

\subsection*{Birational model}
Let $\pi\colon\widetilde{X}\to X$ be the blow up of the nodes of $X$.
Then there exists an $\mathrm{Aut}(X)$-equivariant diagram:
$$
\xymatrix{
\widetilde{X}\ar@{-->}[rr]^{\rho}\ar@{->}[d]_{\pi}&&\widehat{X}\ar@{->}[d]^{\phi}\\%
X&&Y}
$$
where $\rho$ is a composition of flops in the strict transform of the lines passing 
through pair of nodes, $\phi$ is a contraction of the strict transform of 
the hyperplane section containing $4$ nodes (the surface $X\cap\{x_5=0\}$) to a smooth point of the threefold $Y$,
and $Y$ is a smooth divisor in $(\mathbb{P}^1)^4$ of degree $(1,1,1,1)$.
Implicitly, the birational map $X\dasharrow Y$ has been constructed in the proof of \cite[Proposition~4.5]{KuznetsovProkhorov2022}. Note that the anticanonical model of $\widetilde{X}$ is a singular Fano threefold with $6$ nodes of degree $16$,
which can be smoothed to a smooth Fano threefold of degree $16$ and Picard rank $1$.

\subsection*{Burnside formalism}
We realize the situation of Proposition~\ref{prop:burnform} in some of the $4$-nodal cases.
\begin{exam}\label{exam:4nodBurn}
Let $X$ be the cubic threefold given in the Case $(\fD_4)$ or $(\fS_4)$ in Theorem~\ref{theo:Aut-4-nodes}. Consider the group $G\subset\Aut(X)$ where $G=\langle\sigma_2,\sigma_1\sigma_3\rangle=C_2^2$. Then we are in the situation of Proposition~\ref{prop:burnform}, and the $G$-action is not linearizable. In particular, $\sigma_1\sigma_3$ fixes a cubic surface receiving a residual $\sigma_2$-action with a $G$-fixed elliptic curve on it. 
\end{exam}
\subsection*{Specialization}


One can equivariantly specialize 4-nodal cubic threefolds to an 8-nodal one: 

\begin{prop}
   \label{prop:special4to8}
    Let $X_b$ be the $4$-nodal cubic threefold defined by 
    \begin{multline*} 
f_b=x_1x_2x_3+x_1x_2x_4+x_1x_3x_4+x_2x_3x_4
+x_5^2\big(x_1+x_2+b(x_3+x_4)\big)=0.
\end{multline*} 
    For all $b\in k$, $X_b$ carries a $G=C_2$-action generated by 
    $$
    \sigma_1: (x_1,x_2,x_3,x_4,x_5)\mapsto(x_2,x_1,x_4,x_3,x_5).
    $$
   Then $X_b$ is not $G$-stably linearizable for a very general $b.$
\end{prop}
\begin{proof}
    Let $\cX\to \bA^1_k$ be the family given by $f_b$. The generic fiber $X_b$ is a 4-nodal cubic of the type $(C_2)$ in Theorem~\ref{theo:Aut-4-nodes}.
The special fiber $X_{9/4}:=\cX_{9/4}$ is an 8-nodal cubic, 
with an {\bf (H1)}-obstruction to stable linearizability of the $G$-action by Corollary~\ref{coro:8-nodalsummary}. The additional 4 nodes have trivial stabilizer and thus are $BG$-rational singularities.
Applying Proposition~\ref{prop:flat} and 
Example~\ref{exam:BG},
one concludes that a very general member in the family $\cX$ is not $G$-stably linearizable.
\end{proof}
One can also specialize to the Segre cubic threefold: 

\begin{prop}
\label{prop:special4to10}
Let $X_a$ be a cubic of type $(\fS_4)$ in Theorem~\ref{theo:Aut-4-nodes}, i.e. $X_a$ is given by 
$$
ax_5^3+x_1x_2x_3+x_1x_2x_4+x_1x_3x_4+x_2x_3x_4+x_5^2(x_1+x_2+x_3+x_4)=0.
$$
Consider the subgroup $G=\langle\sigma_1,\sigma_2\rangle\simeq C_2^2\subset\Aut(X_a)$. Then, for a very general $a\in k$, the $G$-action on  $X_{a}$ is not stably linearizable. 
\end{prop}
\begin{proof}
Let $\cX\to \bA^1_k$ be the family consisting of $X_a$. 
The special fiber $X_0:=\cX_0$ is a 10-nodal cubic, 
with an {\bf (H1)}-obstruction to stable linearizability of the $G$-action, from computations in \cite{CTZ}. The additional 6 nodes have $C_2$-stabilizers. They are $BG$-rational singularities, by Example~\ref{exam:BG}. 
Applying Proposition~\ref{prop:flat} and 
Example~\ref{exam:BG},
one concludes that a very general fiber is not $G$-stably linearizable.
\end{proof}

\begin{rema}
    We note that the degeneration of cubics in Proposition~\ref{prop:special4to10} is equivalent to the degeneration of divisors in $(\bP^1)^4$ of degree $(1,1,1,1)$, which was studied in \cite[Section 7]{KuznetsovProkhorov2022} and \cite{CFKK}. In particular, the product of projections from four planes in the tetrahedron formed by the four nodes of the cubics gives an $G$-equivariant birational map from the cubics to divisors in $(\bP^1)^4$ of degree $(1,1,1,1)$.
\end{rema}

\subsection*{Cubics with a plane}
Now we treat the case when the four nodes are contained in a distinguished, $G$-stable plane $\Pi$. 
This is case (J6) in \cite{Finkelnbergcubic}. 
Unprojecting from $\Pi$, we have a $G$-equivariant birational map
$$
\phi: X\dashrightarrow X_{2,2},
$$
where $X_{2,2}$ is a smooth complete intersection in $\bP^5$ of two quadrics with a $G$-fixed point $P\in X_{2,2}$, 
and the map $\phi^{-1}$ is a projection from $P$. Linearizability of actions on smooth $X_{2,2}$ is determined by existence of invariant lines \cite[Theorem 24]{HT-intersect}. In particular, we have

\begin{prop}
\label{prop:gfour}
The $G$-action on $X$ is not linearizable 
if and only if no singular points of $X$ are fixed by $G$,
and $X$ does not contain $G$-stable lines that are disjoint from $\Pi$.
\end{prop}

\begin{proof}
We may assume that no singular points of $X$ is $G$-fixed.
If $X$ contains a $G$-stable line that is disjoint from $\Pi$,
then the $G$-action on $X_{2,2}$ is linearizable, by Lemma~\ref{lemm:line}.
Conversely, if the $G$-action on $X_{2,2}$ is linearizable,
then it follows from \cite[Theorem 24]{HT-intersect}
that $X_{2,2}$ contains a $G$-stable line $\ell$.
And $P\not\in\ell$, because otherwise the preimage of $\ell$ on $X$ would be a $G$-fixed singular point.
Similarly, we see that $\ell$ must be disjoint from the four lines in $X_{2,2}$ containing $P$.
Then $\ell$ is mapped by $\phi^{-1}$ to a $G$-stable line in $X$ that is disjoint from the plane $\Pi$. 
\end{proof}

Examples of nonlinearizable actions, based on the Burnside formalism \cite{BnG} or the adaptation to the equivariant context of the {\em torsors over intermediate Jacobians} formalism from \cite{HT-torsor}, \cite{BW-tor}, can be found in \cite[Sections 8.3 and 8.4]{HT-intersect}. As a special case, we have:

\begin{exam}
\label{exam:quad}
We may assume that $X_{2,2}$ is given by 
$$
\sum_{i=1}^6 a_ix_i^2 = \sum_{i=1}^6 x_i^2 =0. 
$$
Let $G=\langle \sigma\rangle$, with $\sigma$ acting diagonally by $(1,1,1,1,-1,-1)$. Then $G$ does not leave invariant any line on $X_{2,2}$ and the action is not linearizable. On the other hand, there is a genus 1 curve $C$ fixed by $G$, obtained by intersecting $X_{2,2}$ with $x_5=x_6=0$. Projecting from any of the  points on $C$, we obtain a singular cubic threefold, generically with four nodes. 
\end{exam}

\begin{exam}
\label{exam:cubic-four}
   Let $X\subset\bP^4_{y_1,\ldots,y_5}$ be the 4-nodal cubic given by 
    $$
    (y_1-y_3)y_2y_4  + (y_2-y_3)y_1y_5  + 
    (y_4-y_5)y_4y_5 - y_4^3 - y_5^3.
    $$
    The four nodes lie on the unique plane $y_4=y_5=0$. 
    The automorphism group $\Aut(X)$ contains $G=C_2^3$ generated by  
    $$
    \iota_1:(y_1,y_2,y_3,y_4,y_5)\mapsto(-y_1,-y_1+y_3,-y_1+y_2,y_4,y_5)
    $$
    $$
    \iota_2: (y_1,y_2,y_3,y_4,y_5)\mapsto(y_2-y_3,y_1-y_3,-y_3,y_4,y_5).
    $$
    and 
    $$
    \iota_3:(y_1,y_2,y_3,y_4,y_5)\mapsto(y_1,y_2,y_3,-y_4,-y_5).
    $$

    Unprojecting $X$ from the unique plane under the map 
    $$
    (y_1,\ldots,y_5)\mapsto (y_1y_5, y_2y_5, y_3y_5, y_4y_5, y_5^2, y_1y_2 - y_2y_3 - y_4^2),
    $$
    one sees that $X$ is $G$-equivariantly birational to a smooth intersection of two quadrics $X_{2,2}\subset\bP^5_{x_1,\ldots,x_6}$ given by 
    $$
    x_1x_2 - x_2x_3- x_4^2 - x_5x_6=x_1x_2 - x_1x_3 + x_4^2 - x_4x_5 - x_5^2 + x_4x_6=0.
    $$
    The $G=C_2^3$ action on the first five coordinates is the same as that on $\bP^4$, $\iota_1$ and $\iota_2$ acts trivially on $x_6$ and $\iota_3$ changes the sign of $x_6$. For any subgroup $G'\subset G$, there is a $G'$-stable line in $X_{2,2}$ if and only if $G'=C_2$ and the character of the $G'$-representation of the ambient $\bA^6_{x_1,\ldots,x_6}$ is
    $$
    (6,0) \quad\text{ or }\quad (6,4).
    $$
    In the first case, $G'$ fixes a singular point of $X$ and thus is linearizable. In the latter case, $G'$ pointwise fixes a smooth intersection of two quadrics in dimension $2$, i.e., a quartic Del Pezzo surface, which contains $16$ lines. The other $C_2$ subgroups have character $(6,2)$. They fix an elliptic curve but do not leave any line invariant in $X_{2,2}$. Any of the other subgroups of $G$ will contain one of the nonlinearizable $C_2$.
\end{exam}




\section{Five nodes}
\label{sect:5}

Now, we suppose that $X$ has $5$ nodes. 

\subsection*{Birational model}
If the nodes are not in general linear position, then there is a distinguished $G$-fixed node, 
and the $G$-action on $X$ is linearizable. 
Hence, we may assume that the nodes of $X$ are 
$$
p_1=[1:0:0:0:0], \quad p_2=[0:1:0:0:0], \quad p_3=[0:0:1:0:0],
$$
$$
p_4=[0:0:0:1:0], \quad p_5=[0:0:0:0:1].
$$
Then $G\subseteq \fS_5$ acts via permutation of coordinates.  We may also assume that $G$ does not fix any of the nodes,
since otherwise the $G$-action is clearly linearizable.

\subsection*{Linearizability}

Using the standard Cremona involution 
$$
\iota: \mathbb{P}^4\dashrightarrow \bP^4,
$$
we obtain a $G$-birational map $\chi\colon X\dasharrow Q$,
where $Q\subset \bP^4$ is a smooth quadric. For more details of this map, see the proof of Theorem~\ref{theo:Avilov}.

\begin{lemm}
Suppose that $G$ does not act transitively on $\mathrm{Sing}(X)$. Then the $G$-action on $X$ is linearizable. 
\end{lemm}

\begin{proof}
Since $G$ does not fix any of the nodes, 
either $G\simeq C_2\times \fS_3$ or $G\simeq C_2\times C_3$.
In both cases, we may assume that $G$ preserves the subset $\{p_1,p_2\}$ and $\{p_3,p_4,p_5\}$.
Then $G$ pointwise fixes the line $l\subset\mathbb{P}^4$ that passes through the points $[1:1:0:0:0]$ and $[0:0:1:1:1]$.
Observe that $\iota(l)=l$, so that the intersection $l\cap Q$ contains $G$-fixed points, which implies the assertion.  
\end{proof}

Thus, we may assume that $G$ acts transitively on the nodes of $X$,
and $G$ contains the 5-cycle $(1,2,3,4,5)$. Then $X$ is defined by  
\begin{multline*}
     x_1x_2x_3  + x_2x_3x_4 + x_1x_2x_5 + x_1x_4x_5  + x_3x_4x_5+\\
    + a(x_1x_2x_4 + x_1x_3x_4+ x_1x_3x_5 + 
    x_2x_3x_5 + x_2x_4x_5)=0,
\end{multline*}
for some $a$. And $Q$ is defined by
\begin{align}\label{eq:Quadric5}
    x_1x_2+x_2x_3 +\cdots +x_5x_1 +a(x_1x_3+x_2x_4 +\cdots +x_5x_2) = 0.
\end{align}
Note that $a\ne -1$, since otherwise $X$ would be $6$-nodal. Then $Q$ is smooth.
For the group $G$, we have the following possibilities: 
\begin{itemize}
\item[(1)] $G\simeq C_5$, 
\item[(2)] $G\simeq \fD_5$,
\item[(3)] $G\simeq C_4\rtimes C_5$ and $a=1$, 
\item[(4)] $G\simeq \mathfrak{A}_5$ and $a=1$, 
\item[(5)] $G\simeq \mathfrak{S}_5$ and $a=1$.
\end{itemize}
In the first case, $G=C_5$, the group $G$ fixes a point in $Q$, and the $G$-action on $X$ is linearizable.
In the second case, the action is necessarily of the form in the following lemma:

\begin{lemm} 
Suppose that $G\simeq \fD_5$ acting on $\bP^4=\bP(\rI\oplus V_2\oplus V_2')$, where $V_2$ and $V_2'$ are two nonisomorphic 2-dimensional irreducible representations of $\fD_5$. Then the $G$-action on every $G$-invariant smooth quadric in $\bP^4$ is linearizable.
\end{lemm}
\begin{proof}
We may assume the $G$ action is generated by 
\begin{align*}
(x_1,\ldots,x_5)& \mapsto(x_4,x_3,x_2,x_1,x_5),
\\
(x_1,\ldots,x_5)& \mapsto(\zeta x_1,\zeta^2x_2,\zeta^3x_3,\zeta^4x_4,x_5),
\end{align*}
where $\zeta=e^{\frac{2\pi i}{5}}$. Smooth $G$-invariant quadrics $Q_{a,b}$ are given by
$$
ax_1x_4-bx_2x_3+x_5^2=0
$$
for $a,b\ne 0$. Notice that each $Q_{a,b}$ is $\fD_5$-isomorphic to $Q_{1,1}$ with the same $\fD_5$-action under a change of variables
$$
x_1'={\sqrt{a}}x_1,x_2'={\sqrt{b}}x_2,x_3'={\sqrt{b}}x_3,x_4'={\sqrt{a}}x_4,x_5'= x_5.
$$
Consider a $G$-invariant conic 
$$
C=\{x_2=x_3=0\}\cap Q_{1,1}
$$
and a $G$-invariant twisted cubic curve
$$
R=\{x_5=x_1x_3-x_2^2=x_2x_4-x_3^2=0\}\cap Q_{1,1}.
$$
The system of quadric hypersurfaces on $\bP^4$ containing both $C$ and $R$ induces a $G$-equivariant birational map $Q_{1,1}\dashrightarrow \bP^3$, see e.g., \cite[Section 5.10]{Chelcalabi}.
\end{proof}


\begin{lemm} 
Suppose that $G\simeq C_4\rtimes C_5$ and $a=1$. 
Then $Q$ from \eqref{eq:Quadric5} contains a $G$-invariant smooth quintic elliptic curve $E$,
and we have the following $G$-Sarkisov link:
$$
\xymatrix{
&\widetilde{Q}\ar@{->}[dl]_{\alpha}\ar@{->}[dr]^{\beta}\\%
Q&& \mathbb{P}^3}
$$
where $\alpha$ is a~blow up of the curve $E$,
and $\beta$ is a blow up of a smooth quintic elliptic curve isomorphic to $E$.
\end{lemm}

\begin{proof}
It is easy to see that $\mathrm{Aut}(Q)$ contains a unique subgroup isomorphic to $C_4\rtimes C_5$.
Thus, we may change coordinates on $\mathbb{P}^4$ as we need 
and, in particular,  assume that $Q$ is given by
$$
\sum_{i=1}^5x_i^2+i\sum_{1\leqslant i<j\leqslant 5}x_ix_j=0, 
$$
and that the action of $G$ on $Q$ is given by 
\begin{align*}
(x_1,x_2,x_3,x_4,x_5)&\mapsto(x_2,x_3,x_4,x_5,x_1),\\
(x_1,x_2,x_3,x_4,x_5)&\mapsto(x_1,x_3,x_5,x_2,x_4).
\end{align*}
Then $Q$ contains the following smooth quintic elliptic curve:
$$
\left\{\aligned
&x_1^2+i(x_3x_4+x_2x_5)=0,\\
&x_2^2+i(x_4x_5+x_3x_1)=0,\\
&x_3^2+i(x_5x_1+x_4x_2)=0,\\
&x_4^2+i(x_1x_2+x_5x_3)=0,\\
&x_5^2+i(x_2x_3+x_1x_4)=0.
\endaligned
\right.
$$
Blowing up $Q$ along this curve, we obtain the claim, cf. \cite{CheltsovPokora}.
\end{proof}

If $G\simeq \mathfrak{S}_5$ and $a=1$, then it follows from \cite{VAZ} that $X$ is $G$-solid,
and the only $G$-Mori fiber spaces $G$-birational to $X$ are $X$ and $Q$. In particular, the $G$-action is not linearizable. If $G\simeq\mathfrak{A}_5$ and $a=1$, we also expect that  $X$ and $Q$ are the only $G$-Mori fiber spaces $G$-birational to $X$, which would imply that the $G$-action is not linearizable.

\section{Six nodes}
\label{sect:six}

\subsection*{Cubics without planes}

Let $X$ be the 6-nodal cubic threefold such that the nodes are in general linear position. 
Then $\mathrm{rk}\,\mathrm{Cl}(X)=2$, so the defect of $X$ is $1$.
This is case (J9) in \cite{Finkelnbergcubic}. Note that $X$ does not contain planes, but it contains two families of cubic scrolls (see Remark~\ref{rema:scrolls} below).
Moreover, by \cite[Section 3]{HT-determinant}, $X$ can be given by 
$$
\mathrm{det}(M)=0
$$
for a $3\times 3$ matrix $M$ whose entries are linear forms. 
Thus, one can define a rational map $X\dasharrow\mathbb{P}^2$ that maps $p\mapsto (a,b,c)$,
where $(a,b,c)$ is a non-zero solution of the equation
$$
M\begin{pmatrix}
a\\
b\\
c
\end{pmatrix}=0.
$$
This map is dominant, it is undefined at the nodes of $X$, 
and its general fiber is a line in $X$.
Similarly, we can define another rational map $X\dasharrow\mathbb{P}^2$ using the transpose of the matrix $M$.
Taking resolution of singularities $X$, we resolve indeterminacy of both of these rational maps,
which yields the following commutative diagram:
\begin{equation}
\label{eq:6-nodes-link}
\xymatrix{
&&\widetilde{X}\ar@{->}[dll]_{h^+}\ar@{->}[drr]^{h^-}\ar@{->}[dd]_{f}&&\\
X^+\ar@{->}[d]_{p^+}\ar@{->}[drr]^{q^+}&&&&X^-\ar@{->}[d]^{p^-}\ar@{->}[dll]_{q^-}\\%
\mathbb{P}^2&&X\ar@{-->}[ll]\ar@{-->}[rr]&&\mathbb{P}^2}
\end{equation}
where $f$ is the standard resolution, $q^+$ and $q^-$ 
are small resolutions,
$h^+$ and $h^-$ are birational morphisms such that $h^-\circ (h^+)^{-1}$ is a composition of six Atiyah flops,
both $p^+$ and $p^-$ are $\mathbb{P}^1$-bundles.
The diagram \eqref{eq:6-nodes-link} is implicitly contained in \cite[\S~7.5]{JahnkePeternellRadloff2011}, as an illustration of the first row in the table there.
Taking a product of morphisms $p^+\circ h^+$ and $p^-\circ h^-$,
we obtain a morphism $\widetilde{X}\to\mathbb{P}^2\times\mathbb{P}^2$ that is birational onto its image (a divisor of degree $(2,2)$ with $15$ nodes).

\begin{rema}
\label{rema:scrolls}
Let $l$ be a general line in $\mathbb{P}^2$.
Set 
$$
\overline{S}=(q^-)_*(p^{-})^*(l)\quad \text{ and } \quad \overline{S}'=(q^+)_*(p^{+})^*(l).
$$
Then $\overline{S}$ and $\overline{S}'$ are smooth cubic scrolls in $X$ that freely generate the class group $\mathrm{Cl}(X)$.
\end{rema}

\begin{rema}
Let $G\subseteq\mathrm{Aut}(X)$. Then the commutative diagram \eqref{eq:6-nodes-link} is $G$-equivariant 
if and only if $\mathrm{rk}\,\mathrm{Cl}^G(X)\ne 1$.  
\end{rema}

To describe possibilities for $\mathrm{Aut}(X)$, we can assume that the nodes of $X$ are the points 
$$
p_1=[1:0:0:0:0],\quad p_2=[1:1:1:1:1], \quad p_3=[0:0:0:0:1].
$$
$$
p_4=[0:0:0:1:0],\quad p_5=[0:0:1:0:0], \quad p_6=[0:1:0:0:0].
$$
Fix the $\fS_6$-action on $\mathbb{P}^4$ generated by
\begin{align}\label{eq:6nodS6}
\tau_{(12)}:(x_1,\ldots,x_5)&\mapsto (-x_1, -x_1+x_2,-x_1+x_3,-x_1+x_4,-x_1+x_5),
\\
\tau_{(1\cdots6)}: (x_1,\ldots,x_5)&\mapsto (-x_1+x_2, -x_1+x_3,-x_1+x_4,-x_1+x_5,-x_1),\nonumber
\end{align}
where the indices corresponds to the permutation of 6 nodes. Then $\mathrm{Sing}(X)$ forms an $\fS_6$-orbit, but $X$ is not $\fS_6$-invariant.
Moreover, it follows from a classical construction \cite{CorayCo} that there exists the following $4$-dimensional 
$\fS_6$-Sarkisov link:
$$
\xymatrix{
U\ar@{->}[d]_{\alpha}\ar@{-->}[rr]^{\beta}&&Y\ar@{->}[d]^{\gamma}\\%
\mathbb{P}^4\ar@{-->}[rr]_{\chi}&&V}
$$
where $V$ is the 10-nodal Segre cubic threefold in $\mathbb{P}^4$,
$\chi$ is given by the linear system of cubic hypersurfaces singular at the points $p_1,\ldots,p_6$,
$\alpha$ is the blowup of $p_1,\ldots,p_6$, 
$\beta$ is a composition of antiflips in 
the strict transforms of the $15$ lines that contain $2$ points among $p_1,\ldots,p_6$,
and $\gamma$ is a $\mathbb{P}^1$-bundle.

Observe that $\Aut(X)\subseteq \fS_6$. 
Restricting the above $\fS_6$-Sarkisov link to $X$, 
we obtain the following $\mathrm{Aut}(X)$-equivariant  diagram:
\begin{equation}
\label{eq:6-nodes-P1-bundle}
\xymatrix{
\widetilde{X}\ar@{->}[d]_{f}\ar@{-->}[rr]&&\widehat{X}\ar@{->}[d]^{\pi}\\%
X &&S}
\end{equation}
where $S$ is a smooth hyperplane section of the Segre cubic $V$, 
$\widetilde{X}\dasharrow\widehat{X}$ is a composition of Atiyah flops in 
the strict transforms of the $15$ lines in $X$ that contains $2$ nodes among $p_1,\ldots,p_6$,
$\pi$ is a $\mathbb{P}^1$-bundle. For more details, see \cite{HT-determinant}.

Our cubic $X$ is given by 
\begin{align}
    \label{eq:6-nodalgeneralpositionform}
    a_1f_1+a_2f_2+a_3f_3+a_4f_4+a_5f_5=0
\end{align}
for some $a_1,a_2,a_3,a_4,a_5\in k$, where
 \begin{align*}
 f_1&=x_1x_2x_3 - x_2x_3x_4 - x_2x_3x_5 - x_1x_4x_5 + 
        x_2x_4x_5 + x_3x_4x_5,\\
    f_2&=x_1x_2x_4 - x_2x_3x_4 - x_1x_4x_5 + x_3x_4x_5,\\
    f_3&=x_1x_2x_5 - x_2x_3x_5 - x_1x_4x_5 + x_3x_4x_5,\\
    f_4&=x_1x_3x_4 - x_2x_3x_4 - x_1x_4x_5 + x_2x_4x_5,\\
    f_5&=x_1x_3x_5 - x_2x_3x_5 - x_1x_4x_5 + x_2x_4x_5. 
\end{align*}
Enumerating $G\subseteq \fS_6$ and searching for $G$-invariant cubics singular at $p_1,\ldots,p_6$, we can find all possibilities for $\mathrm{Aut}(X)$. In particular, $\Aut(X)=1$ for general $a_1,\ldots, a_5$. Moreover, one has

\begin{prop}\label{prop:6nodegeneralprop}
Let $X\subset \bP^4$ be a 6-nodal cubic threefold given by \eqref{eq:6-nodalgeneralpositionform}. Assume that none of the nodes of $X$ is fixed by $\Aut(X)$. Then under the $\fS_6$-action specified in~\eqref{eq:6nodS6}, one of the following holds:
    \begin{enumerate}
        \item $a_1+a_2+a_4+a_5=0$, and 
        $$
        \Aut(X)\simeq C_2=\langle (1, 3)(2, 5)(4, 6)\rangle.
        $$
        \item $a_1+a_3=a_2-a_3+a_4+a_5=0$, and 
        $$
        \Aut(X)\simeq \fS_3=\langle (1, 3)(2, 5)(4, 6), (1, 4, 5)(2, 6, 3)\rangle.
        $$
        \item $a_1+a_4=a_2+a_5=a_3-a_4=0$, and 
        $$
        \Aut(X)\simeq \fS_4=\langle (1, 3)(2, 5)(4, 6), (3,4,5,6)\rangle.
        $$
        \item $a_1+a_4=a_2+a_5=a_3+a_4=0$, and 
        $$
        \Aut(X)\simeq \fD_4=\langle(3,5), (1,3,2,5)(4, 6)\rangle.
        $$
        \item $a_1+a_4+2a_5=a_2-a_5=0$, and 
        $$
        \Aut(X)\simeq C_2^2=\langle(1, 2)(3, 5)(4, 6), (1, 3)(2, 5)(4, 6)\rangle.
        $$
        \item $a_1+a_4+2a_5=a_2-a_5=a_3-a_4-2a_5=0$, and 
        $$
        \Aut(X)\simeq \fD_6=\langle (1, 3)(2, 5)(4, 6), (1, 6, 5, 2, 4, 3)\rangle.
        $$
        \item $a_1=a_3=a_5=1, a_2=a_4=-1$, and 
        $$
        \Aut(X)\simeq \fS_3^2\rtimes C_2=\langle(1,3)(2,5)(4,6), (2, 4), (1, 5)(2, 3, 4, 6)\rangle.
        $$
        \item $a_1=a_4=1, a_2=a_3=a_5=-1$, and 
        $$
        \Aut(X)\simeq \fS_5=\langle(1,3)(2,5)(4,6),(1, 2, 5, 6, 4)\rangle.
        $$
    \end{enumerate}
\end{prop}
\begin{proof}
  Enumerating all (conjugacy classes of) subgroups $G$ of $\fS_6$ which do not fix any point among $p_1,\ldots, p_6$, and computing all $G$-invariant cubics singular at $p_1,\ldots,p_6$, we obtain the list of (families of) 6-nodal cubics whose automorphism groups do not fix any of the nodes. These are the eight families of cubics listed above. Since $\Aut(X)\subset \fS_6$, one can find the full automorphism groups $\Aut(X)$.
\end{proof}
As in \cite{Avilov-note}, we find two maximal subgroups $\fS_5$ and $\fS_3^2\rtimes C_2$ such that (up to conjugation in $\fS_6$) $G$ and $X$ can be described as follows:
\begin{enumerate}
\item $G=\fS_5=\langle (1,3)(2,5)(4,6),(1, 2, 5, 6, 4)\rangle$ and $X$ is given by
\begin{multline}\label{eq:S56nodal}
     x_1x_2x_3 - x_1x_2x_4 + x_1x_3x_4 - x_2x_3x_4 - 
            x_1x_2x_5 - x_1x_3x_5 +\\ +x_2x_3x_5 + x_1x_4x_5 + 
            x_2x_4x_5 - x_3x_4x_5=0,
\end{multline}
\item $G=\fS_3^2\rtimes C_2=\langle(1,3)(2,5)(4,6), (2, 4), (1, 5)(2, 3, 4, 6)\rangle$ and $X$ is given by 
\begin{multline}
 x_1x_2x_3 - x_1x_2x_4 - x_1x_3x_4 + x_2x_3x_4 + 
            x_1x_2x_5 + x_1x_3x_5 +\\ +x_2x_4x_5 -3x_2x_3x_5 - x_1x_4x_5 + 
             x_3x_4x_5=0.
\end{multline}
\end{enumerate}
In the first case, $\mathrm{Aut}(X)\simeq \fS_5$, $\mathrm{rk}\,\mathrm{Cl}^{\fS_5}(X)=1$, 
and \eqref{eq:6-nodes-P1-bundle} is a $\fS_5$-Sarkisov link such that $S$ is the Clebsch diagonal cubic surface.
In the second case, $\mathrm{Aut}(X)\simeq \fS_3^2\rtimes C_2$, $\mathrm{rk}\,\mathrm{Cl}^{\fS_3^2\rtimes C_2}(X)=1$, 
and \eqref{eq:6-nodes-P1-bundle} is a $\fS_3^2\rtimes C_2$-Sarkisov link such that $S$ is the Fermat cubic surface. 

\begin{lemm}
    \label{lemm:6nodescroll}
    Let $X\subset \bP^4$ be a 6-nodal cubic threefold such that the nodes are in general linear position. If $\mathrm{Aut}(X)$ contains an involution $\sigma$ not fixing any node, then $\mathrm{rk}\, \mathrm{Cl}^{\langle\sigma\rangle}(X)=1$.
\end{lemm}
\begin{proof}
Since $\sigma$ does not fix any node, we may assume that 
$$
\sigma=(1,3)(2,5)(4,6)
$$ 
and $X$ is one of the cases in Proposition~\ref{prop:6nodegeneralprop}. From the diagram~\eqref{eq:6-nodes-P1-bundle}, we know that 
$$
\mathrm{rk}\,\mathrm{Cl}^{\langle\sigma\rangle}(X)+3=\mathrm{rk}\,\mathrm{Cl}^{\langle\sigma\rangle}(\widetilde X)=\mathrm{rk}\,\mathrm{Cl}^{\langle\sigma\rangle}(\widehat X)=\mathrm{rk}\,\mathrm{Cl}^{\langle\sigma\rangle}(S)+1,
$$    
where $S$ is a smooth cubic surface contained in a hyperplane $H\subset\bP^4$. By Lefschetz fixed-point theorem, one has \cite[Section 6]{DI}
$$
\mathrm{rk}\, \mathrm{Cl}^{\langle\sigma\rangle}(S)=\frac12\left(7+\mathrm{Tr}_2\left(\sigma^*\right)\right),
$$
$$
\mathrm{Tr}_2(\sigma^*)=s-2+\sum_i(2-2g_i),
$$
where $\mathrm{Tr}_2(\sigma^*)$ is the trace of $\sigma^*$-action on $\rH^2(S,\bC)$, $s$ is the number of isolated $\sigma$-fixed points on $S$ and $g_i$ are the genera of fixed curves. In our case, we compute that the induced $\sigma$-action on $H\simeq\bP^3$ has weights $(1,1,1,-1)$. The fixed locus $S^\sigma$ consists of one point and a smooth cubic curve. Substituting into the formulae above we obtain 
$$
\mathrm{rk} \,\mathrm{Cl}^{\langle\sigma\rangle}(S)=3,
$$
which implies $\mathrm{rk}\, \mathrm{Cl}^{\langle\sigma\rangle}(X)=1.
$
\end{proof}
\begin{prop}
\label{prop:H16nodescroll}
Let $X\subset \bP^4$ be a 6-nodal cubic threefold such that the nodes are in general linear position
and $\mathrm{Aut}(X)$ contains an involution $\sigma$.
Let $\widetilde{X}$ be the standard resolution of $X$. 
Then the action of $\langle \sigma\rangle\simeq C_2$  on $\Pic(\widetilde{X})$ fails {\bf (H1)}  if and only if
$\sigma$ does not fix any node.
 \end{prop}

\begin{proof}
 We know that $\Pic(\widetilde{X})$ is generated by the pullback of the hyperplane section $H$, six exceptional divisors  $E_1,\ldots,E_6$, and 
the classes of the strict transforms of two cubic scrolls $S$ and $S'$ (see Remark~\ref{rema:scrolls}), subject to the relation
$$
2H = S+S'+\sum_{i=1}^6E_i. 
$$
There is a short exact sequence of $\Aut(X)$-modules
$$
0\to \bigoplus_{i=1}^6 E_i\to\Pic(\widetilde{X})\to\mathrm{Cl}(X)\to 0,
$$
giving rise to the long exact sequence of cohomology groups
\begin{multline*}
\ldots\to\rH^1(\langle\sigma\rangle,\bigoplus_{i=1}^6E_i)\to \rH^1(\langle \sigma\rangle, \Pic(\widetilde X))\to  \rH^1(\langle \sigma\rangle,\mathrm{Cl}(X))\\
\to\rH^2(\langle \sigma\rangle,\bigoplus_{i=1}^6E_i)\to\ldots.
\end{multline*}
By our assumption, $\sigma$ permutes the $E_i$ without fixing any $E_i$. So
$$
\rH^1(\langle\sigma\rangle,\bigoplus_{i=1}^6E_i)=\rH^2(\langle\sigma\rangle,\bigoplus_{i=1}^6E_i)=0,
$$
and 
$$
 \rH^1(\langle \sigma\rangle, \Pic(\widetilde X))=\rH^1(\langle \sigma\rangle,\mathrm{Cl}(X)).
$$
If $\sigma$ does not fix any node, Lemma~\ref{lemm:6nodescroll} implies that 
 $\mathrm{rk}\, \mathrm{Cl}^{\langle\sigma\rangle}(X)=1$. So $\sigma$ acts on $\mathrm{Cl}(X)$ via
$$
\sigma(H)=H,\quad\sigma(S)=S'=2H-S.
$$
In another basis of $\mathrm{Cl}(X)$, namely $H$ and $H-S$, the action becomes 
$$
\sigma(H)=H,\quad\sigma(H-S)=-H+S.
$$
Then
$$
 \rH^1(\langle \sigma\rangle, \Pic(\widetilde X))=\rH^1(\langle \sigma\rangle,\mathrm{Cl}(X))=\bZ/2.
$$
Conversely, if $\langle\sigma\rangle$ fails {\bf(H1)}, it is not stably linearizable and thus cannot fix any node.
\end{proof}

\begin{exam}\label{exam:6nodinv}
Let $X\subset \bP^4$ be a 6-nodal cubic threefold in one of the 8 cases in Proposition~\ref{prop:6nodegeneralprop}.
Then $\mathrm{Aut}(X)$ contains the involution 
$$
\sigma=(1,3)(2,5)(4,6),
$$ 
satisfying the conditions in Proposition~\ref{prop:H16nodescroll}. 
Therefore, the $\sigma$-action on any 6-nodal cubic is not stably linearizable.
\end{exam}

\subsection*{Cubics with a plane}
This is case (J8) in \cite{Finkelnbergcubic}.
Four of the six nodes of $X$ are contained in a unique, and thus $G$-stable plane $\Pi\subset X$. 
The other two are on a $G$-stable line $\ell$.
Note that  $\ell\cap\Pi=\varnothing$, since otherwise the hyperplane containing $\Pi$ and $\ell$
would intersect $X$ by three planes. So, the action of $\mathrm{Aut}(X)$ on $X$ is linearizable by
Lemma~\ref{lemm:line}.

\subsection*{Cubics with three planes}
Let $X$ be a cubic threefold in $\mathbb{P}^4$ with 6 nodes such that $X$ contains three planes $\Pi_1$, $\Pi_2$, $\Pi_3$.
Then $X$ belongs to a four-parameter family, which is denoted by (J11) in \cite{Finkelnbergcubic}.
It follows from  \cite{Finkelnbergcubic} that $\Pi_1+\Pi_2+\Pi_3$ is cut out by a hyperplane.
Thus, we may assume that this hyperplane is $\{x_1=0\}$, and 
\begin{align*}
\Pi_1&=\big\{x_1=0,x_2=0\big\},\\
\Pi_2&=\big\{x_1=0,x_3=0\big\},\\
\Pi_3&=\big\{x_1=0,x_4=0\big\}.
\end{align*}

Observe the existence of the following diagram:
$$
\xymatrix{
&\widetilde{X}\ar@{->}[dl]_{\pi}\ar@{->}[dr]^{\eta}\\%
X && Y}
$$
where $\pi$ is the standard resolution, $Y$ is the double cover of $(\mathbb{P}^1)^3$ branched over a singular divisor of degree $(2,2,2)$,
and $\eta$ is a birational morphism that contracts the strict transforms of $\Pi_1$, $\Pi_2$, $\Pi_3$.
Note that $\mathrm{Aut}(Y)$ contains a Galois involution of the double cover, 
and this involution acts biregularly on $X$, which follows from: 

\begin{prop}
\label{prop:6-nodes-3-planes}
Up to a change of coordinates, $X$ is given by
\begin{equation}
\label{equation:6-nodes-3-planes}
x_2x_3x_4+ax_1^3+x_1^2(b_1x_2+b_2x_3+b_3x_4)+x_1(x_2^2+x_3^2+x_4^2-x_5^2)=0,
\end{equation}
for some $a,b_1,b_2,b_3$.
\end{prop}

\begin{proof}
A priori, the cubic $X$ is given by
\begin{multline*}
x_2x_3x_4+ax_1^3+x_1^2\big(b_1x_2+b_2x_3+b_3x_4+cx_5\big)+\\
+x_1\big(e_1x_2^2+e_2x_3^2+e_3x_4^2+e_4x_2x_3+e_5x_2x_4+e_6x_3x_4\big)+\\
+x_1\big(x_5(d_1x_2+d_2x_3+d_3x_4)-x_5^2\big)=0
\end{multline*}
for some $a,b_1,b_2,b_3,c,e_1,e_2,e_3,e_4,e_5,e_6,d_1,d_2,d_3$.
Changing $x_2$, $x_3$, $x_4$, we may assume that  $e_4=e_5=e_6=0$.
Moreover, up to scaling, there exists a unique such choice of coordinates $x_2$, $x_3$, $x_4$
that preserves the equations of the planes $\Pi_1$, $\Pi_2$, $\Pi_3$.

Similarly, changing the coordinate $x_5$, we may further assume that $c=d_1=d_2=d_3=0$.
As above, we see that there is a unique such choice for $x_5$ up to scaling. 

Finally, using the fact that $X$ has six nodes, we see that $e_1\ne 0$, $e_2\ne 0$, $e_3\ne 0$.
Hence, scaling the coordinates $x_1,x_2,x_3,x_4,x_5$, we may also assume that $e_1=e_2=e_3=1$,
which completes the proof.
\end{proof}

\begin{rema}
\label{rema:6-nodes}
If we permute $b_1,b_2,b_3$ in \eqref{equation:6-nodes-3-planes},
or simultaneously change signs of two of them,
we obtain an isomorphic cubic threefold. 
\end{rema}

From now on, we assume that the cubic threefold $X$ is given by \eqref{equation:6-nodes-3-planes}.
Then the nodes of $X$ are 
\begin{align*}
p_1&=[0:0:0:1:1], &p_2&=[0:0:0:-1:1],\\ 
p_3&=[0:0:1:0:1], &p_4&=[0:0:-1:0:1],\\
p_5&=[0:1:0:0:1], &p_6&=[0:-1:0:0:1]. 
\end{align*}

\begin{rema}
\label{rema:6-nodes-S3}
Let $S_3$ be the cubic surface $\{x_5=0\}\cap X$. Then $S_3$ is smooth. The proof of Proposition~\ref{prop:6-nodes-3-planes} shows that $S_3$ is $\mathrm{Aut}(X)$-invariant.
\end{rema}

\begin{rema}
\label{rema:6-nodes-degeneration}
One can find an explicit condition on $a,b_1,b_2,b_3$ that guarantees that \eqref{equation:6-nodes-3-planes} defines a 6-nodal cubic, but it is too bulky to present here.
However, if the equation \eqref{equation:6-nodes-3-planes} has additional symmetries, 
the condition simplifies a lot. 
For instance, if $b_1=b_2=b_3=b$, then \eqref{equation:6-nodes-3-planes} defines a 6-nodal cubic if and only if 
$$
(4a-b^2+8b+16)(4b^3+a^2-6ab-3b^2+4a)\ne 0.
$$
Moreover, in this very special case, we have the following possibilities:
\begin{enumerate}
\item if $4b^3+a^2-6ab-3b^2+4a=0$ and \\ $(a,b)\not\in\{(1,1),(-4,0),(-8,28)\}$, then \eqref{equation:6-nodes-3-planes} defines a 7-nodal cubic;
\item if $(a,b)=(1,1)$, then \eqref{equation:6-nodes-3-planes} defines a cubic with six nodes and one double non-nodal singularity;
\item if $4a-b^2+8b+16=0$ and \\
$(a,b)\not\in\{(-4,0),(-8,28)\}$, then  \eqref{equation:6-nodes-3-planes} defines a 9-nodal cubic; 
\item if $(a,b)\in\{(-4,0),(-8,28)\}$, then \eqref{equation:6-nodes-3-planes} defines the Segre cubic.
\end{enumerate}
\end{rema}

As we mentioned earlier, $\mathrm{Aut}(X)$ is never trivial, since it contains the involution:
$$
\iota_5: (x_1,x_2,x_3,x_4,x_5)\mapsto(x_1,x_2,x_3,x_4,-x_5).
$$
Moreover, if $b_1,b_2,b_3$ in \eqref{equation:6-nodes-3-planes} are general enough, then $\mathrm{Aut}(X)=\langle\iota_5\rangle$.
In fact, we can say more: 

\begin{prop}
\label{prop:Aut6nod3pl}
Suppose $\mathrm{Aut}(X)\ne\langle\iota_5\rangle$.
Then, up to a permutation of coordinates $x_2,x_3,x_4$ and changing signs of two of them,
one of the following holds:
\begin{enumerate}
\item $b_1\ne b_2,\,\, b_2=b_3,\,\, b_1,b_2\ne 0$, so $X$ is given by 
$$
x_2x_3x_4+ax_1^3+x_1^2(b_1x_2+b_2(x_3+x_4))+x_1(x_2^2+x_3^2+x_4^2-x_5^2)=0,
$$
and $\Aut(X)\simeq C_2^2$, generated by $\iota_5$ and
$$
\sigma_{34}: (x_1,x_2,x_3,x_4,x_5)\mapsto(x_1,x_2,x_4,x_3,x_5);
$$

\item $b_1\ne 0,\,\, b_2=b_3=0$, so $X$ is given 
$$
x_2x_3x_4+ax_1^3+b_1x_1^2x_2+x_1(x_2^2+x_3^2+x_4^2-x_5^2)=0,
$$
and $\Aut(X)\simeq C_2^3$, generated by $\iota_5$, $\sigma_{34}$ and 
$$
\iota_{34}: (x_1,x_2,x_3,x_4,x_5)\mapsto(x_1,x_2,-x_3,-x_4,x_5);
$$

\item $b_1=b_2=b_3\ne 0$, so $X$ is given by
$$
x_2x_3x_4+ax_1^3+b_1x_1^2(x_2+x_3+x_4)+x_1(x_2^2+x_3^2+x_4^2-x_5^2)=0,
$$
and $\Aut(X)\simeq C_2\times\fS_3$, generated by $\iota_5$, $\sigma_{34}$ and 
$$   
\sigma_{234}: (x_1,x_2,x_3,x_4,x_5)\mapsto(x_1,x_3,x_4,x_2,x_5);
$$

\item $b_1=b_2=b_3=0$, so $X$ is given by
$$
x_2x_3x_4+ax_1^3+x_1(x_2^2+x_3^2+x_4^2-x_5^2)=0,
$$
and $\Aut(X)\simeq C_2\times\fS_4$, generated by $\iota_{5},\sigma_{34},\sigma_{234},\iota_{34}$ and 
$$
\iota_{24}: (x_1,x_2,x_3,x_4,x_5)\mapsto(x_1,-x_2,x_3,-x_4,x_5).
$$
\end{enumerate}
\end{prop}

\begin{proof}
Permuting the coordinates $x_2,x_3,x_4$, we may assume that one of the following four cases hold:
\begin{enumerate}
\item $b_1\ne 0$, $b_2\ne 0$, $b_3\ne 0$;  
\item $b_1=0$, $b_2\ne 0$, $b_3\ne 0$; 
\item $b_1=0$, $b_2=0$, $b_3\ne 0$; 
\item $b_1=0$, $b_2=0$, $b_3=0$.
\end{enumerate}
In the first two cases, we may assume $b_2$ and $b_3$ have the same sign by changing the signs of two among three variables $x_2,x_3$ and $x_4$.

There is a natural homomorphism $\phi\colon\mathrm{Aut}(X)\to\mathfrak{S}_3$ given by the action of $\mathrm{Aut}(X)$ on the planes $\Pi_1$, $\Pi_2$, $\Pi_3$. 
Arguing as in the proof of Proposition~\ref{prop:6-nodes-3-planes}, we see that an element in the kernel of $\phi$ is given by
$$
(x_1,x_2,x_3,x_4,x_5)\mapsto(x_1,\lambda_1x_2,\lambda_2x_3,\lambda_3x_4,\lambda_4x_5)
$$
for some non-zero $\lambda_1,\lambda_2,\lambda_3,\lambda_4$.
Using this, we see that the kernel of $\phi$ can be described as follows:
\begin{itemize}
\item if $b_2\ne 0$, $b_3\ne 0$, then $\mathrm{ker}(\phi)=\langle\iota_5\rangle\simeq C_2$,
\item if $b_1=b_2=0$, $b_3\ne 0$, then $\mathrm{ker}(\phi)=\langle\iota_5,\iota_{23}\rangle\simeq C_2^2$,
\item if $b_1=b_2=b_3=0$, then $\mathrm{ker}(\phi)=\langle\iota_5,\iota_{23},\iota_{24}\rangle\simeq C_2^3$.
\end{itemize}

Let $G=\Aut(X)$.  First, assume $(2,3)\in \phi(G)$, i.e., there exists an element $\sigma\in G$ swapping the planes $\Pi_2$ and $\Pi_3$ and leaving $\Pi_1$ invariant. Then $\sigma$ takes the form 
$$
\begin{pmatrix}
1&0&0&0&0\\
s_1&s_9&0&0&0\\
s_2&0&0&s_{10}&0\\
s_3&0&s_{11}&0&0\\
s_4&s_5&s_6&s_7&s_8
\end{pmatrix}
$$ 
for parameters $s_1,\ldots,s_{11}$. Note that $\sigma^2$ is contained in the kernel of $\phi$,
which implies that $s_8=\pm 1$. Moreover, we may assume that $s_8=1$ by replacing $\sigma$ by $\sigma\circ\iota_5$.
Furthermore, the fact that $\sigma$ leaves $X$ invariant imposes relations on the parameters. Solving for the equations, we obtain solutions
\begin{itemize}
\item $b_2=b_3$, $s_1=\ldots=s_7=0$, $s_{10}=s_{11}=s_9=1$,
\item $b_2=-b_3$, $s_1=\ldots=s_7=0$, $s_{10}=s_{11}=-1$, $s_9=1$.
\end{itemize}

Similarly, if $(1,2,3)\in\phi(G)$, i.e., there exists an element $\sigma\in G$ translating three planes. 
As above, we see that $\sigma$ takes the form 
$$
\begin{pmatrix}
1&0&0&0&0\\
s_1&0&0&s_9&0\\
s_2&s_{10}&0&0&0\\
s_3&0&s_{11}&0&0\\
s_4&s_5&s_6&s_7&1
\end{pmatrix}.
$$ 
In this case, we obtain 4 solutions 
\begin{itemize}
\item $b_1=b_2=-b_3$, $s_1=\ldots=s_7=0$, $s_9=s_{11}=-1$, $s_{10}=1$,
\item $b_1=-b_2=-b_3$, $s_1=\ldots=s_7=0$, $s_9=s_{10}=-1$, $s_{11}=1$,
\item $b_1=-b_2=b_3$, $s_1=\ldots=s_7=0$, $s_9=1$, $s_{10}=s_{11}=-1$,
\item $b_1=b_2=b_3$, $s_1=\ldots=s_7=0$, $s_9=s_{10}=s_{11}=1$.
\end{itemize}
Combining these solutions with symmetries, we obtain the result.
\end{proof}

\subsection*{Linearization}
Let $X$ be the 6-nodal cubic given by \eqref{equation:6-nodes-3-planes}. In this subsection we solve the linearizability problem for subgroups in $\mathrm{Aut}(X)$, almost completely. We use notation introduced in the previous subsection, and let $S_3$ be the cubic surface $\{x_5\}\cap X$. Then $S_3$ is smooth by Remark~\ref{rema:6-nodes-S3},
which implies:

\begin{lemm}
\label{lemm:6-nodes-iota-4}
Let $G=\langle\iota_5\rangle$. Then the $G$-action on $X$ is linearizable.
\end{lemm}

\begin{proof}
The surface $S_3$ is pointwise fixed by $\iota_5$, and $\Pi_1\cap S_3$ is a line.
Since $S_3$ is smooth, it contains another line $l$ disjoint from $\Pi_1\cap S_3$.
Hence, $l$ is disjoint from $\Pi_1$.
Since $\Pi_1$ is $G$-invariant, the $G$-action is linearizable by Lemma~\ref{lemm:line}.
\end{proof}

Similarly, we prove  

\begin{lemm}\label{lemm:63linear}
Suppose that $b_2=b_3=0$, so $X$ is given by
$$
x_2x_3x_4+ax_1^3+b_1x_1^2x_2+x_1(x_2^2+x_3^2+x_4^2-x_5^2)=0.
$$
Let $G=\langle\iota_5,\iota_{34}\rangle\simeq C_2^2$.
Then the $G$-action on $X$ is linearizable.
\end{lemm}

\begin{proof}
Note that $G$ leaves invariant the planes $\Pi_1$, $\Pi_2$, $\Pi_3$.
As in the proof of Lemma~\ref{lemm:6-nodes-iota-4}, we see that $S_3$ contains a $G$-invariant line that is disjoint from one of these planes. Indeed, if $r$ is a root of $r^2+b_1r+a=0$,
then $S_3$ contains the reducible conic
$$
x_5=x_2-rx_1=rx_3x_4+x_3^2+x_4^2=0,
$$
and its irreducible components are $G$-invariant lines disjoint from $\Pi_2$ and $\Pi_3$. 
Hence, the $G$-action is linearizable by Lemma~\ref{lemm:line}.
\end{proof}

Now, let us discuss nonlinearizable actions. We start with

\begin{lemm}\label{lemm:63Burn}
Suppose that $b_2=b_3$, and let $G=\langle\iota_5,\sigma_{34}\rangle\cong C_2^2$.
Then the $G$-action on $X$ is not linearizable.
\end{lemm}

\begin{proof}
The involution $\iota_5$ pointwise fixes the $G$-invariant surface $S_3$, while the involution $\sigma_{34}$ pointwise fixes the cubic curve 
$$
C=\{x_3=x_4\}\cap S_3\subset S_3.
$$ 
One can check that a singular point on $C$ is also a singular point of $X$. 
Then $C$ is necessarily smooth since the six nodes on $X$ are away from $S_3$. 
Therefore $C$ is a genus $1$ curve, and by Proposition~\ref{prop:burnform}, the $G$-action is not linearizable.
\end{proof}

\begin{rema}\label{rema:63Burn}
    The same argument shows that  the following two $G$-actions on $X$ (when they act) are not linearizable:
$$
G=\langle\iota_5,\sigma_{34}\iota_{34}\rangle\simeq C_2^2,\quad\text{ and }\quad G=\langle\sigma_{234},\sigma_{34}\rangle\simeq\fS_3.
$$
In each of the two cases, there is a cubic surface in $X$ fixed by a subgroup $C_2\subset G$ and receiving a $G/C_2$-residual action which fixes an elliptic curve. Therefore, Proposition~\ref{prop:burnform} is also applicable to these cases.
\end{rema}
Using Proposition~\ref{prop:Aut6nod3pl} and Lemma~\ref{lemm:6-nodes-iota-4}, we obtain

\begin{coro}
The action of $\mathrm{Aut}(X)$ on $X$ is linearizable if and only if $\mathrm{Aut}(X)=\langle\iota_5\rangle$.
\end{coro}

We proceed to the actions of other subgroups of the full automorphism groups from Proposition~\ref{prop:Aut6nod3pl}. 

\begin{lemm}
\label{lemm:spec63pl-9nod}
Let $X_a\subset\bP^4$ be a 6-nodal cubic threefold given by 
\begin{align}\label{eq:6nodspeceq}
    a x_1^3 + b(x_2 + x_3 + x_4)x_1^2 + x_1(x_2^2 + x_3^2 + x_4^2 - x_5^2) + x_2x_3x_4=0,
\end{align}
for some $b\not\in\{0,28\}$, and let $G\simeq C_3$ be a group acting on $\mathbb{P}^4$ by 
$$
(x_1,x_2,x_3,x_4,x_5)\mapsto(x_1,x_4,x_3,x_2,x_5).
$$
Then $X$ is $G$-invariant, and the $G$-action on $X_a$ is not stably linearizable for a very general $a\in k$.
\end{lemm}

\begin{proof}
Fixing $b\not\in\{0,28\}$, consider the family $\cX\to \bA^1_k$ whose fiber over $a\in k$ is the cubic given by~\eqref{eq:6nodspeceq}. From Remark~\ref{rema:6-nodes-degeneration}, we know that the generic fiber $X_a:=\cX_a$ is 6-nodal
if $a$ is very general. 
On the other hand, the special fiber $X_{\epsilon}:=\cX_{\epsilon}$ is 9-nodal when $\epsilon=\frac{b^2}{4}-2b-4$ and the additional 3 nodes have trivial stabilizer.

Now set $\epsilon=\frac{b^2}{4}-2b-4$.
Then the $G$-action on planes in $X_{\epsilon}$ has a $G$-orbit of length 3, consisting of three planes
$$
\Pi_4=\{x_2-2x_1=(bx_1 + 2x_3 + 2x_4 + 2x_5)=0\},
$$
$$
\Pi_5=\{x_3-2x_1=(bx_1 + 2x_4 + 2x_2 + 2x_5)=0\},
$$
$$
\Pi_6=\{x_4-2x_1=(bx_1 + 2x_2 + 2x_3 + 2x_5)=0\}.
$$
Moreover, $\Pi_4+\Pi_5+\Pi_6$ is a Cartier divisor on $X_{\epsilon}$, then by Remark~\ref{rema:c3cohomo}, the $G$-action on $X_{\epsilon}$ is not stably linearizable. Applying Proposition~\ref{prop:flat}, we conclude that the $G$-action on $X_a$ is not stably linearizable for a very general $a\in k$.
\end{proof}

\begin{lemm}\label{lemm:spec63pl-10nod}
Let $X_a\subset\bP^4$ be a cubic threefold given by 
\begin{align}\label{eq:6nodspec2eq}
    a x_1^3 + x_1(x_2^2 + x_3^2 + x_4^2 - x_5^2) + x_2x_3x_4=0.
\end{align}
 Consider a group $G\simeq C_2^2$ acting on $\mathbb{P}^4$ via
$$
\iota_{24}: (x_1,x_2,x_3,x_4,x_5)\mapsto(x_1,-x_2,x_3,-x_4,x_5),
$$
$$
\iota_{34}: (x_1,x_2,x_3,x_4,x_5)\mapsto(x_1,x_2,-x_3,-x_4,x_5).
$$
Then $X_a$ is $G$-invariant, and the $G$-action on $X_a$ is not stably linearizable for a very general $a\in k$.
\end{lemm}

\begin{proof}
Consider the family $\cX\to \bA^1_a$ whose fiber over $a\in k$ is the cubic given by \eqref{eq:6nodspec2eq}. From Remark~\ref{rema:6-nodes-degeneration}, we know that $X_a$ is 6-nodal
for a very general $a$. 
On the other hand, the special fiber $X_a'$, when $a'=-4$, is the Segre cubic with 10 nodes. The $G$-action on $X_{-4}$ leaves invariant three planes, namely 
$$
\Pi_i=\{x_1=x_i=0\},\quad i=2,3,4.
$$
The action has an orbit of nodes of length 4 and three orbits of nodes of length 2. By \cite[Section 6]{CTZ}, the $G$-action on $X_{-4}$ does not satisfy {\bf (H1)} and is not stably linearizable. Moreover, the four additional nodes are in one $G$-orbit, and are $BG$-rational singularities. 
Applying Proposition~\ref{prop:flat},
one concludes that a very general member in the family $\cX$ is not $G$-stably linearizable.
\end{proof}

\begin{rema}
\label{rema:C2C4spec6nod}
    The same argument shows that for the same family of cubics, the action on $X_a$ of the group 
$$
G=\langle\iota_5,\iota_{23}\sigma_{34}\rangle\simeq C_2\times C_4
$$
for a very general $a$ is not stably linearizable. The action specializes to the unique $C_2\times C_4$-action on the Segre cubic $X_{-4}$. This action on $X_{-4}$ does not satisfy {\bf (SP)}. The four additional nodes have stabilizer $C_2$, and they are $BG$-rational singularities, see Example~\ref{exam:BG}.
\end{rema}
Let us summarize what is left using the notation of Proposition~\ref{prop:Aut6nod3pl}.
\begin{enumerate}
\item When $\Aut(X)\simeq C_2^2$, we are left with $\langle\iota_5\sigma_{34}\rangle\simeq C_2$,
\item When $\Aut(X)\simeq C_2^3$, we are left with $5$ groups: 
\begin{center}
    \begin{tabular}{|c|c|c|c|}
    \hline
Group&Generators&Orbit of nodes&Orbit of planes  \\\hline
    $C_2$ & $\iota_5\sigma_{34}$&2+2+2&1+2\\\hline
     $C_2$ & $\iota_5\sigma_{34}\iota_{34}$&2+2+2&1+2\\\hline
     $C_2^2$&$\iota_5\sigma_{34},\iota_{34}$&2+4&1+2\\\hline
     $C_2^2$&$\iota_5\sigma_{34},\iota_{34}\sigma_{34}$&2+2+2&1+2\\\hline
     $C_2^2$&$\iota_5\iota_{34},\sigma_{34}$&2+2+2&1+2\\
     \hline
     \end{tabular}
\end{center}

\item When $\Aut(X)\simeq C_2\times\mathfrak{S}_3$ and $a,b_1$ are very general, we are left with 
 $\langle \iota_5\sigma_{34} \rangle \simeq C_2$.

\item When $\Aut(X)\simeq C_2\times\mathfrak{S}_4$, for very general $a$, we are left with: 
\begin{center}
    \begin{tabular}{|c|c|c|c|}
    \hline
Group&Generators&Orbit of nodes&Orbit of planes  \\\hline
    $C_2$ & $\iota_5\sigma_{34}$&2+2+2&1+2\\\hline
    $C_3$&$\sigma_{234}$&3+3&3\\\hline
   
     $C_2^2$&$\iota_5\sigma_{34},\iota_{34}$&2+4&1+2\\\hline
     $C_4$&$\iota_{23}\sigma_{34}$&2+4&1+2\\\hline
     $C_2^2$&$\iota_5\sigma_{34},\iota_{34}\sigma_{34}$&2+2+2&1+2\\\hline
     $C_6$&$\iota_5\sigma_{234}$&6&3\\\hline
     $\fS_3$& $\iota_5\sigma_{34},\sigma_{234}$ &3+3&3\\\hline
     $\fD_4$&$\iota_5\iota_{24},  \iota_{23}\sigma_{34}$&2+4&1+2\\\hline
     \multicolumn{4}{|c|}{\small Specialization in 
Lemma~\ref{lemm:spec63pl-9nod} does not apply to the second row.}\\\hline
     \end{tabular}
\end{center}
\end{enumerate}
    We also note that in each of the remaining cases, the construction in Lemma~\ref{lemm:line} does not apply. In particular, every $G$-invariant line in these cases intersects with the $G$-invariant plane (when it exists) at one point.

\section{Eight nodes}
\label{sect:8}

The 8-nodal cubic threefolds form a two-parameter family, which is denoted by (J13) in \cite{Finkelnbergcubic}. 
Let $X$ be one such cubic, and $G=\mathrm{Aut}(X)$. Then $\mathrm{Cl}(X)=\bZ^4$,
and $X$ contains 5 planes $\Pi_{1},\ldots,\Pi_{5}$ that form a very special configuration \cite{Finkelnbergcubic}.
If $p_1,\ldots,p_8$ are the nodes of $X$ then
$$
\Pi_1\supset\{p_1,p_2,p_6,p_8\},
$$
$$
\Pi_2\supset\{p_1,p_2,p_5,p_7\},
$$
$$
\Pi_3\supset\{p_5,p_6,p_7,p_8\},
$$
$$
\Pi_4\supset\{p_3,p_4,p_5,p_6\},
$$
$$
\Pi_5\supset\{p_3,p_4,p_7,p_8\}.
$$
From this configuration, there are two distinguished sets of nodes 
\begin{equation}
\label{eqn:nodes}
\{p_1,p_2,p_3,p_4\} \quad \text{ and  } \quad  \{p_5,p_6,p_7,p_8\}.
\end{equation}
The planes $\Pi_1,\Pi_2, \Pi_3$ form one tetrahedron (without a face), 
and $\Pi_3,\Pi_4, \Pi_5$ form another one. In particular, $\Pi_3$ is distinguished,
and must be $G$-invariant.

Unprojecting from $\Pi_3$, we obtain a $G$-equivariant birational map 
$$
X\dasharrow X_{2,2} =Q\cap Q'\subset \bP^5
$$
to a singular complete intersection of two quadrics $Q$ and $Q'$ in $\bP^5$ that are singular along lines.
The threefold $X_{2,2}$ has 4 nodes: $Q\cap\mathrm{Sing}(Q')$ and $Q'\cap\mathrm{Sing}(Q)$, 
and $G$ fixes a point in $X_{2,2}$ --- the inverse map $X_{2,2}\dasharrow X$ is just a projection from this point. Thus, one could study the geometry of $X$ using $X_{2,2}$ as in Section~\ref{sect:four}. 

\subsection*{Standard form}
We change the coordinates in $\bP^4$ so that 
$$
\Pi_3=\{x_4=x_5=0\}
$$
and
$$
p_1=[0:0:0:1:0],\quad p_3=[0:0:0:0:1];
$$
this is possible since the line through $p_1$ and $p_3$ is disjoint from $\Pi_3$.
Changing the variables $x_1,x_2,x_3$, we may assume that
$$
p_5=[1:1:1:0:0],\quad p_6=[-1:1:1:0:0],
$$
$$
p_7=[1:-1:1:0:0],\quad p_8=[1:1:-1:0:0].
$$
This specifies the equations of the planes:
\begin{align*} 
&\Pi_1=\{x_1+x_3=x_5=0\}, \\
&\Pi_2=\{x_1-x_3=x_5=0\}, \\
&\Pi_3=\{x_4=x_5=0\},  \\
&\Pi_4=\{x_2-x_3=x_4=0\}, \\
&\Pi_5=\{x_2+x_3=x_4=0\}.
\end{align*}
A cubic threefold containing $\Pi_1,\ldots,\Pi_5$ and singular at $p_1,p_3,p_5,\ldots,p_8$ has the form 
$$
(a_{22}x_1+a_{12}x_2+a_6x_3)x_4x_5+a_9(x_3^2-x_1^2)x_4+a_8(x_3^2-x_2^2)x_5=0,
$$
for some $a_6,a_8,a_9,a_{12},a_{22}$. Since $X$ is 8-nodal, we have 
$$
a_{8}, a_{9}, a_{12}, a_{22}\ne 0.
$$
Scaling coordinates, we may assume that $a_{8}=a_{9}=a_{12}=1$, and we let $a_{22}=a$ and $a_{6}=b$. 
Thus, $X=X_{a,b}$ is given by
\begin{equation}
\label{8-nodes-form}    
(ax_1+x_2+bx_3)x_4x_5+x_4(x_3^2-x_1^2)+x_5(x_3^2-x_2^2)=0,
\end{equation}
for parameters $a,b$, where $a\ne 0$. Notice that $a$ and $b$ are defined up to $\pm1$. For very general $a$ and $b$,  \eqref{8-nodes-form} defines an $8$-nodal cubic with 
nodes at $p_1,\ldots,p_8$, where $p_1,p_3,p_5,\ldots,p_8$ are described above, and 
$$
p_2=[0:1:0:1:0],\quad\text{and} \quad p_4=[a:0:0:0:1].
$$
For special parameters $a$ and $b$, \eqref{8-nodes-form} defines a cubic with additional singularities, 
for instance, the Segre cubic, 
when $b=0$ and $a=1$.

\subsection*{Automorphisms}

Recall that $G=\mathrm{Aut}(X)$ and $\Pi_3$ is $G$-invariant. 
Let $l_{12}$ be the line passing through $p_1$ and $p_2$,
and $l_{34}$ the line through $p_3$ and $p_4$. Then the curve $l_{12}+l_{34}$ is $G$-invariant.
On the other hand, we have a group homomorphism 
$$
\phi\colon G\to\PGL_3(k),
$$ 
arising from the action of $G$ on $\Pi_3$. Since $\phi(G)$ permutes the points $p_5,p_6,p_7,p_8$,
we see that $\phi(G)\subseteq \mathfrak{S}_4\subset\PGL_3(k)$,  permuting the coordinates $x_1,x_2,x_3$ and changing signs of these variables.
Moreover, the set $(l_{12}+l_{34})\cap\Pi_3$ is $\phi(G)$-invariant, which implies that $\phi(G)$
is contained in $\mathfrak{D}_4\subset\mathfrak{S}_4$ generated by:
\begin{align*}
(x_1, x_2, x_3)&\mapsto (-x_2,x_1,x_3),\\
(x_1,x_2,x_3)&\mapsto (-x_1,x_2,x_3).
\end{align*}

\begin{lemm}
\label{lemma:8-nodal-kernel}
The kernel $\ker(\phi)$ of $\phi$ is nontrivial if and only if $b=0$.
Moreover, if $b=0$, then $\ker(\phi)\simeq C_2$, generated by 
$$
(x_1,x_2,x_3,x_4,x_5)\mapsto (x_1-ax_5,x_2-x_4,x_3,-x_4,-x_5).
$$
\end{lemm}

\begin{proof}
An element $\tau\in \ker(\phi)$ preserves $\Pi_3$, the points $p_6,p_7,p_8,p_9$,
and each line $l_{12}$ and $l_{34}$. 
Moreover, since $\tau$ leaves the subsets $\{p_1,p_2\}$ $\{p_3,p_4\}$ invariant, we have the following three possibilities:
\begin{enumerate}
\item $\tau(p_1)=\tau(p_1)$, $\tau(p_2)=\tau(p_2)$, $\tau(p_3)=\tau(p_4)$, $\tau(p_4)=\tau(p_3)$,
\item $\tau(p_1)=\tau(p_2)$, $\tau(p_2)=\tau(p_1)$, $\tau(p_3)=\tau(p_3)$, $\tau(p_4)=\tau(p_4)$,
\item $\tau(p_1)=\tau(p_2)$, $\tau(p_2)=\tau(p_1)$, $\tau(p_3)=\tau(p_4)$, $\tau(p_4)=\tau(p_3)$.
\end{enumerate}
These impose linear conditions on $\tau$. Solving them, we see that $\tau$ is one of the following linear transformations:
\begin{enumerate}
\item $(x_1,x_2,x_3,x_4,x_5)\mapsto(x_1,x_2-x_4,x_3,-x_4,x_5)$,
\item $(x_1,x_2,x_3,x_4,x_5)\mapsto(-ax_5+x_1,x_2,x_3,x_4,-x_5)$
\item $(x_1,x_2,x_3,x_4,x_5)\mapsto(x_1-ax_5,x_2-x_4,x_3,-x_4,-x_5)$.
\end{enumerate}
However, \eqref{8-nodes-form} must be preserved by $\tau$,
which implies that only the third case is possible, and only in the case when $b=0$.
\end{proof}

We are ready to classify all possibilities for $G=\mathrm{Aut}(X)$.

\begin{prop}
\label{prop:8-nodes-Aut}
Let $X\subset \bP^4$ be an 8-nodal cubic threefold given by 
$$
(ax_1+x_2+bx_3)x_4x_5+x_4(x_3^2-x_1^2)+x_5(x_3^2-x_2^2)=0
$$
and $G=\Aut(X)$. Then one of the following holds:
\begin{itemize}
\item $b\ne 0$ and $a\ne \pm 1$, $G\simeq C_2^2=\langle\iota_1,\iota_2\rangle$, generated by 
\begin{align*}
    \iota_1: (x_1,x_2,x_3,x_4,x_5)&\mapsto (a x_5-x_1,x_2,x_3,x_4,x_5),\\
\iota_2: (x_1,x_2,x_3,x_4,x_5)&\mapsto (x_1,x_4-x_2,x_3,x_4,x_5);
\end{align*}
\item $b\ne 0$ and $a=\pm 1$, $G\simeq \mathfrak{D}_4$, generated by $\iota_2$ and 
$$
\sigma_1:(x_1,x_2,x_3,x_4,x_5)\mapsto(\pm x_2,x_5\mp x_1,x_3,x_5,x_4);
$$

\item $b=0$, $a\ne \pm i$, $G\simeq C_2^3$, generated by $\iota_1$, $\iota_2$ and 
$$
\tau: (x_1,x_2,x_3,x_4,x_5)\mapsto (x_1-ax_5,x_2-x_4,x_3,-x_4,-x_5);
$$

\item $b=0$, $a=\pm i$, $G\simeq C_2.\fD_4\simeq C_2^2\rtimes C_4$, generated by $\tau,\iota_2$ and
\begin{align*}
\sigma_2:(x_1,x_2,x_3,x_4,x_5)\mapsto (\mp x_2,\pm x_1,x_3,ix_5,ix_4).
\end{align*}
\end{itemize}
\end{prop}

\begin{proof}
Observe that $G$ always contains $\iota_1$ and $\iota_2$, and $\langle\iota_1,\iota_2\rangle\cong C_2^2$,
which shows that $\phi(G)$ is at least $C_2^2$.
Moreover, if $b=0$, then it follows from Lemma~\ref{lemma:8-nodal-kernel}
that $G$ also contains the involution $\tau$, 
so together the involutions $\iota_1$, $\iota_2$, $\tau$ generate a subgroup $C_2^3$ in this case.
If $\phi(G)\simeq C_2^2$, this gives us all possibilities for $G$.
To complete the proof, we have to find all $a$ and $b$ such that $\phi(G)\simeq\mathfrak{D}_4$,
and describe $G$ in these cases. This can be done explicitly.

If $\phi(G)\simeq\mathfrak{D}_4$, then $G$ contains a $\sigma$ such that $\phi(\sigma)$ is given by
$$
(x_1,x_2,x_3)\mapsto (x_2,-x_1,x_3),
$$
which implies that $\sigma$ is given by the matrix
$$
\begin{pmatrix}
0& -1& 0& 0 & 0\\
1 & 0 & 0 & 0 & s_{25}\\
0 & 0& 1& 0& 0\\
s_{14} & 0& 0& 0& s_{54}\\
0& s_{25}& 0& s_{45}& 0
\end{pmatrix}
$$
for some $s_{14}$, $s_{25}$, $s_{45}$, $s_{54}$. Since $\sigma$ preserves \eqref{8-nodes-form},
we obtain constraints on these entries, which result in the following possibilities:
\begin{enumerate}
\item $a=1$, $s_{14}=0$, $s_{25}=1$, $s_{45}=1$, $s_{54}=1$;
\item $a=-1$, $s_{14}=-1$, $s_{25}=0$, $s_{45}=1$, $s_{54}=1$; 
\item $b=0$, $a=-i$, $s_{14}=0$, $s_{25}=0$, $s_{45}=i$, $s_{54}=i$;
\item $b=0$, $a=i$, $s_{14}=-1$, $s_{25}=i$, $s_{45}=i$, $s_{54}=i$. 
\end{enumerate} 
Using them, we obtain all possibilities for $G$ listed above.
\end{proof}

\subsection*{Cohomology}
Let $\widetilde{X}\to X$ be the standard resolution,
let $E_1,\ldots,E_{8}$ be exceptional divisors over $p_1,\ldots, p_8$, and let $\widetilde{\Pi}_1,\ldots,\widetilde{\Pi}_5$ be the strict transforms of the planes $\Pi_1,\ldots,\Pi_5$ on  $\widetilde{X}$, respectively.
Then $\Pic(\widetilde{X})$ is generated by $E_1,\ldots,E_{8},\widetilde{\Pi}_1,\ldots,\widetilde{\Pi}_5$.
These are subject to the relation
$$
\widetilde{\Pi}_1+\widetilde{\Pi}_2-\widetilde{\Pi}_4-\widetilde{\Pi}_5=E_1+E_2-E_3-E_4.
$$
Notice that this presentation of the lattice $\Pic(\widetilde X)$ is independent of the equation of the cubic threefold $X$. To compute the {\bf (H1)}-obstruction on $\Pic(\widetilde X)$, for generality, we may work with the maximal symmetry group appearing in Proposition~\ref{prop:8-nodes-Aut}. Let $G=C_2^2\rtimes C_4$, as defined above. Then $G$ acts on nodes via permutation of indices 
\begin{align*} 
    &\tau:(1, 2)(3, 4),
    \\
    &\iota_1:(3, 4)(5, 6)(7, 8),\\
    &\iota_2:(1, 2)(5, 7)(6, 8),\\
    &\sigma_2:(1, 3)(2, 4)(5, 7, 8, 6).
\end{align*}
There is a unique (conjugacy class of) $C_2=\langle\iota_1\iota_2\rangle$ contributing to 
$$
\rH^1(C_2,\Pic(\widetilde X))=\bZ/2.
$$
   Indeed, this $C_2$ acts on nodes via the permutation of indices
   $$
  (1,2)(3,4)(5,8)(6,7).
   $$
   Under the basis 
   $$
   \widetilde{\Pi}_1-\widetilde{\Pi}_5-E_2+E_4,\, \widetilde{\Pi}_1,\ldots,\widetilde{\Pi}_5, E_3,\ldots,E_8,
   $$
   one can see that 
   $C_2$ acts on $\widetilde{\Pi}_1-\widetilde{\Pi}_5-E_2+E_4$ by $-1$, and on the rest as a permutation module. 

Note that this $C_2$ is contained in $\Aut(X)=C_2^2$ for generic $X$ in Proposition~\ref{prop:8-nodes-Aut}, i.e., when $b\ne 0$ and $a\ne \pm1$.

\subsection*{Linearization}
Let $X$ be an 8-nodal cubic threefold. 
The classification in Proposition~\ref{prop:8-nodes-Aut} implies:
\begin{itemize}
    \item $b\ne 0$ and $a\ne \pm1$: a subgroup of $\Aut(X)\simeq C_2^2$ is linearizable if and only if it fixes a singular point; otherwise, it fails {\bf (H1)}. 
    \item $b\ne 0$ and $a=\pm1$: a subgroup of $\Aut(X)\simeq \fD_4$ is linearizable if and only if it fixes a singular point; otherwise, it fails {\bf (H1)}. 
    \item $b=0$ and $a\ne\pm i$:  Excluding subgroups failing {\bf(H1)} or with a fixed singular point, we are left with the following (classes of) subgroups
   \begin{enumerate}
       \item $C_2^2=\langle(3, 4)(5, 7)(6, 8),
    (1, 2)(3, 4)\rangle$, acting via
    \begin{align*}
         (x_1,x_2,x_3,x_4,x_5)&\mapsto (x_1+i x_5,-x_2,x_3,-x_4,-x_5),\\(x_1,x_2,x_3,x_4,x_5)&\mapsto (x_1+i x_5,x_2-x_4,x_3,-x_4,-x_5).
    \end{align*}
    \item $C_2^2=\langle\iota_1,\tau\rangle = \langle(3, 4)(5, 6)(7, 8),
    (1, 2)(3, 4)\rangle$, acting via
    \begin{align*}
           (x_1,x_2,x_3,x_4,x_5)&\mapsto (-x_1-i x_5,x_2,x_3,x_4,x_5),\\
     (x_1,x_2,x_3,x_4,x_5)&\mapsto (x_1+i x_5,x_2-x_4,x_3,-x_4,-x_5).
    \end{align*}  
   \end{enumerate}
    
     \item $b= 0$ and $a=\pm i$: Excluding groups with an {\bf(H1)}-obstruction or with a fixed singular point, we are left with 
      \begin{enumerate}
  
       \item $C_2^2=\langle    (3, 4)(5, 6)(7, 8),(1, 2)(3, 4)\rangle$, acting via 
       \begin{align*}
       (x_1,x_2,x_3,x_4,x_5)&\mapsto (-x_1-i x_5,x_2,x_3,x_4,x_5), \\
       (x_1,x_2,x_3,x_4,x_5)&\mapsto (x_1+i x_5, x_2-x_4,x_3,-x_4,-x_5),
       \end{align*}
       \item  $C_4=\langle  (1, 3)(2, 4)(5, 7, 8, 6)\rangle$, acting via 
       \begin{align*}
   (x_1,x_2,x_3,x_4,x_5)&\mapsto (x_2,-x_1,x_3,i x_5,i x_4).   
   \end{align*}
    \end{enumerate}
\end{itemize}

We turn to linearization constructions for subgroups unobstructed by cohomology and not fixing singular points. Consider the maximal symmetry group, in the case $b=0$ and $a=\pm i$. We have two unobstructed cases:
\begin{itemize}
    \item $G=C_2^2=\langle(3, 4)(5, 6)(7, 8),(1, 2)(3, 4)\rangle$. The group $G$ swaps the planes $\Pi_1$ and $\Pi_2$ and preserves $\Pi_4$ and $\Pi_5$. The line passing through $p_1$ and $p_2$ is a $G$-invariant line disjoint from $\Pi_4.$ Then the $G$-action on $X$ is linearizable by Lemma~\ref{lemm:line}.

    \item $G=C_4=\langle  (1, 3)(2, 4)(5, 7, 8, 6)\rangle$. In this case, $G$ swaps the planes $\Pi_1,\Pi_2$ and swaps $\Pi_4,\Pi_5$. But $\Pi_3$ is $G$-invariant. The line passing through $p_1$ and $p_3$ is a $G$-invariant line disjoint from $\Pi_3$. The $G$-action on $X$ is linearizable by Lemma~\ref{lemm:line}.
    
\end{itemize}
Note that the constructions above only depend on the group actions on singular points and planes. One can also establish the same linearization results for the two unobstructed $C_2^2$ in the case when $b=0$ and $a\ne \pm i$. We summarize this section by:
\begin{coro}\label{coro:8-nodalsummary}
    Let $X$ be an $8$-nodal cubic threefold and $G\subseteq\Aut(X)$. The $G$-action on $X$ is linearizable if and only if it satisfies {\bf (H1)}, if and only if $G$ does not contain a subgroup isomorphic to $C_2$ which does not fix any nodes of $X$; in particular, if it is not linearizable then it is not stably linearizable.   
\end{coro}
\section{Nine nodes}
\label{sect:nine}

\subsection*{Standard form}

We follow \cite{ShB}: 9-nodal
cubic threefolds $X_a$ are given 
in $\bP^5$
by equations 
$$
x_1 x_2 x_3-x_4x_5x_6 =a(x_1+x_2+x_3)+x_4+x_5+x_6=0,\quad a^3\ne0, -1. 
$$
Their automorphisms depend on the parameter $a$ as follows:
$$
\Aut(X_a)=\begin{cases}
    \fS_3^2& \text{when } a^3\ne1 ,\\
     \fS_3^2\rtimes C_2& \text{otherwise}.
\end{cases}
$$
These groups act via $\fS_3$-permutations of two sets of coordinates: $x_1, x_2, x_3$, and $x_4,x_5,x_6$. When $a=1$, the additional $C_2$ switches 
$
x_{i}\leftrightarrow x_{3+i},\, i=1,2,3.
$
In both cases, the $9$ nodes are given by 
$$
\{x_{i_1}=x_{i_2}=x_{j_1}=x_{j_2}=0,\quad x_{j_3}+ax_{i_3}=0\},
$$
where 
$$
i_1\ne i_2\ne i_3\in\{1,2,3\},\quad
 j_1\ne j_2\ne j_3\in\{4,5,6\}.
$$
There are also $9$ distinguished planes, given by 
$$
\Pi_{i,j}=\{x_i=x_{3+j}=0\}\cap X, \quad i,j\in\{1,2,3\}.
$$ 
The $G$-action on $X_a$ fixes a singular point if and only if $G$ is a $2$-group. 
\subsection*{Fixed point obstruction}
Let $G=C_3^2$ be the group generated by 
$$
(x_1,x_2,x_3,x_4,x_5,x_6)\mapsto(x_3,x_1,x_2,x_4,x_5,x_6), 
$$
$$
(x_1,x_2,x_3,x_4,x_5,x_6)\mapsto(x_1,x_2,x_3,x_6,x_4,x_5).
$$
Then $X_a^G=\emptyset$, for all $a$ such that $a^3\ne 0,  -1$. By Lemma~\ref{lemm:fixedptabelian}, the $G$-action on $X$ is not linearizable. 
\subsection*{Cohomology}
Let $\tilde X_a\to X_a$ be the blowup of $X_a$ at $9$ nodes. Then $\Pic(\tilde X_a)$ is generated by  $E_i, i=1,\ldots, 9$, the exceptional divisors over the $9$ nodes, the pullbacks
$\widetilde{\Pi}_{i,j}$
of $\Pi_{i,j}$, and $H$, the pullback of the hyperplane section. They are subject to relations
\begin{align*}
H= \widetilde{\Pi}_{1,1}+\widetilde{\Pi}_{1,2}+\widetilde{\Pi}_{1,3}+E_2+E_3+E_5+E_6+E_8+E_9,\\
H= \widetilde{\Pi}_{1,1}+\widetilde{\Pi}_{2,1}+\widetilde{\Pi}_{3,1}+E_1+E_2+E_3+E_4+E_5+E_6,\\
H= \widetilde{\Pi}_{1,2}+\widetilde{\Pi}_{2,2}+\widetilde{\Pi}_{3,2}+E_1+E_2+E_3+E_7+E_8+E_9,\\
H= \widetilde{\Pi}_{1,3}+\widetilde{\Pi}_{2,3}+\widetilde{\Pi}_{3,3}+E_4+E_5+E_6+E_7+E_8+E_9,\\
H= \widetilde{\Pi}_{2,1}+\widetilde{\Pi}_{2,2}+\widetilde{\Pi}_{2,3}+E_1+E_3+E_4+E_6+E_7+E_9,\\
H= \widetilde{\Pi}_{3,1}+\widetilde{\Pi}_{3,2}+\widetilde{\Pi}_{3,3}+E_1+E_2+E_4+E_5+E_7+E_8. 
\end{align*}
When $a^3\ne 1$, computation yields two minimal classes of groups contributing to nonvanishing cohomology:
$$
\rH^1(G',\Pic(\tilde X_a))=\bZ/3,
$$
for $G'=C_3=\langle(1,2,3)\rangle$ or $\langle(4,5,6)\rangle$, realized as permutations of indices of the coordinates. When $a=1$, these two classes of $C_3$ are conjugate in $\Aut(X_1)$, and thus we found a unique class of groups contributing to nonvanishing cohomology:
$$
\rH^1(G',\Pic(\tilde X_a))=\bZ/3,
$$
for $G'=C_3=\langle(1,2,3)\rangle$. Any subgroup of $\Aut(X_a)$ containing those classes has {\bf (H1)}-obstructions to stable linearizability.

\begin{rema}\label{rema:c3cohomo}
    One can characterize geometrically the $C_3$-action contributing to {\bf (H1)}-obstructions as follows: let $X_a$ be a 9-nodal cubic threefold, and $G=C_3\subseteq \Aut(X_a)$. Then the $G$-action on $X_a$ does not satisfy {\bf(H1)} if and only if there exists a $G$-orbit of planes of length 3, which forms a Cartier divisor.
\end{rema}

Excluding $G\subseteq \Aut(X_a)$ with {\bf (H1)}-obstruction or with $G$-fixed singular points, one is left with
\begin{itemize}
    \item When $a^3= 1$, the unobstructed groups are
    \begin{align}\label{eq:9nodalremain}
        \fD_6, \fS_3, \fS_3', C_6, C_3,
    \end{align}
\noindent
where $\fD_6$ acts on $\{x_1, x_2, x_3\}$ and $\{x_4, x_5, x_6\}$ via diagonal $\fS_3$ permutations and $C_2$ swapping them, i.e., $x_i\leftrightarrow x_{3+i}, i=1,2,3$. The other groups are all subgroups of $\fD_6$. 
\item 
When $a^3\ne 1$, we are left with 
$$
\fS_3\quad\text{ and }\quad C_3,
$$
where $\fS_3$ is the diagonal permutation and $C_3$ its subgroup. 
\end{itemize}
Next, we show that the actions of these unobstructed groups on $X_a$ are equivariantly birational to actions on a smooth quadric threefold. In particular, the actions of cyclic groups $C_3$ and $C_6$ are linearizable.

\subsection*{Linearization}
Consider the family of degree $(1,1)$ divisors in $(\bP^2)^2$ 
$$
W_{b}\subset \bP^2_{t_1,t_2,t_3}\times\bP^2_{z_1,z_2,z_3},\quad b\in\bC\setminus\{0,-1,\zeta_3, \zeta_3^2\},
$$
given by 
$$
(-t_1z_2+t_2z_1)+b(-t_1z_2+t_3z_3)=0,
$$
with a $G=\fS_3$-action generated by 
$$
\iota:t_1\leftrightarrow z_2, \quad t_2\leftrightarrow z_1, \quad t_3\leftrightarrow z_3
$$
and
$$
\sigma: t_1\mapsto \zeta t_1, \quad t_2\mapsto \zeta^2t_2, \quad z_1\mapsto \zeta z_1, \quad z_2\mapsto \zeta^2z_2.
$$
Let $p_1,p_2,p_3\in W_b$ be the points
$$
[1:1:1]\times[1:1:1],\quad [\zeta:\zeta^2:1]\times[\zeta:\zeta^2:1], \quad [\zeta^2:\zeta:1]\times[\zeta^2:\zeta:1],
$$
where $\zeta=e^{\frac{2\pi i}{3}}$. Note that $\{p_1, p_2, p_3\}$ forms one $G$-orbit. The linear system 
$$
|H-p_1-p_2-p_3|
$$
consisting of hyperplanes on $\bP^4$ containing points $p_1, p_2$ and $p_3$
has projective dimension $4$. Under a chosen basis, it gives a birational map to a 9-nodal cubic hypersurface $Y_{b}\subset \bP^4$, with equation
\begin{multline*}
y_1y_2y_3 + y_1y_5^2 - y_2^2y_4  + y_2y_4^2 - 
    y_1^2y_5-\frac{b}{b + 1}y_1y_3y_4 +
    \\+\frac{b}{(b+1)^2}y_3^3- \frac{b}{b + 1}y_2y_3y_5 -\frac{1}{b + 1}y_3y_4y_5 =0    
\end{multline*}
Up to a change of variables by
$$
\begin{pmatrix}
    y_1&y_2&y_3&y_4&y_5
\end{pmatrix}\cdot\begin{pmatrix}
1&   \zeta&   \zeta^2&   \frac{-\zeta b +\zeta^2}{b-\zeta^2}&   \frac{-\zeta^2 b + 1}{b -\zeta^2}\\
1&   \zeta^2&   \zeta&   \frac{-b + \zeta}{b -\zeta^2}&   \frac{-\zeta^2 b + 1}{b - \zeta^2}\\
\frac{\zeta b-\zeta^2}{b + 1}&   \frac{\zeta b -\zeta^2}{b + 1}&   \frac{\zeta b -\zeta^2}{b + 1}&   \frac{-\zeta b+\zeta^2}{b + 1}&   \frac{-\zeta b +\zeta^2}{b + 1}\\
\zeta&   1&   \zeta^2&   \frac{-\zeta^2 b + 1}{b-\zeta^2}&   \frac{-\zeta b +\zeta^2}{b - \zeta^2}\\
\zeta&   \zeta^2&   1&   \frac{-b + \zeta}{b -\zeta^2}&   \frac{-\zeta b+\zeta^2}{b-\zeta^2}
\end{pmatrix},
$$
$Y_{b}$ is $G$-isomorphic to  
$$
\{y_1y_2y_3 + \lambda_b y_4y_5(y_1+y_2+y_3 +y_4 + y_5)=0\}\subset\bP^4,\quad\lambda_b=-\left(\frac{b-\zeta^2}{b-\zeta}\right)^3,
$$
i.e.,
$$
X_{a}=\{x_1x_2x_3-x_4x_5x_6=a(x_1+x_2+x_3)+x_4+x_5+x_6=0\}\subset\bP^5,
$$
where 
$$
a=-\frac{b-\zeta^2}{b-\zeta}.
$$
The $G$-action on $X_{a}$ is given by the diagonal permutation of coordinates 
$x_1,x_2,x_3$ and $x_4,x_5,x_6$. When $b\ne 0, -1,\zeta, \zeta^2$, i.e., $a^3\ne-1,0$, one sees that $W_{b}$ (and thus $X_a$) is $G$-equivariantly birational to
\begin{align}\label{eqn:9nodquadric}
    Q_{b}=\{(b+1)t_1z_2-t_2z_1+bz^2=0\}\subset\bP^4_{t_1,t_2,z_1,z_2,z},
\end{align}
realized as the equivariant compactification of the affine chart of $W_b$ given by 
$$
\{t_3\ne0, z_3\ne 0\}\subset W_b,
$$
with the natural action of $\iota$ and $\sigma$ (acting trivially on $z$).

When $a^3=1$, i.e. $b=1,-2$ or $-\frac12$, there is extra symmetry on $W_{b}$ and $Q_{b}$. For example, when $b=-\frac12$, $W_{b}$ and $Q_{b}$ are invariant under the additional involution 
$$
\tau: t_1\leftrightarrow t_2,\quad z_1\leftrightarrow z_2.
$$
The group $G'=\langle \iota,\sigma,\tau\rangle$ is isomorphic to $\fD_6$. The corresponding $G'$-action on $X_{1}$ is generated by the diagonal 
$\fS_3$-permutation and by swapping two sets of coordinates 
$
\{x_1,x_2,x_3\}
$
and
$
\{x_4,x_5,x_6\}.
$ 
We do not know whether or not this action is linearizable.
\begin{coro}
    Let $X_a$ be a 9-nodal cubic threefold as above. The $C_3$-action on $X_a$ via permutation of coordinates 
    $$
(x_1,x_2,x_3,x_4,x_5,x_6)\mapsto(x_3,x_1,x_2,x_6,x_4,x_5)
    $$
    is linearizable for all $a^3\ne 0,1$. When $a=1$, the $C_6$-action on $X_1$ via 
    $$
(x_1,x_2,x_3,x_4,x_5,x_6)\mapsto(x_6,x_4,x_5,x_3,x_1,x_2).
    $$
    is linearizable.
\end{coro}
\begin{proof}
    By constructions above, these actions are equivariantly birational to actions on the corresponding smooth quadric $Q_{b}$, necessarily with fixed points. Projection from a fixed point on $Q_{b}$ gives linearizations.  
\end{proof}

\subsection*{Birational rigidity}
Let $X$ be the 9-nodal cubic threefold in $\bP^4\subset \bP^5$ given by
$$
x_1 x_2 x_3-x_4x_5x_6 = x_1+x_2+x_3+x_4+x_5+x_6=0, 
$$
and let $G=\mathrm{Aut}(X)=\fS_3^2\rtimes C_2$. We claim that $X$ is $G$-birationally super-rigid. 
We start with several preliminary results.

\begin{lemm}
\label{lem:orbits}
If $\Sigma$ is a $G$-orbit in $X$ of length $<12$, then $|\Sigma|\in\{6,9\}$.
\end{lemm}

\begin{proof}
Left to the reader.
\end{proof}

Set 
$$
S=\{x_1+x_2+x_3-x_4-x_5-x_6\}\cap X.
$$
Then $S$ is the unique $G$-invariant hyperplane section of $X$.
Moreover, the cubic surface $S$ is smooth, and $G$ acts faithfully on it.
This implies that $S$ is isomorphic to the Fermat cubic surface \cite{DolgachevBook,DI}.
Consider
$$
\alpha_G(S)=\mathrm{sup}\left\{\lambda\in\mathbb{Q}\ \left|\ \aligned
&\text{the pair}\ \left(S, \lambda D\right)\ \text{is log canonical for every}\\
&\text{effective $G$-invariant $\mathbb{Q}$-divisor}\ D\sim_{\mathbb{Q}} -K_{S}\\
\endaligned\right.\right\}.
$$

\begin{lemm}[{cf. \cite{CheltsovGAFA,CheltsovWilson}}]
\label{lem:alpha}
One has $\alpha_G(S)=2$.
\end{lemm}

\begin{proof}
One can check that $\mathrm{Pic}^G(S)=\mathbb{Z}[-K_S]$.
Note that the group $G$ is missed in \cite[Theorem 6.14]{DI}.
Note also that the linear system $|-K_S|$ does not contain $G$-invariant divisors,
but $|-2K_S|$ contains a $G$-invariant divisor. 
Applying Lemma~\ref{lem:orbits} and \cite[Lemma 5.1]{CheltsovGAFA}, we obtain $\alpha_G(S)=2$.
\end{proof}

\begin{lemm}
\label{lem:curves}
Let $C\subset X$ be a $G$-irreducible curve of degree $<12$. Then $C\subset S$.
\end{lemm}

\begin{proof}
Assume $C\not\subset S$. Set $d=\mathrm{deg}(C)$.
Intersecting $C$ with $S$,
we immediately obtain $d=6$ or $d=9$, by Lemma~\ref{lem:orbits}.
Moreover, we also see that 
$$
|S\cap C|=d,
$$
so that $C$ is smooth at every point in $S\cap C$, and $S$ intersects $C$ transversally.
Hence, if $C$ is irreducible and $C\not\subset S$, then $G$ acts faithfully on $C$, 
which implies that the stabilizer of any point in $C\cap S$ is cyclic, which is impossible,
since $G$ does not have cyclic subgroups of index $6$ and $9$.

To complete the proof, we may assume that $C$ is reducible.
Let $r$ be the number of its irreducible components.
Write 
$$
C=C_1+\cdots +C_r,
$$
where each $C_i$ is an irreducible component of $C$. Set $d_1=\mathrm{deg}(C_1)$, and let $H_1$ be the stabilizer of the component $C_1$ in $G$.
Then $d=d_1r$, and, since $G$ does not have subgroups of index $3$, we have one of the following cases:
\begin{enumerate}
\item[(1)] $d=9$, $r=9$, $d_1=1$, $H_1\simeq\mathfrak{D}_4$,
\item[(2)] $d=6$, $r=6$, $d_1=1$, $H_1\simeq\mathfrak{D}_6$,
\item[(3)] $d=6$, $r=2$, $d_1=3$, $H_1\simeq\mathfrak{S}_3^2$ or $H_1\simeq C_3^2\rtimes C_4$.
\end{enumerate}
We exclude these cases one by one. In Case (1), there is a unique class of subgroups isomorphic to $\fD_4$, and the $\fD_4$-linear representation decomposes as
$$
\bP(\rI\oplus\chi^2\oplus V),
$$
i.e., a sum of the trivial representation $\rI$, two copies of a nontrivial 1-dimensional subrepresentation $\chi$, and an irreducible 2-dimensional representation $V$. By Schur's lemma, $V$ is the unique irreducible 2-dimensional representation in the ambient space of $X$. The projectivization $\bP(V)$ defines an invariant line contained in $S$,
$$
l=\{x_1+x_2=x_3=x_5=0\}\subset S.
$$
The plane $\bP(\rI\oplus\chi^2)\subset \bP^4$ intersects $X$ along an irreducible cubic curve, and contains no line. It follows that $l$ is the only $H_1$-invariant line in $X$ and thus $C\subset S$.

In Case (2), there are two classes of subgroups isomorphic to $\fD_6$. In one class, the $\fD_6$-linear representation is 
$$
\bP(\rI^2\oplus\chi\oplus V),
$$
i.e., the sum of two copies of the trivial 1-dimensional representation $\rI$, a nontrivial 1-dimensional representation $\chi$ and an irreducible 2-dimensional representation $V$. Again, $V$ is the unique irreducible 2-dimensional representation. But in this case, the line $\bP(V)$ is not contained in $X$. And the plane $\bP(\rI^2\oplus\chi)$ intersects $X$ along an irreducible cubic curve. Therefore, there is no $H_1$-invariant line. The other class of $\fD_6$ decomposes as representation as
$$
\bP(\chi\oplus V_1\oplus V_2),
$$
i.e., the sum of a nontrivial 1-dimensional representation $\chi$ and two nonisomorphic irreducible 2-dimensional representations $V_1$ and $V_2$. Here, $\bP(V_1)$ defines a line contained in $S$:
$$
l=\{x_1-x_4=x_2-x_4=x_3+x_4+x_5=0\}\subset S,
$$
while $\bP(V_2)$ is not contained in $X$. In this case, we also have $C\subset S$.

In Case (3), suppose that $d=6$, $r=2$, $d_1=3$. 
Then the hyperplane 
$$
\{x_1+x_2+x_3+x_4+x_5+x_6=0\}
$$ 
is the unique $H_1$-invariant hyperplane,
and every $H_1$-invariant plane in $\mathbb{P}^4$ is contained in this hyperplane.
This implies that $C\subset S$.
\end{proof}


\begin{theo}
The Fano threefold $X$ is $G$-birationally super-rigid.
\end{theo}

\begin{proof}
Suppose that $X$ is not  $G$-birationally super-rigid.
Then it follows from the equivariant version of the Noether--Fano inequality \cite{CheltsovShramov} that 
there exists a $G$-invariant non-empty mobile linear system $\mathcal{M}$ on $X$ such that
the singularities of the log pair $(X,\lambda\mathcal{M})$ are not canonical for $\lambda\in\mathbb{Q}_{>0}$
such that $\lambda\mathcal{M}\sim_{\mathbb{Q}} -K_X$. We seek a contradiction.

First, we claim that the singularities of the log pair $(X,\lambda\mathcal{M})$ 
are canonical away from finitely many points. 
Indeed, if this is not the case, then there exists a $G$-irreducible curve $C\subset X$
such that 
$$
\mathrm{mult}_{C}\big(\mathcal{M}\big)>\frac{1}{\lambda},
$$
which immediately implies that the degree of $C$ is less than $12$, which implies that 
$C\subset S$ by Lemma~\ref{lem:curves}, so that the log pair $(S,\lambda\mathcal{M}\vert_{S})$ is not log canonical,
which contradicts Lemma~\ref{lem:alpha}, since $\lambda\mathcal{M}\vert_{S}\sim_{\mathbb{Q}}-2K_S$.

Next, we claim that the log pair $(X,\lambda\mathcal{M})$ is canonical at every singular point of $X$.
Indeed, let $f\colon\widetilde{X}\to X$ be the blow up of all singular points of $X$,
let $E_1,\ldots,E_9$ be the $f$-exceptional surfaces, let $\widetilde{\mathcal{M}}$ be the strict transform on $\widetilde{X}$
of the~linear system $\mathcal{M}$, and let $\widetilde{M}$ be a general surface in  $\widetilde{\mathcal{M}}$.
Then, since $\mathrm{Sing}(X)$ forms one $G$-orbit, we have
$$
\lambda\widetilde{M}\sim_{\mathbb{Q}}f^*\big(-K_{X}\big)-a\sum_{i=1}^{9}E_i,
$$
for some  integer $a>1$, by \cite[Theorem~1.7.20]{CheltsovUMN} or \cite[Theorem~3.10]{Co00}. 
Recall that $X$ contains $9$ planes 
$$
\Pi_{i,j}=\{x_i=0,x_{3+j}=0\}\subset \bP^4,
$$
and each of them contains four singular points of $X$. 
Let $\Pi$ be one of the planes, $C_2$ a general conic in $\Pi$ that contains  $\Pi\cap\mathrm{Sing}(X)$,
and $\widetilde{C}_2$ its strict transform on $\widetilde{X}$. Then $\widetilde{C}_2\not\subset\widetilde{M}$, so that
\begin{align*}
0\leq\lambda\widetilde{M}\cdot\widetilde{C}_2& =\left(f^*\big(-K_{X}\big)-a\sum_{i=1}^{9}E_i\right)\cdot\widetilde{C}_2 \\
& =4-a\sum_{i=1}^{9}E_i\cdot\widetilde{C}_2=4-4a<0,
\end{align*}
which is absurd. 

Let $P$ be a point in $X$ such that  the log pair $(X,\lambda\mathcal{M})$ is not canonical at $P$.
Then $(X,\lambda\mathcal{M})$ is canonical in a punctured neighborhood of $P$,
and it follows from \cite[Remark~3.6]{VAZ} that the log pair $(X,\frac{3\lambda}{2}\mathcal{M})$ is not log canonical at $P$.
Arguing as in the proof of \cite[Proposition~3.5]{VAZ}, we obtain a contradiction.
\end{proof}

\bibliographystyle{plain}
\bibliography{sing-cube}

\end{document}